\documentclass[11pt, notitlepage]{article}
\usepackage{amsmath, amsthm, amssymb, todonotes, enumerate,color,pigpen, mathrsfs, extarrows}
\usepackage{graphicx}
\usepackage[makeroom]{cancel}
\usepackage[all,cmtip]{xy}
\usepackage[hidelinks]{hyperref}
\xyoption{all}
\numberwithin{equation}{section}

\theoremstyle{plain}
\newtheorem{theorem}[equation]{Theorem}

\newtheorem{thm}{Theorem}
\newtheorem{cor}{Corollary}

\newtheorem{proposition}[equation]{Proposition}
\newtheorem{lemma}[equation]{Lemma}
\newtheorem{corollary}[equation]{Corollary}

\theoremstyle{definition}
\newtheorem{definition}[equation]{Definition}
\newtheorem{example}[equation]{Example}
\newtheorem{remark}[equation]{Remark}

\newcommand{\po}{\ar@{}[dr]|{\text{\pigpenfont R}}}
\newcommand{\pb}{\ar@{}[dr]|{\text{\pigpenfont J}}}

\DeclareMathOperator{\THR}{THR}
\DeclareMathOperator{\TCR}{TCR}

\DeclareMathOperator{\THH}{THH}

\DeclareMathOperator{\KR}{KR}
\DeclareMathOperator{\KH}{KH}
\DeclareMathOperator{\K}{K}

\DeclareMathOperator{\tr}{tr}

\DeclareMathOperator{\id}{id}
\DeclareMathOperator{\sd}{sd}
\DeclareMathOperator*{\hocolim}{hocolim}
\DeclareMathOperator{\Top}{Top}
\DeclareMathOperator{\Hom}{Hom}
\DeclareMathOperator{\Herm}{Herm}
\DeclareMathOperator{\sym}{Sym}
\DeclareMathOperator{\conn}{conn}
\DeclareMathOperator{\Map}{Map}
\DeclareMathOperator{\ev}{ev}
\DeclareMathOperator{\Z}{\mathbb{Z}}
\DeclareMathOperator{\Lq}{L^q}
\DeclareMathOperator{\Lg}{L^g}
\newcommand{\Lqsub}[1]{\operatorname{L}^{\operatorname{q}}_{#1}}
\newcommand{\Lgsub}[1]{\operatorname{L}^{\operatorname{g}}_{#1}}

\begin{document}

\title{$K$-theory of Hermitian Mackey functors and a reformulation of the Novikov Conjecture}
\date{}

\author{
\begin{tabular}{cccccccccc}
Emanuele Dotto&&&&&&&&& Crichton Ogle\\
Dept.\ of Mathematics&&&&&&&&& Dept.\ of Mathematics\\
 University of Bonn&&&&&&&&& The Ohio State University
\end{tabular}
}

\maketitle

\abstract{
We define a genuine $\mathbb{Z}/2$-equivariant real algebraic $K$-theory spectrum $\KR(A)$ for every genuine $\mathbb{Z}/2$-equivariant spectrum $A$ equipped with a compatible multiplicative structure. This construction extends the real $K$-theory of Hesselholt-Madsen for discrete rings and the Hermitian $K$-theory of Burghelea-Fiedorowicz for simplicial rings. We construct a natural trace map of $\mathbb{Z}/2$-spectra $\tr\colon \KR(A)\to \THR(A)$ to the real topological Hochschild homology spectrum, which extends the $K$-theoretic trace of B\"{o}kstedt-Hsiang-Madsen.

The trace provides a splitting of the real $K$-theory of the spherical group-ring. We use this splitting on the geometric fixed points of $\KR$, which we regard as an $L$-theory of genuinely equivariant ring spectra, to reformulate the Novikov conjecture on the homotopy invariance of the higher signatures purely in terms of the module structure of the rational $L$-theory of the ``Burnside group-ring''.
}

\tableofcontents

\newpage

\section*{Introduction} 

In \cite{IbLars} Hesselholt and Madsen construct a $\mathbb{Z}/2$-equivariant spectrum ${\KR}(\mathscr{C})$ from an exact category with duality $\mathscr{C}$, whose underlying spectrum is the $K$-theory spectrum $\K(\mathscr{C})$ of \cite{Quillen} and whose fixed-points spectrum is the connective Hermitian $K$-theory spectrum $\KH(\mathscr{C})$ of \cite{Schlichting}. Specified to the category of free modules over a discrete ring with anti-involution $R$ this construction provides a $\mathbb{Z}/2$-equivariant spectrum ${\KR}(R)$ whose fixed points are the connective Hermitian $K$-theory $\KH(R)$ of \cite{Karoubi} (when $1/2\in R$) and \cite{BF}. The construction of $\KH(R)$ is extended in \cite{BF} from discrete rings to simplicial rings, and the homotopy type of $\KH(R)$  depends both on the homotopy types of $R$ and of the fixed points space $R^{\mathbb{Z}/2}$.

In this paper we propose a further extension of the real $K$-theory functor ${\KR}$ to the category of ring spectra with anti-involution. These are genuine $\mathbb{Z}/2$-equivariant spectra with a suitably compatible multiplication (see \S\ref{secRealSpectra}). The (derived) fixed-points spectrum and the geometric fixed-points spectrum 
\[
\KH(A):={\KR}(A)^{\mathbb{Z}/2}\ \ \ \ \ \ \ \ \ \ \ \mbox{and}\ \ \ \ \ \ \ \ \ \ \  \Lg(A):=\Phi^{\Z/2}\KR(A)
\]
behave, respectively, as a Hermitian $K$-theory and $L$-theory for ring spectra with anti-involution. They depend on the genuine equivariant homotopy type of $A$ and differ from other constructions in the literature (e.g. from the Hermitian $K$-theory of \cite{Spitzweck} and from the quadratic or symmetric $L$-theory of \cite{Lurie}).
The main application of this paper uses the trace map $\tr\colon \KR\to \THR$ to reformulate the Novikov conjecture in terms of the module structure of $\Lg$. 

There is an algebraic case of particular interest that lies between discrete rings and ring spectra: the ring spectra with anti-involution whose underlying $\mathbb{Z}/2$-spectrum is the Eilenberg-MacLane spectrum $HM$ of a $\Z/2$-Mackey functor $M$. In this case the multiplicative structure on $HM$ specifies to a ring structure on the underlying Abelian group $\pi_0HM$ and a multiplicative action of $\pi_0HM$ on $\pi_0(HM)^{\mathbb{Z}/2}$, suitably compatible with the restriction and the transfer maps. We call such an object a Hermitian Mackey functor. 
A class of examples comes from Tambara functors, where the underlying ring acts on the fixed points datum via the multiplicative transfer.
We start our paper by constructing the Hermitian $K$-theory of a Hermitian Mackey functor in \S\ref{secone}, as the group completion of a certain symmetric monoidal category of Hermitian forms $\Herm_M$ over $M$
\[\KH(M):=\Omega B(Bi\Herm_M,\oplus).\]
The key idea for the definition of $\Herm_M$ is that the fixed points datum of the Mackey functor specifies a refinement of the notion of ``symmetry'' used in the classical definition of Hermitian forms over a ring. In \S\ref{secthree} we extend these ideas to ring spectra and we give the full construction of the $\KR$ functor.

The main feature of our real $K$-theory construction is that it comes equipped with a natural trace map to the real topological Hochschild homology spectrum ${\THR}(A)$ of \cite{IbLars} (see also \cite{Thesis}, \cite{Amalie} and \cite{THRmodels}). The following is in \S\ref{sectrace}.

\begin{thm}
Let $A$ be a connective ring spectrum with anti-involution. There is a natural transformation of $\mathbb{Z}/2$-spectra
\[
\tr\colon {\KR}(A)\longrightarrow {\THR}(A)
\]
whose underlying map of spectra is the B\"{o}kstedt-Hsiang-Madsen trace map $\K(A)\to \THH(A)$ from \cite{BHM}.
\end{thm}

In the case of a discrete ring with anti-involution $R$ the trace provides a map of spectra $\KH(R)\to \THR(R)^{\mathbb{Z}/2}$ which is a refinement of earlier constructions of the Chern Character from Hermitian $K$-theory to dihedral homology appearing in \cite{gc}. In \S\ref{secass} we define, for any topological group $\pi$ and ring spectrum with anti-involution $A$, an assembly map
\[
{\KR}(A)\wedge B^{\sigma}\pi_+\longrightarrow {\KR}(A[\pi])
\]
where $A[\pi]:=A\wedge\pi_+$ is the group-ring with the anti-involution induced by the inversion of $\pi$, and $B^{\sigma}\pi$ is a delooping of $\pi$ with respect to the sign-representation. 
 We define a map
\[
Q\colon {\KR}(\mathbb{S}[\pi])\stackrel{\tr}{\longrightarrow} {\THR}(\mathbb{S}[\pi])\simeq \mathbb{S}\wedge B^{di}\pi_+\longrightarrow \mathbb{S}\wedge B^{\sigma}\pi_+
\]
where $B^{di}\pi\to B^{\sigma}\pi$ is the projection from the dihedral nerve of $\pi$, and the equivalence is from \cite{Amalie}. The following is proved in \S\ref{secsplitting}.

\begin{thm}\label{splitKRintro}
The map $Q$ defines a natural retraction in the homotopy category of $\mathbb{Z}/2$-spectra for the restricted assembly map
\[
\mathbb{S}\wedge B^{\sigma}\pi_+\xrightarrow{\eta\wedge\id}{\KR}(\mathbb{S})\wedge B^{\sigma}\pi_+\longrightarrow {\KR}(\mathbb{S}[\pi]),
\]
where $\eta\colon \mathbb{S}\to {\KR}(\mathbb{S})$ is the unit map. Thus the real $K$-theory of the spherical group-ring splits off a copy of the equivariant suspension spectrum $\mathbb{S}\wedge B^{\sigma}\pi_+$. If $\pi$ is discrete, the Hermitian $K$-theory spectrum ${\KH}(\mathbb{S}[\pi]):={\KR}(\mathbb{S}[\pi])^{\Z/2}$ splits off a copy of
\[ (\mathbb{S}\wedge B^{\sigma}\pi_+)^{\Z/2}\simeq \mathbb{S}\wedge ((B\pi\times \mathbb{RP}^\infty)\amalg \coprod_{\{[g]\ |\ g^2=1\}}BZ_\pi\langle g\rangle)_+,\]
where the disjoint union runs through the conjugacy classes of the order two elements of $\pi$, and $Z_\pi\langle g\rangle$ is the centralizer of $g$ in $\pi$.
\end{thm}

Non-equivariantly this is the splitting of \cite{WJR}. It is unclear at the moment if the other summand of ${\KR}(\mathbb{S}[\pi])$ directly relates to a geometric object. It was brought to our attention by Kristian Moi and Thomas Nikolaus that the fixed-points spectrum ${\KH}(\mathbb{S}[\pi])$ might relate to the spectrum $VLA^\bullet$ from \cite{WW3}. 

The geometric application of this splitting that we propose uses the rationalization of ${\KR}(\mathbb{S}[\pi])$ to reformulate the Novikov conjecture. The Novikov conjecture of \cite{Novikov} for a discrete group $\pi$ is equivalent to the injectivity of the assembly map
\[
\Lq(\mathbb{Z})\wedge B\pi_+\longrightarrow \Lq(\mathbb{Z}[\pi])
\]
on rational homotopy groups, where $\Lq$ denotes the quadratic $L$-theory spectrum and $\mathbb{Z}[\pi]$ is the integral group-ring with the anti-involution induced by the inversion in $\pi$ (see e.g. \cite{KL}). Burghelea and Fiedorowicz show in \cite{BF} that there is a rational decomposition 
\[\KH_\ast(R)\otimes\mathbb{Q}\cong (\Lqsub{\ast}(R)\otimes\mathbb{Q})\oplus (\K_\ast(R)\otimes\mathbb{Q})^{\mathbb{Z}/2}\]
for every discrete ring with anti-involution $R$, where $\K_\ast(R)$ are the algebraic $K$-theory groups. In Proposition \ref{Lgeom} we reinterpret this result as the splitting of the isotropy separation sequence of the rational $\KR$ spectrum, thus identifying rationally the connective $L$-theory spectrum with the geometric fixed-points spectrum $\Phi^{\mathbb{Z}/2}{\KR}(R)$ (see  \cite[7.6]{SchlDer} for a similar result). We then define
\[
\Lg(A):=\Phi^{\mathbb{Z}/2}{\KR}(A),
\]
the ``genuine'' $L$-theory spectrum of the ring spectrum with anti-involution $A$.
The trace induces a map on the rationalized geometric fixed-points spectra
\[
\tr\colon \Lq(\mathbb{Z}[\pi])\otimes\mathbb{Q}\simeq  \Lg(\mathbb{Z}[\pi])\otimes\mathbb{Q}\longrightarrow\Phi^{\mathbb{Z}/2}({\THR}(\mathbb{Z}[\pi]))\otimes\mathbb{Q},
\]
that one could try to exploit to detect the injectivity of the $L$-theoretic assembly. The rational geometric fixed points $\Phi^{\mathbb{Z}/2}{\THR}(R)\otimes\mathbb{Q}$ have  however been computed to be contractible in \cite{THRmodels}, as long as $R$ is a \textit{discrete} ring. This is in line with the results of \cite{gc}, where the Chern Character to dihedral homology factors through algebraic $K$-theory via the forgetful map, and therefore vanishes on the $L$-theory summand.

The rational geometric fixed points $\Phi^{\mathbb{Z}/2}{\THR}(A)\otimes\mathbb{Q}$ are generally non-trivial when the input $A$ is not the Eilenberg McLane spectrum of a discrete ring. The starting point of our analysis is to replace the ring of integers with the Burnside Mackey functor, much in the same way one replaces the integers with the sphere spectrum in the proof of the $K$-theoretic Novikov conjecture of \cite{BHMNovikov}. 
We define a Hermitian Mackey functor 
\[\mathbb{A}_{\frac{1}{2}}[\pi]:=\underline{\pi}^{\mathbb{Z}/2}_0(\mathbb{S}\wedge\pi_+)[\tfrac{1}{2}],\]
the ``Burnside group-ring'' of a discrete group $\pi$ (see Definition \ref{groupMackey}) with $2$ inverted. There is a restriction map $d\colon\mathbb{A}_{\frac{1}{2}}[\pi]\to \underline{\mathbb{Z}}_{\frac{1}{2}}[\pi]$ coming from the augmentation of the Burnside ring, where $\underline{\mathbb{Z}}_{\frac{1}{2}}[\pi]$ is the Mackey functor associated to the integral group-ring $\mathbb{Z}[\pi]$ with $2$ inverted.
The following is proved in \S \ref{secmain}.

\begin{thm}\label{mainintro} Let $\pi$ be a discrete group. There is a lift $\overline{\mathcal{A}}_{\mathbb{Z}[\pi]}$ of the $L$-theoretic connective assembly map of the integral group-ring
\[
\xymatrix@C=75pt@R=20pt{
&& \Lgsub{\ast}(\mathbb{A}_{\frac 12}[\pi])\otimes\mathbb{Q}\ar[d]^d
\\
\mathbb{Q}[\beta]\otimes H_\ast (B\pi;\mathbb{Q})\ar[r]_-{\mathcal{A}_{\mathbb{Z}[\pi]}}\ar[urr]^-{\overline{\mathcal{A}}_{\mathbb{Z}[\pi]}}&&\llap{$\Lqsub{\ast\geq 0}(\mathbb{Z}[\pi])\otimes\mathbb{Q}\cong$}   \Lgsub{\ast}(\underline{\mathbb{Z}}_{\frac 12}[\pi])\otimes\mathbb{Q}\rlap{\ .}
}
\]
For every polynomial $\underline{x}\in\mathbb{Q}[\beta]\otimes H_\ast (B\pi;\mathbb{Q})$ with non-zero constant term, $\overline{\mathcal{A}}_{\mathbb{Z}[\pi]}(\underline{x})\neq 0$.
It follows that the Novikov conjecture holds for $\pi$ if and only if the image of $\overline{\mathcal{A}}_{\mathbb{Z}[\pi]}$ intersects the kernel of $d$ trivially.
\end{thm}

We prove this theorem by detecting $\overline{\mathcal{A}}_{\mathbb{Z}[\pi]}(\underline{x})$ using the trace map.
We define a map
\[
T\colon \Lgsub{\ast}(\mathbb{A}_{\frac{1}{2}}[\pi])\otimes\mathbb{Q}\stackrel{\tr}{\longrightarrow}\Phi^{\mathbb{Z}/2}{\THR}_\ast(\mathbb{A}_{\frac{1}{2}}[\pi])\otimes\mathbb{Q}\cong 
H_\ast((B^{di}\pi)^{\mathbb{Z}/2};\mathbb{Q})\stackrel{p}{\longrightarrow}
H_\ast(B\pi;\mathbb{Q}),
\]
where $B^{di}\pi$ is the dihedral nerve of $\pi$, and $p\colon (B^{di}\pi)^{\Z/2}\to B\pi$ is a certain projection map. The image of the constant term $T\overline{\mathcal{A}}_{\mathbb{Z}[\pi]}(1\otimes x_n)$ of a polynomial 
\[\underline{x}=1\otimes x_n+\beta\otimes x_{n-4}+\dots+\beta^{k}\otimes x_{n-4k}\ \ \ \in \ \ (\mathbb{Q}[\beta]\otimes H_\ast(B\pi;\mathbb{Q}))_n\]
of total degree $n$,
with $x_n\neq 0$, is non-zero essentially by Theorem \ref{splitKRintro}. A naturality argument then shows that $T\overline{\mathcal{A}}_{\mathbb{Z}[\pi]}$ vanishes on the positive powers of $\beta$, and therefore that $\overline{\mathcal{A}}_{\mathbb{Z}[\pi]}(\underline{x})$ is not zero.
By the periodicity of $\Lqsub{\ast}(\Z[\pi])$, if $\mathcal{A}_{\mathbb{Z}[\pi]}$ does not annihilate the polynomials with non-zero constant terms it must be injective (see Remark \ref{remconnective}). Thus the Novikov conjecture holds precisely when $d$ does not kill $\overline{\mathcal{A}}_{\mathbb{Z}[\pi]}(\underline{x})$.

We can further reduce the Novikov conjecture to an algebraic property of $\Lgsub{\ast}(\mathbb{A}_{\frac 12}[\pi])$ as an $\Lgsub{0}(\mathbb{A}_{\frac 12})$-module. The rank map $d$ above admits a natural splitting $s_\pi$ which includes 
$\Lqsub{\ast\geq 0}(\Z[\pi])\otimes\mathbb{Q}$ as a summand of $\Lgsub{\ast}(\mathbb{A}_{\frac 12}[\pi])\otimes\mathbb{Q}$. In particular the unit of $\Lqsub{\ast\geq 0}(\Z)\otimes\mathbb{Q}\cong \mathbb{Q}[\beta]$ defines an element $a$ in the ring $\Lgsub{0}(\mathbb{A}_{\frac 12})\otimes\mathbb{Q}$.

\begin{cor}\label{cormainintro}
Every element $\underline{x}\in\mathbb{Q}[\beta]\otimes H_\ast (B\pi;\mathbb{Q})$ satisfies the identity
\[
s_\pi\mathcal{A}_{\mathbb{Z}[\pi]}(\underline{x})=a\cdot \overline{\mathcal{A}}_{\mathbb{Z}[\pi]}(\underline{x})\ \ \ \ \in \ \ \Lgsub{\ast}(\mathbb{A}_{\frac 12}[\pi])\otimes\mathbb{Q},
\]
where $s_\pi$ is injective and $\overline{\mathcal{A}}_{\mathbb{Z}[\pi]}(\underline{x})$ is non-zero when $\underline{x}$ has non-zero constant term.
It follows that the Novikov conjecture holds for $\pi$ if and only if the multiplication map
\[
a\cdot (-)\colon\Lgsub{\ast}(\mathbb{A}_{\frac 12}[\pi])\longrightarrow \Lgsub{\ast}(\mathbb{A}_{\frac 12}[\pi])
\]
is injective on the image of $\overline{\mathcal{A}}_{\mathbb{Z}[\pi]}$.
\end{cor}

The element $a$ is sent to zero by the trace $\tr\colon \Lgsub{\ast}(\mathbb{A}_{\frac 12}[\pi])\to \Phi^{\mathbb{Z}/2}{\THR}_\ast(\mathbb{A}_{\frac 12}[\pi])$, and therefore $a\cdot \overline{\mathcal{A}}_{\mathbb{Z}[\pi]}(\underline{x})$ cannot be detected by the trace map to $\THR$.
In future work we hope to be able to show that
\[\overline{\tr}(a\cdot \overline{\mathcal{A}}_{\mathbb{Z}[\pi]}(\underline{x}))=\overline{\tr}(a)\cdot \mathcal{A}_{\TCR}(\overline{\tr}(\underline{x}))\]
is non-zero in $\Phi^{\Z/2}\TCR_\ast(\mathbb{A}_{\frac 12}[\pi])\otimes \mathbb{Q}$ by direct calculation, where $\overline{\tr}\colon \KR\to \TCR$ is a lift of the trace map to the real topological cyclic homology spectrum, and $\mathcal{A}_{\TCR}$ is the corresponding assembly map.

A brief outline of the paper follows. In \S\ref{secone} we define Hermitian Mackey functors, we construct some examples, and we define their Hermitian $K$-theory.
In \S\ref{secthree} we construct the real $K$-theory of a ring spectrum with anti-involution. We prove that its fixed points recover the connective Hermitian $K$-theory of simplicial rings of \cite{BF} and the Hermitian $K$-theory of Hermitian Mackey functors of \S\ref{secone}, and we prove that its geometric fixed points are rationally equivalent to the connective $L$-theory of discrete rings. Under these equivalences we recover the $L$-theoretic and Hermitian assembly maps from the $\KR$ assembly. In \S\ref{secfour} we recollect some of the basics on real topological Hochschild homology. We then construct the real trace map and the splitting of the real $K$-theory of the spherical group-ring. Finally in \S\ref{sectwo} we relate the trace map to the Novikov conjecture.

\section*{Acknowledgments}
We sincerely thank Ib Madsen and Lars Hesselholt for sharing so much of their current work on real $K$-theory with the first author, and for the guidance offered over several years. We thank Irakli Patchkoria for pointing out a missing condition in Definition \ref{defHermMackey}.
We also thank Markus Land, Wolfgang L\"{u}ck, Kristian Moi, Thomas Nikolaus, Irakli Patchkoria, Oscar Randal-Williams, Marco Schlichting, Stefan Schwede and Christian Wimmer for many valuable conversations.

\subsection*{Notation and conventions}

A space will always mean a compactly generated weak Hausdorff topological space. These form a category, which we denote by $\Top$. We let $\Top^G$ be its category of $G$-objects, where $G$ is a finite group, mostly $G=\Z/2$. An equivalence of $G$-spaces is a continuous $G$-equivariant map which induces a weak equivalence on $H$-fixed-points, for every subgroup $H$ of $G$.

By a spectrum, we will always mean an orthogonal spectrum. A $G$-spectrum is a $G$-object in the category of orthogonal spectra, and an equivalence of $G$-spectra is a stable equivalence with respect to a complete $G$-universe.


\section{Hermitian Mackey functors and their $K$-theory}\label{secone}

\subsection{Hermitian Mackey functors}\label{secMackey}

The standard input of Hermitian $K$-theory is a ring $R$ with an anti-involution $w\colon R^{op}\to R$, or in other words an Abelian group $R$ with a $\mathbb{Z}/2$-action and a ring structure
\[R\otimes_{\mathbb{Z}}R\longrightarrow R\]
which is equivariant with respect to the $\mathbb{Z}/2$-action on the tensor product that swaps the two factors and acts on both variables. In equivariant homotopy theory Abelian groups with $\mathbb{Z}/2$-actions are replaced by the more refined notion of $\mathbb{Z}/2$-Mackey functors. In what follows, we will define a suitable multiplicative structure on a Mackey functor which extends the notion of a ring with anti-involution.

We recall that a $\mathbb{Z}/2$-Mackey functor $L$ consists of two Abelian groups $L(\mathbb{Z}/2)$ and $L(\ast)$, a $\mathbb{Z}/2$-action $w$ on $L(\mathbb{Z}/2)$, and $\mathbb{Z}/2$-equivariant maps (with respect to the trivial action on $L(\ast)$)
\[R\colon L(\ast)\longrightarrow L(\mathbb{Z}/2)\ \ \ \ \ \ \ \ \ \ \ \ \  T\colon L(\mathbb{Z}/2)\longrightarrow L(\ast),\]
called respectively the restriction and the transfer, subject to the relation
\[
RT(a)=a+w(a)
\]
for every $a\in L(\mathbb{Z}/2)$.

\begin{definition}\label{defHermMackey}
A Hermitian Mackey functor is a $\mathbb{Z}/2$-Mackey functor $L$, together with a multiplication on $L(\mathbb{Z}/2)$ that makes it into a ring, and a multiplicative left action of $L(\mathbb{Z}/2)$ on the Abelian group $L(\ast)$ which satisfy the following conditions:
\begin{enumerate}[i)]
\item $w(aa')=w(a')w(a)$ for all $a,a'\in L(\mathbb{Z}/2)$, and $w(1)=1$,
\item $R(a\cdot b)=aR(b)w(a)$ for all $a\in L(\mathbb{Z}/2)$ and $b\in L(\ast)$,
\item $a\cdot T(c)=T(acw(a))$ for all $a,c\in L(\mathbb{Z}/2)$,
\item $(a+a')\cdot b=a\cdot b+a'\cdot b+T(aR(b)w(a'))$ for all $a,a'\in L(\mathbb{Z}/2)$ and $b\in L(\ast)$, and $0\cdot b=0$.
\end{enumerate}
\end{definition}

\begin{example}
Let $R$ be a ring with anti-involution $w\colon R^{op}\to R$. The Mackey functor $\underline{R}$ associated to $R$ has values $\underline{R}(\mathbb{Z}/2)=R$ and $\underline{R}(\ast)=R^{\mathbb{Z}/2}$, the Abelian subgroup of fixed points. The restriction map is the inclusion of fixed points $R^{\mathbb{Z}/2}\to R$, and the transfer is $T(a)=a+w(a)$.
The multiplication on $R$ defines an action of $R$ on $R^{\mathbb{Z}/2}$ by
\[a\cdot b=abw(a)\]
for $a\in R$ and $b\in R^{\mathbb{Z}/2}$. This gives $\underline{R}$ the structure of a Hermitian Mackey functor.
\end{example}

\begin{example}
Let $\mathbb{A}$ be the Burnside $\mathbb{Z}/2$-Mackey functor. The Abelian group $\mathbb{A}(\mathbb{Z}/2)$ is the group completion of the monoid of isomorphism classes of finite sets, and it has the trivial involution. The Abelian group $\mathbb{A}(\ast)$ is the group completion of the monoid of isomorphism classes of finite $\mathbb{Z}/2$-sets.
The restriction forgets the $\mathbb{Z}/2$-action, and the transfer sends a set $A$ to the free $\mathbb{Z}/2$-set $A\times\mathbb{Z}/2$.
The underlying Abelian group $\mathbb{A}(\mathbb{Z}/2)$ has a multiplication induced by the cartesian product, and it acts on $\mathbb{A}(\ast)$ by
\[A\cdot B=(\prod_{\mathbb{Z}/2}A)\times B.\]
Explicitly, $\mathbb{A}(\mathbb{Z}/2)$ is isomorphic to $\mathbb{Z}$ as a ring, $\mathbb{A}(\ast)$ is isomorphic to $\mathbb{Z}\oplus\mathbb{Z}$ with generators the trivial $\mathbb{Z}/2$-set with one element and the free $\mathbb{Z}$/2-set $\mathbb{Z}/2$. The restriction is the identity on the first summand and multiplication by $2$ on the second summand, and the transfer sends the generator of $\mathbb{Z}$ to the generator of the second $\mathbb{Z}$-summand.
The underlying ring $\mathbb{Z}$ then acts on $\mathbb{Z}\oplus\mathbb{Z}$ by
\[a\cdot (b,c)=(ab,\frac{ba(a-1)}{2}+a^2c).\]
\end{example}

The Hermitian structure on the  Burnside Mackey functor is a special case of the following construction. If the multiplication of a ring $R$ is commutative, then an anti-involution on $R$ is simply an action of $\mathbb{Z}/2$ by ring maps. The Mackey-version of a commutative ring is a Tambara functor, and we show that there is indeed a forgetful functor from $\mathbb{Z}/2$-Tambara functor to Hermitian Mackey functors. We recall that a $\mathbb{Z}/2$-Tambara functor is a Mackey functor where both $L(\mathbb{Z}/2)$ and $L(\ast)$ are commutative rings, and with an additional equivariant multiplicative transfer $N\colon L(\mathbb{Z}/2)\to L(\ast)$, called the norm, which satisfies the properties
\begin{enumerate}[i)]
\item $T(a)b=T(aR(b))$ for all $a\in L(\mathbb{Z}/2)$ and $b\in L(\ast)$,
\item $RN(a)=aw(a)$ for all $a\in L(\mathbb{Z}/2)$ ,
\item $N(a+a')=N(a)+N(a')+T(aw(a'))$ for all $a,a'\in L(\mathbb{Z}/2)$, and $N(0)=0$.
\end{enumerate}
see \cite{Tambara} and \cite{Strickland}.

\begin{example}\label{exTambara}
A Tambara functor $L$ has the structure of a Hermitian Mackey functor by defining the  $L(\mathbb{Z}/2)$-action on $L(\ast)$ as
\[a\cdot b=N(a)b,\]
where the right-hand product is the multiplication in $L(\ast)$, and then forgetting the multiplication of $L(\ast)$ and the norm.

Let us verify the axioms of a Hermitian Mackey functor. The first axiom is satisfied because the multiplication is commutative and equivariant. The second axiom is
\[R(a\cdot b)=R(N(a)b)=R(N(a))R(b)=aw(a)R(b)
=aR(b)w(a)\]
and the third is
\[a\cdot T(c)=N(a)T(c)=T(R(N(a))c)=T(aw(a)c)=T(acw(a)).\]
The last axiom is clear from the third condition of a Tambara functor.
\end{example}

We conclude the section by extending to Hermitian Mackey functors two standard constructions of rings with anti-involution: the matrix ring and the group-ring.

If $R$ is a ring with anti-involution and $n$ is a positive integer, the ring $M_n(R)$ of $n\times n$-matrices has a natural anti-involution defined by conjugate transposition $w(A)_{ij}:=w(A_{ji})$. A fixed point in $M_n(R)$ is a matrix whose diagonal entries belong to $R^{\mathbb{Z}/2}$, and where the entries $A_{i>j}$ are determined by the entries $A_{i<j}$, by $A_{i>j}=w(A_{j<i})$. Inspired by this example, we give the following definition.

\begin{definition}\label{MackeyMatrix}
Let $L$ be a Hermitian Mackey functor. The Hermitian Mackey functor $M_n(L)$ of $n\times n$-matrices in $L$ is defined by the Abelian groups
\[M_n(L)(\mathbb{Z}/2)=M_n(L(\mathbb{Z}/2)) \ \ \ \ \ \ \ \ \ \ \ \ M_n(L)(\ast)=(\bigoplus_{1\leq i<j\leq n}L(\mathbb{Z}/2))\oplus(\bigoplus_{1\leq i=j\leq n}L(\ast)).\]
The anti-involution on $M_n(L)(\mathbb{Z}/2)$ is the anti-involution of $L(\mathbb{Z}/2)$ applied entrywise followed by matrix transposition.
The restriction of an element $B$ of $M_n(L)(\ast)$ has entries
\[R(B)_{ij}=\left\{\begin{array}{ll}
B_{ij} & \mbox{if } i<j\\
w(B_{ji}) & \mbox{if } i>j\\
R(B_{ii}) & \mbox{if } i=j\ .\\
\end{array}\right.\] 
The transfer of an $n\times n$-matrix $A$ with coefficients in $L(\mathbb{Z}/2)$ has components
\[T(A)_{ij}=\left\{\begin{array}{ll}
A_{ij}+w(A_{ji}) & \mbox{if } i<j\\
T(A_{ii}) & \mbox{if } i=j\ .\\
\end{array}\right.\]
The multiplication on $M_n(L)(\mathbb{Z}/2)$ is the standard matrix multiplication. The action of $M_n(L)(\mathbb{Z}/2)$ on $M_n(L)(\ast)$ is defined by
\[(A\cdot B)_{ij}=\left\{\begin{array}{ll}
(AR(B)w(A))_{ij}
& \mbox{if } i<j\\
T(\sum\limits_{1\leq k<l\leq n} A_{ik}B_{kl}w(A_{il}))+\sum\limits_{1\leq k\leq n} A_{ik}\cdot B_{kk}& \mbox{if } i=j,\\
\end{array}\right.\]
that is, by the conjugation action on the off-diagonal entries, and through the Hermitian structure on the diagonal entries.
\end{definition}

\begin{lemma}\label{MatMackeywelldef}
The object $M_{n}(L)$ defined above is a Hermitian Mackey functor, and if $R$ is a ring with anti-involution $\underline{M_n(R)}\cong M_n(\underline{R})$.
\end{lemma}

\begin{proof}
It is clearly a well-defined Mackey functor, since
\[RT(A)_{ij}=\left\{\begin{array}{ll}
T(A)_{ij} & \mbox{if } i<j\\
w(T(A)_{ji}) & \mbox{if } i>j\\
R(T(A)_{ii}) & \mbox{if } i=j\\
\end{array}\right.=
\left\{\begin{array}{ll}
A_{ij}+w(A_{ji})& \mbox{if } i<j\\
w(A_{ji}+w(A_{ij}))=w(A_{ji})+A_{ij}& \mbox{if } i>j\\
R(T(A_{ii}))=A_{ii}+w(A_{ii}) & \mbox{if } i=j\\
\end{array}\right.
\]
is equal to $(A+w(A))_{ij}$.
Let us verify that the formula above indeed defines a monoid action. This is immediate for the components $i<j$. For the diagonal components let us first verify that the identity matrix $I$ acts trivially. This is because
\[
(I\cdot B)_{ii}=T(\sum\limits_{1\leq k<l\leq n} I_{ik}B_{kl}w(I_{il}))+\sum\limits_{1\leq k\leq n} I_{ik}\cdot B_{kk}=0+B_{kk}.
\]
In order to show associativity we calculate the diagonal components of $(AC)\cdot B$ for matrices $A,C\in M_n(L(\mathbb{Z}/2))$ and $B\in  M_n(L)(\ast)$.
These are
\begin{align*}
((AC)\cdot B)_{ii}&=
T(\sum\limits_{p<q} (AC)_{ip}B_{pq}w((AC)_{iq}))+\sum\limits_{t} (AC)_{it}\cdot B_{tt}
\\&=
T(\sum\limits_{p<q}\sum\limits_{k,l} A_{ik}C_{kp}B_{pq}w(A_{il}C_{lq}))+\sum\limits_{t}(\sum_u A_{iu}C_{ut})\cdot B_{tt}.
\end{align*}
An easy induction argument on the fourth axiom of a Hermitian functor shows that
\[
(\sum_{1\leq h\leq n}a_h)\cdot b=\sum_{1\leq h\leq n}(a_h\cdot b)+\sum_{1\leq k<l\leq n}T(a_kR(b)w(a_l)),
\]
and the expression above becomes
\begin{align*}
((AC)\cdot B)_{ii}&=
T(\sum\limits_{p<q}\sum\limits_{k,l} A_{ik}C_{kp}B_{pq}w(A_{il}C_{lq}))+\sum\limits_{t}\sum_u (A_{iu}C_{ut}\cdot B_{tt})+\\
&+ \sum\limits_{t}\sum_{k<l} T(A_{ik}C_{kt}R(B_{tt})w(A_{il}C_{lt})).
\end{align*}
On the other hand the diagonal components of $(A\cdot (C\cdot B))_{ii}$ are
\begin{align*}
\ &T(\sum\limits_{k<l}  A_{ik}(C\cdot B)_{kl}w(A_{il}))+\sum\limits_{u} A_{iu}\cdot (C\cdot B)_{uu}
\\&=
T(\sum\limits_{k<l}\sum_{p,q}  A_{ik} C_{kp}R(B)_{pq}w(A_{il}C_{lq}))+\sum\limits_{u} A_{iu}\cdot (
T(\sum_{p<q}C_{up}B_{pq}w(C_{uq}))+\sum_{t}C_{ut}\cdot B_{tt})
\\&=
T(\sum\limits_{k<l}\sum_{p,q}  A_{ik} C_{kp}R(B)_{pq}w(A_{il}C_{lq}))+
\\
&+
T(\sum\limits_{u}\sum_{p<q}A_{iu}C_{up}B_{pq}w(A_{iu}C_{uq}))+\sum\limits_{u}\sum_{t} A_{iu}\cdot C_{ut}\cdot B_{tt}
\\&=
T(\sum\limits_{k<l}\sum_{p<q}  A_{ik} C_{kp}(B_{pq}+w(B_{pq}))w(A_{il}C_{lq}))+T(\sum\limits_{k<l}\sum_{t}  A_{ik} C_{kt}R(B_{tt})w(A_{il}C_{lt}))
\\
&+
T(\sum\limits_{u}\sum_{p<q}A_{iu}C_{up}B_{pq}w(A_{iu}C_{uq}))+\sum\limits_{u}\sum_{t} A_{iu}\cdot C_{ut}\cdot B_{tt}.
\end{align*}
We see that the second and the fourth term of this expression cancel respectively with the third and second term of the expression of $(AC)\cdot B$. Finally, by using that the transfer is equivariant we rewrite the sum of the first and third terms as
\begin{align*}
&T(\sum\limits_{k<l}\sum_{p<q}  A_{ik} C_{kp}(B_{pq}+w(B_{pq}))w(A_{il}C_{lq}))
+
T(\sum\limits_{u}\sum_{p<q}A_{iu}C_{up}B_{pq}w(A_{iu}C_{uq}))
\\
&=
T(\sum\limits_{k<l}\sum_{p<q}  A_{ik} C_{kp}B_{pq}w(A_{il}C_{lq}))
+
T(\sum\limits_{k<l}\sum_{p<q}A_{il}C_{lq}B_{pq} w(A_{ik} C_{kp}))
+
\\
&+
T(\sum\limits_{u}\sum_{p<q}A_{iu}C_{up}B_{pq}w(A_{iu}C_{uq}))
\\
&=
T(\sum\limits_{k<l}\sum_{p<q}  A_{ik} C_{kp}B_{pq}w(A_{il}C_{lq}))
+
T(\sum\limits_{k>l}\sum_{p<q}A_{ik}C_{kq}B_{pq} w(A_{il} C_{lp})
+
\\
&+T(\sum\limits_{u}\sum_{p<q}A_{iu}C_{up}B_{pq}w(A_{iu}C_{uq}))
=
T(\sum\limits_{k,l}\sum_{p<q}  A_{ik} C_{kp}B_{pq}w(A_{il}C_{lq})).
\end{align*}

Let us now verify the other axioms of a Hermitian Mackey functor. The compatibility between the action and the restriction holds since
\[R(A\cdot B)_{ij}=\left\{\begin{array}{ll}
(A\cdot B)_{ij} & \mbox{if } i<j\\
w((A\cdot B)_{ji})=w((AR(B)w(A))_{ji}) & \mbox{if } i>j\\
R((A\cdot B)_{ii})=\sum\limits_{1\leq k<l\leq n} RT(A_{ik}B_{kl}w(A_{il}))+\sum\limits_{1\leq k\leq n} R(A_{ik}\cdot B_{kk}) & \mbox{if } i=j\\
\end{array}\right.
\]
\[=\left\{\begin{array}{ll}
(AR(B)w(A))_{ij} & \mbox{if } i\neq j\\
\sum\limits_{1\leq k<l\leq n} (A_{ik}B_{kl}w(A_{il})+A_{il}w(B_{kl})w(A_{ik}))+\sum\limits_{1\leq k\leq n} A_{ik}R(B_{kk})w(A_{ik}) & \mbox{if } i=j\\
\end{array}\right.
\]
which is equal to $(AR(B)w(A))_{ij}$.
The compatibility between the action and the transfer is 
\[(A\cdot T(C))_{ij}=
\left\{\begin{array}{ll}
(AR(T(C))w(A))_{ij}=(A(C+w(C))w(A))_{ij}
& \mbox{if } i<j\\
T(\sum\limits_{1\leq k<l\leq n} A_{ik}T(C)_{kl}w(A_{il}))+\sum\limits_{1\leq k\leq n} A_{ik}\cdot T(C)_{kk}& \mbox{if } i=j\\
\end{array}\right.=
\]
\[=
\left\{\begin{array}{ll}
(ACw(A)+w(ACw(A)))_{ij}=T(ACw(A))_{ij}
& \mbox{if } i<j\\
T(\sum\limits_{1\leq k<l\leq n} A_{ik}(C_{kl}+w(C_{lk}))w(A_{il}))+\sum\limits_{1\leq k\leq n} A_{ik}\cdot T(C_{kk})& \mbox{if } i=j\\
\end{array}\right.
\]
\[=
\left\{\begin{array}{ll}
T(ACw(A))_{ij}
& \mbox{if } i<j\\
T(\sum\limits_{1\leq k<l\leq n} A_{ik}C_{kl}w(A_{il})+A_{ik}w(C_{lk})w(A_{il}))+\sum\limits_{1\leq k\leq n}T(A_{ik}C_{kk}w(A_{ik}))& \mbox{if } i=j\\
\end{array}\right.,
\]
and this is by definition $T(ACw(A))_{ij}$.
The distributivity of the action over the sum in $M_n(L)(\ast)$ is easy to verify for the components $i<j$. In the diagonal components we have that 
\begin{align*}
((A+A')\cdot B)_{ii}&=T(\sum\limits_{1\leq k<l\leq n} (A+A')_{ik}B_{kl}w((A+A')_{il}))+\sum\limits_{1\leq k\leq n} (A+A')_{ik}\cdot B_{kk}=
\\
&=
T(\sum\limits_{1\leq k<l\leq n} A_{ik}B_{kl}w(A_{il}))+T(\sum\limits_{1\leq k<l\leq n} A'_{ik}B_{kl}w(A'_{il}))+
\\&
+T(\sum\limits_{1\leq k<l\leq n} A_{ik}B_{kl}w(A'_{il}))+T(\sum\limits_{1\leq k<l\leq n} A'_{ik}B_{kl}w(A_{il}))+
\\&+
\sum\limits_{1\leq k\leq n} A_{ik}\cdot B_{kk}+\sum\limits_{1\leq k\leq n} A'_{ik}\cdot B_{kk}+\sum\limits_{1\leq k\leq n} T(A_{ik} R(B_{kk})w(A'_{ik}))=
\\&
=(A\cdot B)_{ii}+(A'\cdot B)_{ii}+T(\sum\limits_{1\leq k<l\leq n} A_{ik}B_{kl}w(A'_{il}))+\\
&+T(\sum\limits_{1\leq k<l\leq n} A'_{ik}B_{kl}w(A_{il}))+\sum\limits_{1\leq k\leq n} T(A_{ik} R(B_{kk})w(A'_{ik})).
\end{align*}
By using that the transfer is equivariant and by reindexing the sum we rewrite the fourth summand as
\[
T(\sum\limits_{1\leq k<l\leq n} A'_{ik}B_{kl}w(A_{il}))=T(\sum\limits_{1\leq k>l\leq n} A'_{il}B_{lk}w(A_{ik}))=T(\sum\limits_{1\leq k>l\leq n} A_{ik}w(B_{lk})w(A'_{il})).
\]
Thus the expression above is equal to
\[(A\cdot B)_{ii}+(A'\cdot B)_{ii}+T(\sum\limits_{1\leq k,l\leq n} A_{ik}R(B)_{kl}w(A'_{il}))=(A\cdot B)_{ii}+(A'\cdot B)_{ii}+T(AR(B)w(A'))_{ii}.\]
Finally, by inspection, we see that $\underline{M_n(R)}\cong M_n(\underline{R})$.
\end{proof}

Let $\pi$ be a discrete group with an anti-involution $\tau\colon \pi^{op}\to \pi$ (for example inversion). If $R$ is a ring with anti-involution, the group-ring $R[\pi]=\oplus_\pi R$ inherits an anti-involution
\[w(\sum_{g\in \pi}a_gg)=\sum_{g\in \pi}w(a_{\tau g})g.\]
A choice of section $s$ of the quotient map $\pi\to \pi/(\mathbb{Z}/2)$ determines an isomorphism
\[(R[\pi])^{\mathbb{Z}/2}\cong R^{\mathbb{Z}/2}[\pi^{\mathbb{Z}/2}]\oplus R[\pi^{free}/(\mathbb{Z}/2)]\]
where $\pi^{free}=\pi-\pi^{\mathbb{Z}/2}$ is the subset of $\pi$ on which $\mathbb{Z}/2$ acts freely. It is defined on the $R^{\mathbb{Z}/2}[\pi^{\mathbb{Z}/2}]$ summand by the inclusion, and on the second summand by sending $cx$ to $cs(x)+w(c)\tau(s(x))$.

\begin{definition}[Group-Mackey functor]\label{groupMackey}
Let $L$ be a Hermitian Mackey functor and $\pi$ a discrete group with anti-involution $\tau\colon \pi^{op}\to \pi$. The associated group-Mackey functor is the Hermitian Mackey functor $L[\pi]$ defined by the Abelian groups
\[L[\pi](\mathbb{Z}/2)=L(\mathbb{Z}/2)[\pi] \ \ \ \ \ \ \ \ \ \ \ \ \ \  L[\pi](\ast)=L(\ast)[\pi^{\mathbb{Z}/2}]\oplus L(\mathbb{Z}/2)[\pi^{free}/\mathbb{Z}/2].\]
The anti-involution on $L(\mathbb{Z}/2)[\pi]$ is the standard anti-involution on the group-ring. The restriction is induced by the restriction map $R\colon L(\ast)\to L(\mathbb{Z}/2)$ and by the inclusion of the fixed points of $\pi$ on the first summand, and by the map
\[R(cx)=cs(x)+w(c)\tau(s(x))\]
on the second summand. The transfer is defined by
\[T(ag)=\left\{\begin{array}{ll}
T(a)g
& \mbox{if } g\in \pi^{\mathbb{Z}/2}\\
a[g]
& \mbox{if } g\in \pi^{free}\mbox{ and }g=s[g]\\
w(a)[g]
& \mbox{if } g\in \pi^{free}\mbox{ and }g=\tau s[g].
\end{array}\right.\]
The multiplication on $L[\pi](\mathbb{Z}/2)$ is that of the group-ring $L(\mathbb{Z}/2)[\pi]$. The action of a generator $ag\in L[\pi](\mathbb{Z}/2)$ on $L[\pi](\ast)$ is extended linearly from
\[ag\cdot bh=(a\cdot b)(gh\tau(g))\]
for $bh\in L(\ast)[\pi^{\mathbb{Z}/2}]$, and
\[ag\cdot cx=\left\{\begin{array}{ll}
acw(a)[gs(x)\tau(g)]&\mbox{if }gs(x)\tau(g)=s[gs(x)\tau(g)]\\
aw(c)w(a)[gs(x)\tau(g)]&\mbox{if }gs(x)\tau(g)=\tau s[gs(x)\tau(g)]
\end{array}\right.\]
for $cx\in L(\mathbb{Z}/2)[\pi^{free}/(\mathbb{Z}/2)]$. It is then extended to the whole group-ring $L[\pi](\mathbb{Z}/2)$ by enforcing condition iv), namely by defining
\[
(\sum_{g\in \pi}a_gg)\cdot \xi=\sum_{g\in \pi}(a_gg\cdot\xi)+\sum_{g<g'}T(a_ggR(\xi)w(a_{g'})\tau(g'))
\]
for some choice of total order on the finite subset of $\pi$ on which $a_g\neq 0$.

When $L=\mathbb{A}$ is the Burnside Mackey functor, we will call $\mathbb{A}[\pi]$ the Burnside group-ring.
\end{definition}

\begin{remark}\label{functorialitygroupring}
The definition of $L[\pi]$ depends on the choice of section up to isomorphism, and it is therefore not strictly functorial in $\pi$. However, it is independent of such choice for the Mackey functors that have trivial action $w$. This will be the case for example for the Burnside Mackey functor $\mathbb{A}$. This construction is always functorial in $L$.
\end{remark}

\begin{lemma}
The functor $L[\pi]$ is a well-defined Hermitian Mackey functor, and if $R$ is a ring with anti-involution $\underline{R[\pi]}\cong \underline{R}[\pi]$.
\end{lemma}

\begin{proof}
We see that $L[\pi]$ is a Mackey functor, since
\[
RT(ag)=\left\{\begin{array}{ll}
R(T(a)g)=(RT(a))g=(a+w(a))g
& \mbox{if } g\in \pi^{\mathbb{Z}/2}\\
R(a[g])=ag+w(a)\tau(g)
& \mbox{if } g\in \pi^{free}\mbox{ and }g=s[g]\\
R(w(a))[g]=w(a)\tau(g)+ag
& \mbox{if } g\in \pi^{free}\mbox{ and }g=\tau s[g]
\end{array}\right.
\]
is equal to $ag+w(a)\tau(g)$.
A calculation analogous to the one of Lemma \ref{MatMackeywelldef} shows that the action of $L[\pi](\mathbb{Z}/2)$ on $L[\pi](\ast)$ is indeed associative (this also follows from \ref{gpring} if $L$ is a Tambara functor, since it can be realized as the $\underline{\pi}_0$ of a commutative $\mathbb{Z}/2$-equivariant ring spectrum).
The compatibility between the action and the restriction is
\[R(ag\cdot bh)=R((a\cdot b)(gh\tau(g)))=R(a\cdot b)(gh\tau(g))=aR(b)w(a)(gh\tau(g))=(ag)R(bh)w(a)\tau(g)\]
for the action on the first summand. On the second summand, this is
\[R(ag\cdot cx)=
\left\{\begin{array}{ll}
R(acw(a)[gs(x)\tau(g)])&\mbox{if }gs(x)\tau(g)=s[gs(x)\tau(g)]\\
R(aw(c)w(a)[gs(x)\tau(g)])&\mbox{if }gs(x)\tau(g)=\tau s[gs(x)\tau(g)]
\end{array}\right.
\]
\[
=\left\{\begin{array}{ll}
acw(a)gs(x)\tau(g)+w(acw(a))\tau(gs(x)\tau(g))&\mbox{if }gs(x)\tau(g)=s[gs(x)\tau(g)]\\
aw(c)w(a)\tau(gs(x)\tau(g))+w(aw(c)w(a))gs(x)\tau(g)
&\mbox{if }gs(x)\tau(g)=\tau s[gs(x)\tau(g)]
\end{array}\right.
\]
which is equal to $(ag)R(cx)(w(a)\tau(g))$. Let us verify the compatibility between the action and the transfer. We have that $ag\cdot T(bh)$ is equal to
\[\left\{\begin{array}{ll}
(ag)\cdot(T(b)h)
=T(abw(a))gh\tau(g)
&\!\! \mbox{if } h\in \pi^{\mathbb{Z}/2}
\\
(ag)\cdot(b[h])=abw(a)[gh\tau(g)]
&\!\! \mbox{if } h\in \pi^{free}\!,h\!=\!s[h], gh\tau(g)\!=\!s[gh\tau(g)]
\\
(ag)\cdot(b[h])=aw(b)w(a)[gh\tau(g)]
&\!\! \mbox{if } h\in \pi^{free}\!,h\!=\!s[h], gh\tau(g)\!=\!\tau s[gh\tau(g)]
\\
(ag)\cdot (w(b)[h])=aw(b)w(a)[g\tau(h)\tau(g)]& \!\!\mbox{if } h\in \pi^{free}\!,h\!=\!\tau s[h], g\tau(h)\tau(g)\!=\!s[g\tau(h)\tau(g)]
\\
(ag)\cdot (w(b)[h])=aw^2(b)w(a)[g\tau(h)\tau(g)]
&\!\! \mbox{if } h\in \pi^{free}\!,h\!=\!\tau s[h], g\tau(h)\tau(g)\!=\!\tau s[g\tau(h)\tau(g)]
\end{array}\right.
\]
\[
=\left\{\begin{array}{ll}
T(abw(a))gh\tau(g)
& \mbox{if } gh\tau(g)\in \pi^{\mathbb{Z}/2}\\
abw(a)[gh\tau(g)]
& \mbox{if } h\in \pi^{free}\mbox{ and }gh\tau(g)=s[gh\tau(g)]\\
aw(b)w(a)[gh\tau(g)]
& \mbox{if } h\in \pi^{free}\mbox{ and }gh\tau(g)=\tau s[gh\tau(g)]
\end{array}\right.=T(abw(a)gh\tau(g)).
\]
The last axiom is satisfied by construction.
By inspection we see that  $\underline{R[\pi]}\cong \underline{R}[\pi]$.
\end{proof}

\subsection{The Hermitian K-theory of a Hermitian Mackey functor}\label{secHermMackey}

Let $L$ be a Hermitian Mackey functor. We use the Hermitian Mackey functors of matrices constructed in Definition \ref{MackeyMatrix} to define a symmetric monoidal category of Hermitian forms, whose group completion will be the Hermitian $K$-theory of $L$.

\begin{definition}\label{defHermform}
Let $L$ be a Hermitian Mackey functor. An $n$-dimensional Hermitian form on $L$ is an element of $M_n(L)(\ast)$ which restricts to an element of $GL_n(L(\mathbb{Z}/2))$ under the restriction map
\[
R\colon M_n(L)(\ast)\longrightarrow M_n(L)(\mathbb{Z}/2)=M_n(L(\mathbb{Z}/2)).
\]
We write $GL_n(L)(\ast)$ for the set of $n$-dimensional Hermitian forms.
A morphism $B\to B'$ of Hermitian forms is a matrix $\lambda$ in $M_n(L(\mathbb{Z}/2))$ which satisfies
\[B=w(\lambda)\cdot B',\]
where the operation is the action of $M_n(L)(\mathbb{Z}/2)$ on $M_n(L)(\ast)$. The multiplication of matrices defines a category of Hermitian forms, which we denote by $\Herm_L$.
\end{definition}

\begin{remark}
Let $R$ be a ring with anti-involution. An $n$-dimensional Hermitian form on the associated Hermitian Mackey functor $\underline{R}$ is an invertible matrix with entries in $R$ which is fixed by the involution $w(A)_{ij}=w(A_{ji})$. This is the same as the datum of an anti-symmetric non-degenerate bilinear pairing $R^{\oplus n}\otimes R^{\oplus n}\to R$, that is a Hermitian form on $R^{\oplus n}$. Since the action of $M_n(\underline{R})(\mathbb{Z}/2)$ on $M_n(\underline{R})(\ast)$ is by conjugation, a morphism of Hermitian forms in the sense of Definition \ref{defHermform} corresponds to the classical notion of isometry.
\end{remark}

The block-sum of matrices on objects and morphisms defines the structure of a permutative category on $\Herm_L$. The symmetry isomorphism from $B\oplus B'$ to $B'\oplus B$, where $B$ is $n$-dimensional and $B'$ is $m$-dimensional, is given by the standard permutation matrix of $GL_{n+m}(L(\mathbb{Z}/2))$ with blocks
\[
\tau_{n,m}:=\left(
\begin{array}{ll}
O_{mn}& I_n
\\
I_m&O_{nm}
\end{array}
\right).
\]
Here $O_{nm}$ is the null $n\times m$-matrix, where the diagonal zeros are those of $L(\ast)$ and the off-diagonal ones are in $L(\mathbb{Z}/2)$. The matrix $I_n$ is the $n\times n$-identity matrix of $L(\mathbb{Z}/2)$. The classifying space $Bi\Herm_L$ of the category of invertible morphisms is therefore an $E_\infty$-monoid.

\begin{definition}\label{defKHMackey}
Let $L$ be a Hermitian Mackey functor. The Hermitian $K$-theory space of $L$ is the group completion 
\[\KH(L):=\Omega B (Bi\Herm_L,\oplus).\]
Segal's $\Gamma$-space construction for the symmetric monoidal category $(i\Herm_L,\oplus)$ provides a spectrum whose infinite loop space is equivalent to $\KH(L)$, that we also denote by $\KH(L)$.
\end{definition}

\begin{remark}\label{KHBar}
If $\lambda\colon B\to B'$ is a morphism of Hermitian forms, the form $B$ is determined by $B'$ and the matrix $\lambda$. Thus a string of composable morphisms
\[
B_0\stackrel{\lambda_0}{\longrightarrow}B_1\stackrel{\lambda_1}{\longrightarrow}\dots \stackrel{\lambda_n}{\longrightarrow}B_n
\]
is determined by the sequence of matrices $\lambda_1,\dots,\lambda_n$, and by the form $B_n$. This gives an isomorphism
\[
Bi\Herm_L\cong  \coprod_{n\geq 0} B\big(\ast, GL_n(L(\mathbb{Z}/2)),GL_n(L)(\ast)\big)
\]
where $B(\ast, GL_n(L(\mathbb{Z}/2)),GL_n(L)(\ast))$ is the Bar construction of the right action of the group $GL_n(L(\mathbb{Z}/2))$ on the set of $n$-dimensional Hermitian forms $w(\lambda)\cdot B$, given by the Hermitian structure of the Mackey functor $M_n(L)$. The action indeed restricts to an action on $GL_n(L)(\ast)$ because if $\lambda$ is in $GL_n(L(\mathbb{Z}/2))$ and the restriction of $B\in M_n(L)(\ast)$ is invertible, then
\[R(w(\lambda)\cdot B)=w(\lambda) R(B)\lambda\]
is also invertible. For rings with anti-involution this is \cite[Rem.1.3]{BF}.
\end{remark}

\begin{remark}
Since the notion of Hermitian forms on Hermitian Mackey functors extends that of Hermitian forms on rings with anti-involution, it follows that our definition of Hermitian $K$-theory extends the Hermitian K-theory construction of \cite{BF}, of the category of free modules over a discrete ring with anti-involution.
\end{remark}

We now make our Hermitian $K$-theory construction functorial.

\begin{definition}
A morphism of Hermitian Mackey functors is a map of Mackey functors $f\colon L\to N$ such that $f_{\mathbb{Z}/2}\colon L(\mathbb{Z}/2)\to N(\mathbb{Z}/2)$ is a ring map, and such that $f_\ast\colon L(\ast)\to N(\ast)$ is $L(\mathbb{Z}/2)$-equivariant, where $N(\ast)$ is a $L(\mathbb{Z}/2)$-set via $f_{\mathbb{Z}/2}$.
\end{definition}

Clearly a map of Hermitian Mackey functors $f\colon L\to N$ induces a symmetric monoidal functor $f_\ast\colon \Herm_L\to \Herm_N$, by applying $f_{\mathbb{Z}/2}$ and $f_{\ast}$ to the matrices entrywise. Thus it induces a continuous map $f_\ast\colon \KH(L)\to \KH(N)$, and a map of spectra $f_\ast\colon {\KH}(L)\to {\KH}(N)$. We will be mostly interested in the following example.

\begin{example}\label{dimonKH}
Let $\mathbb{Z}$ be the ring of integers with the trivial anti-involution, and $\underline{\mathbb{Z}}$ the corresponding Hermitian Mackey functor. There is a morphism of Hermitian Mackey functors
\[d\colon \mathbb{A}\longrightarrow \underline{\mathbb{Z}}\]
from the Burnside Mackey functor. The map $d_{\mathbb{Z}/2}$ is the identity of $\mathbb{Z}$, and the map
\[
d_\ast\colon \mathbb{Z}\oplus \mathbb{Z}\longrightarrow \mathbb{Z}
\]
is the identity on the first summand and multiplication by $2$ on the second. In terms of finite $\mathbb{Z}/2$-sets, it sends a set to its cardinality. This is in fact a morphism of Tambara functors for the standard multiplicative structures on $\mathbb{A}$ and $\underline{\mathbb{Z}}$, and since the Hermitian structures are defined via the multiplicative norms it follows that $d$ is a map of Hermitian Mackey functors.

If moreover $\pi$ is a discrete group with anti-involution, the map $d$ induces a morphism on the associated group-Mackey functors $d\colon \mathbb{A}[\pi]\to \underline{\mathbb{Z}}[\pi]$. The underlying map $d_{\mathbb{Z}/2}$ is again the identity on $\mathbb{Z}[\pi]$, and the map
\[
d_{\ast}\colon \mathbb{A}[\pi](\ast)=(\mathbb{Z}\oplus\mathbb{Z})[\pi]\oplus \mathbb{Z}[\pi^{free}/(\mathbb{Z}/2)]\longrightarrow (\mathbb{Z}[\pi])^{\mathbb{Z}/2}=\mathbb{Z}[\pi]\oplus \mathbb{Z}[\pi^{free}/(\mathbb{Z}/2)]
\]
is $d[\pi]$ on the first summand and the identity on the second summand.
This map therefore induces a map on Hermitian $K$-theory spectra
\[
d\colon {\KH}(\mathbb{A}[\pi])\longrightarrow {\KH}(\underline{\mathbb{Z}}[\pi])={\KH}(\mathbb{Z}[\pi]).
\]
\end{example}

\subsection{Multiplicative structures}

We saw in Example \ref{exTambara} that $\Z/2$-Tambara functors provide a supply of Hermitian Mackey functors. In this section we show that the Hermitian $K$-theory spectrum of a Tambara functor is in fact a ring spectrum. We will generalize this construction when the input is a commutative $\Z/2$-equivariant ring spectrum in \S\ref{secpairings}, but the construction for Tambara functors will give us slightly more functoriality, which will come in handy in \S\ref{secmain}.

Let $L$ be a Tambara functor. We define a pairing of categories
\[
\otimes\colon \Herm_{L}\times \Herm_L\longrightarrow \Herm_L
\]
by means of an extension of the standard tensor product of matrices. On objects, we define the tensor product of two Hermitian forms $B$ and $B'$ on $L$ of respective dimensions $n$ and $m$ to be the $nm$-dimensional form $B\otimes B'$ with diagonal components
\[(B\otimes B')_{ii}=B_{kk}\cdot B'_{uu} \ \ \ \ \ \ \ \mbox{where} \ \ \ \ k=\lfloor\frac{i-1}{n}\rfloor+1\ \  ,\  \  u=i-n\lfloor\frac{i-1}{n}\rfloor\]
where the multiplication denotes the multiplication in the commutative ring $L(\ast)$. The off-diagonal term $1\leq i<j\leq nm$ of $B\otimes B'$ is defined by
\[(B\otimes B')_{ij}=
R(B)_{kl}\cdot R(B')_{uv}
\ \ \ \ 
\mbox{where} \ \ \ \begin{array}{ll}
k=\lfloor\frac{i-1}{n}\rfloor+1,&  l=\lfloor\frac{j-1}{n}\rfloor+1\\
u=i-n\lfloor\frac{i-1}{n}\rfloor,&  v=j-n\lfloor\frac{j-1}{n}\rfloor
\end{array}
\]
and $R\colon M_{n}(L(\Z/2))\to M_{n}(L)(\ast)$ is the restriction of the Mackey functor of matrices of Definition \ref{MackeyMatrix}.
This is the standard formula of the Kronecker product of matrices
, where the diagonal elements are lifted to the fixed points ring $L(\ast)$.

\begin{example}
In the case $m=n=2$ the product above is given by the matrix
\[
\left(
\begin{array}{ll}
B_{11}&B_{12}\\
&B_{22}
\end{array}
\right)\otimes 
\left(
\begin{array}{ll}
B'_{11}&B'_{12}\\
&B'_{22}
\end{array}
\right)
=
\left(
\begin{array}{llll}
B_{11}B'_{11}&R(B_{11})B'_{12}
&
B_{12}R(B'_{11})&B_{12}B'_{12}
\\
&B_{11}B'_{22}
&
B_{12}w(B'_{12})&B_{12}R(B'_{22})
\\
&
&
B_{22}B'_{11}&R(B_{22})B'_{12}
\\
&&&B_{22}B'_{22}
\end{array}
\right).
\]
\end{example}
Since the restriction map $R\colon L(\ast)\to L(\Z/2)$ is a ring map this operation lifts the standard Kronecker product of matrices, in the sense that
\[
R(B\otimes B')=R(B)\otimes R(B')
\]
as $nm\times nm$-matrices with coefficients in $L(\mathbb{Z}/2)$. We define the pairing $\otimes$ on morphisms by the standard Kronecker product of matrices.

\begin{lemma}
The pairing $\otimes\colon \Herm_{L}\times \Herm_L\to \Herm_L$ is a well-defined functor.
\end{lemma}

\begin{proof}
A tedious but straightforward verification shows that for every pair of matrices $A\in M_{n}(L(\Z/2))$ and  $A'\in M_{m}(L(\Z/2))$, and forms $B\in M_{n}(L)(\ast)$ and  $B'\in M_{m}(L)(\ast)$ we have that
\[
(A\cdot B)\otimes (A'\cdot B')=(A\otimes A')\cdot (B\otimes B')
\]
where the dot is the action of the Mackey structure of $M_{n}(L)$. This uses the identities $T(a)b=T(aR(b))$ and $RN(a)=aw(a)$ of the Tambara structure.

Thus if $\lambda\colon B\to C$ and $\lambda'\colon B'\to C'$ are morphisms of Hermitian forms, we have that
\[
w(\lambda\otimes \lambda')\cdot (C\otimes C')=(w(\lambda)\otimes w(\lambda'))\cdot (C\otimes C')=(w(\lambda)\cdot C)\otimes (w(\lambda')\cdot  C')=B\otimes B'.
\]
Thus $\lambda\otimes \lambda'\colon B\otimes B'\to C\otimes C'$ is a well-defined morphism. The composition of morphisms happens in the matrix rings $M_{n}(L(\Z/2))$, and therefore it is respected by $\otimes$. Similarly, $\otimes$ preserves the identity morphisms.
\end{proof}

It is moreover immediate to verify that the standard compatibility conditions between $\otimes$ and the direct sum are satisfied for forms:
\begin{enumerate}[i)]
\item $(B\oplus B')\otimes B''=(B\otimes B'')\oplus (B'\otimes B'')$,
\item $B\otimes (B'\oplus B'')=\sigma (B\otimes B')\oplus(B\otimes B'')\sigma^{-1}$, where $\sigma$ is a permutation matrix,
\item $0\otimes B=0$ and $B\otimes 0=0$.
\item $1\otimes B=B\otimes 1=B$, where $1$ is the $1$-form with entry the unit of the ring $L(\ast)$.
\end{enumerate}
By property $ii)$ the permutation $\sigma$ defines an isomorphism of forms
\[B\otimes (B'\oplus B'')\cong (B\otimes B')\oplus(B\otimes B''),\]
and one can easily verify that this isomorphism satisfies the higher coherences required to give the following.

\begin{proposition}\label{multiMackey}
The pairing $\Herm_{L}\times \Herm_L\to \Herm_L$
is a pairing of permutative categories, thus inducing a morphism of spectra
\[
\otimes\colon {\KH}(L)\wedge {\KH}(L)\longrightarrow {\KH}(L)
\]
exhibiting ${\KH}(L)$ as a ring spectrum.\qed
\end{proposition}

The morphism $f\colon \Herm_L\to \Herm_{N}$ induced by a morphism of Tambara functors $f\colon L\to N$ clearly commutes with the monoidal structure $\otimes$, thus inducing a morphism of ring spectra $f\colon {\KH}(L)\to {\KH}(N)$.

\begin{remark}\label{remsmult}
Let $L$ and $N$ be Tambara functors, and suppose that $f\colon L\to N$ is a morphism of Hermitian Mackey functors such that $f_\ast\colon L(\ast)\to N(\ast)$ is multiplicative, but not necessarily unital. Then the induced functor $f\colon \Herm_L\to \Herm_{N}$ preserves the tensor product, but not its unit, and the map $f\colon {\KH}(L)\to {\KH}(N)$ is a morphism of non-unital ring spectra. The example we will be interested in is the morphism $\frac{T}{2}\colon \underline{\mathbb{Z}}[\frac 12]\to\mathbb{A}[\frac 12]$, defined by the identity map on underlying rings, and by half the transfer $(0,\frac 12)\colon \mathbb{Z}[\frac 12]\to \mathbb{Z}[\frac 12]\oplus \mathbb{Z}[\frac 12]$ on fixed points.
\end{remark}


\section{Real $K$-theory}\label{secthree}

The aim of this section is to construct the free real $K$-theory $\mathbb{Z}/2$-spectrum of a ring spectrum with anti-involution and its assembly map, and relate these objects to the classical constructions in the case of discrete rings.

\subsection{Real semi-simplicial spaces and the real and dihedral Bar constructions}\label{secBar}

In this section we investigate the Bar construction and the cyclic Bar construction associated to a monoid with an anti-involution. This is essentially a recollection of materials from \cite{LodayDihedral}, \cite[\S 1]{BF} and \cite{IbLars},
but we need a context without degeneracies which requires particular care.

We let $\Delta_+$ be the subcategory of $\Delta$ of all objects and injective morphisms. The category $\Delta$ has an involution that fixes the objects and that sends a morphism $\alpha\colon [n]\to [k]$ to $\overline{\alpha}(i)=k-\alpha(n-i)$. This involution restricts to $\Delta_+$. We recall from \cite{IbLars} that a real simplicial space is a simplicial space $X$ together with levelwise involutions $w\colon X_n\to X_n$ which satisfy $w\circ\alpha^\ast=\overline{\alpha}^{\ast}\circ w$ for every morphism $\alpha\in \Delta$. This can be conveniently reformulated as a $\mathbb{Z}/2$-diagram $X\colon \Delta^{op}\to \Top$, in the sense of \cite[Def.1.1]{Gdiags}.
Similarly, we define a real semi-simplicial space to be a $\mathbb{Z}/2$-diagram $X\colon \Delta^{op}_+\to \Top$.

We will be mostly concerned with the following two examples. By a non-unital topological monoid we mean a possibly non-unital monoid in the monoidal category of spaces with respect to the cartesian product. Let $M$ be a non-unital topological monoid which is equipped with an anti-involution, that is a continuous map of monoids $w\colon M^{op}\to M$ that satisfies $w^2=\id$.

\begin{example}
The real nerve of $M$ is the semi-simplicial space $NM$, the nerve of $M$, with $n$-simplices $N_nM=M^{\times n}$,
and with the levelwise involution
\[(m_1,\dots,m_n)\longmapsto (w(m_n),\dots, w(m_1)).\]
The resulting real semi-simplicial space is denoted $N^{\sigma}M$ (cf. \cite[Def.1.12]{BF}).
\end{example}

\begin{example}
The dihedral nerve of $M$ is the cyclic nerve $N^{cy}M$, with $n$-simplices $N^{cy}_nM=M^{\times n+1}$,
and with the involution defined degreewise by
\[(m_0,m_1,\dots,m_n)\longmapsto (w(m_0),w(m_n),\dots, w(m_1)).\]
The resulting real semi-simplicial space is denoted $N^{di}M$. This involution combined with the semi-cyclic structure define a semi-dihedral object.
\end{example}

Segal's edgewise subdivision functor $\sd_e$ turns a real semi-simplicial space  $X$ into a semi-simplicial $\mathbb{Z}/2$-space. It is defined by precomposing a real semi-simplicial space $X\colon \Delta^{op}_+\to\Top$ with the endofunctor of $\Delta_+$ that sends $[n]=\{0,\dots,n\}$ to $[2n+1]=[n]\amalg [n]^{op}$, and a morphism $\alpha\colon [n]\to [k]$ to $\alpha\amalg\overline{\alpha}$. Since $\sd_e (X^{op})=\sd_e X$, the levelwise involution on $X$ defines a semi-simplicial involution on $\sd_e X$. Thus the thick geometric realization $\|\sd_e X\|$ inherits a $\mathbb{Z}/2$-action.

\begin{definition}\label{defB11}\label{defdy}
The real Bar construction of a non-unital topological monoid with anti-involution $M$ is the $\mathbb{Z}/2$-space $B^{\sigma}M$ defined as the geometric realization of the semi-simplicial space
\[B^{\sigma}M:=\|\sd_e N^{\sigma}M\|\]
with the involution induced by the semi-simplicial involution of $\sd_eN^{\sigma}M$. Similarly, the dihedral Bar construction of $M$ is the $\mathbb{Z}/2$-space $B^{di}M$ defined as the geometric realization of the semi-simplicial space
\[B^{di}M:=\|\sd_e  N^{di}M\|.\]
\end{definition}

\begin{example}\label{classifyingbundles}
Let $\pi$ be a discrete group with the anti-involution defined by inversion $w=(-)^{-1}\colon \pi^{op}\to \pi$. The $\mathbb{Z}/2$-space $B^{\sigma}\pi$ is a classifying space for principal $\pi$-bundles of $\mathbb{Z}/2$-spaces. A model for such a universal bundle is constructed in \cite{MayBdls} as the map
\[Map(E\mathbb{Z}/2,E\pi)\longrightarrow Map(E\mathbb{Z}/2,B\pi),\]
where $E\pi$ denotes the free and contractible $\pi$-space.
The base space is equivalent to the nerve of the functor category $Cat(\mathcal{E}\mathbb{Z}/2,\pi)$ where $\mathcal{E}\mathbb{Z}/2$ is the translation category of the left $\mathbb{Z}/2$-set $\mathbb{Z}/2$ (whose nerve is the classical model for $E\mathbb{Z}/2$), see \cite{MonaModels}. It is easy to see that the nerve of $Cat(\mathcal{E}\mathbb{Z}/2,\pi)$ and the edgewise subdivision of $N^{\sigma}\pi$ are equivariantly isomorphic.
\end{example}

\begin{remark}\label{thinvsthick}
In contrast with the simplicial case, the geometric realization of a semi-simplicial space is in general not equivalent to the geometric realization of its subdivision. However, this is the case if the semi-simplicial space $X$ admits a (levelwise) weak equivalence $X\stackrel{\simeq}{\to} Y$, where $Y$ is a semi-simplicial space which is the restriction of a proper simplicial space. This is because of the commutative diagram
\[\xymatrix{
\|\sd_eX\|\ar[d]\ar[r]^\sim&\|\sd_eY\|\ar[d]\ar[r]^\sim &|\sd_eY|\ar[d]^{\cong}
\\
\|X\|\ar[r]_\sim&\|Y\|\ar[r]_{\sim}&|Y|
}\]
where $|-|$ denotes the thin geometric realization. For the nerve and the cyclic nerve this condition holds if the monoid $M$ is weakly equivalent to a unital monoid. In the examples of interest in this paper we will always be in this situation. Thus the underlying non-equivariant homotopy types of $B^{\sigma}M$ and $B^{di}M$ are those of the Bar construction $BM$ and the cyclic Bar construction $B^{cy}M$, respectively.

If $X$ is a real semi-simplicial space,  $\|X\|$ also inherits an involution, by the formula
\[[x\in X_n,(t_0,\dots, t_n)\in \Delta^n]\longmapsto [w(x)\in X_n,(t_n,\dots, t_0)\in \Delta^n].\]
The map $\gamma\colon \|\sd_e X\|\to \|X\|$ that sends $[x,(t_0,\dots,t_n)]$ to $[x,\frac{1}{2}(t_0,\dots,t_n,t_n,\dots,t_0)]$ is equivariant.
If $X$ admits a levelwise equivariant equivalence $X\stackrel{\simeq}{\to}Y$, where $Y$ is the restriction of a proper real simplicial space, then $\gamma$ is an equivariant equivalence. This is again because of the above diagram, since in the presence of degeneracies the map $\gamma$ descends to an equivariant homeomorphism $|\sd_eY|\cong |Y|$. In general these two actions do not agree, and we choose to work with the subdivided version because it gives us control over the fixed points. We will give a weaker condition that guarantees that these actions are equivalent for nerves of monoids in Lemma \ref{lemmathicksub}.
\end{remark}

We now proceed by analyzing the fixed points of $B^{\sigma}M$ and $B^{di}M$.
The fixed points of $B^{\sigma}M$ are modeled not by a monoid, but by a category.
Let us define a topological category $\sym M$ (without identities) as follows. Its space of objects is the fixed points space $M^{\mathbb{Z}/2}$, and the morphisms $m\to n$ consist of the subspace of elements $l\in M$ with $m=w(l) n l$. Composition is defined by $l\circ k=l\cdot k$. The following is analogous to \cite[Prop.1.13]{BF}.

\begin{proposition}\label{fixedB11}
Let $M$ be a non-unital topological monoid with anti-involution. The $\mathbb{Z}/2$-fixed points of $B^{\sigma}M$ are naturally homeomorphic to the classifying space of $\sym M$, whose nerve is the Bar construction
\[N\sym M\cong N(M;M^{\mathbb{Z}/2})\]
of the right action of $M$ on $M^{\mathbb{Z}/2}$ given by $n\cdot l=w(l)nl$.
\end{proposition}

\begin{proof}
The geometric realization of semi-simplicial sets commutes with fixed points of finite groups. This can be easily proved by induction on the skeleton filtration, since fixed points commute with pushouts along closed inclusions and with filtered colimits along closed inclusions. Thus the fixed points space $(B^{\sigma}M)^{\mathbb{Z}/2}$ is homeomorphic to the geometric realization of the semi-simplicial space $(\sd_e N^\sigma M)^{\mathbb{Z}/2}$.

There is an equivariant isomorphism of semi-simplicial $\mathbb{Z}/2$-spaces between $\sd_e NM$ and $N\sd_e M$, where $\sd_e M$ is the edgewise subdivision of the category $M$ (a.k.a. the twisted arrow category). This is the topological category with $\mathbb{Z}/2$-action whose space of objects is $M$, and where the space of morphisms $m\to n$ is the subspace of $M\times M$ of pairs $(l,k)$ such that $n=lmk$. Composition is defined by
\[(l,k)\circ (l',k')=(l'l,kk').\]
The involution on $\sd_e  M$ sends an object $m$ to $w(m)$, and a morphism $(l,k)$ to $(w(k),w(l))$. Since the the nerve functor commute with fixed points, the fixed points of $N\sd_e M$ are isomorphic to the nerve of the fixed points category of $\sd_e M$. Its objects are the fixed objects $M^{\mathbb{Z}/2}$, and its morphisms the pairs $(l,k)$ where $k=w(l)$. This is isomorphic to the category $\sym M$.
\end{proof}

A similar argument shows that the fixed points of the subdivided dihedral nerve of $M$ are isomorphic to the two-sided Bar construction
\[
(N_{2n+1}^{di}M)^{\mathbb{Z}/2}\cong N_n(M^{\mathbb{Z}/2};M;M^{\mathbb{Z}/2})
\]
of the left action of $M$ on $M^{\mathbb{Z}/2}$ defined by $m\cdot n:=mnw(m)$ and the right action $n\cdot m:=w(m)nm$. Thus the semi-simplicial space $(\sd_e N^{di}M)^{\mathbb{Z}/2}$ is isomorphic to the nerve of a category $\sym^{cy}M$. Its objects are the pairs $(n_0,n_1)$ of fixed points of $M^{\mathbb{Z}/2}$. A morphism $m\colon (n_0,n_1)\to (n'_0,n'_1)$ is an element $m\in M$ such that $n_0'=m\cdot n_0$ and $n_1= n_1'\cdot m$. We then obtain the following.
\begin{proposition}\label{fixeddihedral}
Let $M$ be a non-unital topological monoid with anti-involution. The $\mathbb{Z}/2$-fixed points of $B^{di}M$ are naturally homeomorphic to the classifying space of $\sym^{cy} M$, whose nerve is the two-sided Bar construction
\[
N\sym^{cy}M\cong N(M^{\mathbb{Z}/2};M;M^{\mathbb{Z}/2}).\tag*{$\qed$}
\]
\end{proposition}

We are now able to determine the homotopy invariance property of $B^\sigma$ and $B^{di}$. Here and throughout the paper, we will call a $\mathbb{Z}/2$-equivariant map of $\mathbb{Z}/2$-spaces $f\colon X\to Y$ a weak equivalence if both $f$ and its restriction on fixed points $f\colon X^{\mathbb{Z}/2}\to Y^{\mathbb{Z}/2}$ induce isomorphisms on all homotopy groups.

\begin{lemma}
Let $f\colon M\to M'$ be a map of non-unital topological monoids with anti-involutions, and suppose that $f$ is a weak equivalence of $\mathbb{Z}/2$-spaces. Then
\[B^{\sigma}f\colon B^{\sigma}M\longrightarrow B^{\sigma}M'\ \ \ \ \ \ \mbox{and} \ \ \ \ \ \ \  B^{di}f\colon B^{di}M\longrightarrow B^{di}M'\]
are $\mathbb{Z}/2$-equivalences of spaces. 
\end{lemma}

\begin{proof}
Non-equivariantly $Bf$ is an equivalence, since realizations of semi-simplicial spaces preserve levelwise equivalences.
Since realizations commute with fixed points it remains to show that $(\sd_e NM)^{\mathbb{Z}/2}\to (\sd_e NM')^{\mathbb{Z}/2}$ is a levelwise equivalence. By Proposition \ref{fixedB11} this is the map 
\[f^{\times n}\times f^{\mathbb{Z}/2}\colon M^{\times n}\times M^{\mathbb{Z}/2}\longrightarrow (M')^{\times n}\times (M')^{\mathbb{Z}/2}\]
which is an equivalence by assumption. A similar argument can be applied to $ B^{di}f$.
\end{proof}

We further analyze the functors $B^{\sigma}$ and $B^{di}$.
The following property will be crucial in the definition of the $L$-theoretic assembly map of \S\ref{secass}.
\begin{lemma}\label{funnymap}
Let $\pi$ be a well-pointed topological group with anti-involution $w=(-)^{-1}$ given by inversion. There is an equivariant map $\lambda\colon B\pi\to B^{\sigma}\pi$,
where $B\pi$ has the trivial involution, which is non-equivariantly homotopic to the identity. On fixed points, the composite
\[
B\pi\stackrel{\lambda}{\longrightarrow} (B^{\sigma}\pi)^{\mathbb{Z}/2}\stackrel{\iota}{\longrightarrow}  B\pi
\]
with the fixed points inclusion $\iota \colon (B^{\sigma}\pi)^{\mathbb{Z}/2}\to B\pi$ is homotopic to the identity. This exhibits $B\pi$ as a retract of $(B^{\sigma}\pi)^{\mathbb{Z}/2}$. If moreover $\pi$ is discrete, there is a further splitting
\[
(B^{\sigma}\pi)^{\Z/2}\simeq \coprod_{\{[g]\ |\ g^2=1\}}BZ_\pi\langle g\rangle,\]
where the coproduct runs through the conjugacy classes of the elements of $\pi$ of order $2$, and $Z_\pi\langle g\rangle$ is the centralizer of $g$ in $\pi$. Then $\lambda$ corresponds to the inclusion of the summand $g=1$.
\end{lemma}

\begin{proof}
By Remark \ref{thinvsthick}  we may work with the thin realization of the nerve of $\pi$.
We define a map $\lambda \colon N_p\pi\to (\sd_e N^{\sigma}\pi)_p=N^{\sigma}_{2p+1}\pi$ degreewise by
\[
\lambda(g_1,\dots,g_p)=(g_1,\dots,g_p,1,g^{-1}_p,\dots,g^{-1}_1).
\]
This map is clearly simplicial and equivariant, and it induces an equivariant map $\lambda \colon |N\pi|\to |\sd_e N^{\sigma}\pi|\cong |N^\sigma\pi|$ on realizations. This map sends
\[
[(g_1,\dots,g_p);(t_0,\dots,t_p)]\longmapsto[(g_1,\dots,g_p,1,g_{p}^{-1},\dots,g_{1}^{-1});\tfrac{1}{2}(t_0,\dots,t_p,t_p,\dots,t_0)].
\]
There is a homotopy $\frac{1}{2}(t,t)\simeq (t,0)$ that keeps the sum of the two components constant. This induces a homotopy between $\lambda$ and
\[
[(g_1,\dots,g_p,1,g_{p}^{-1},\dots,g_{1}^{-1});(t_0,\dots,t_p,0,\dots,0)]=[(g_1,\dots,g_p);(t_0,\dots,t_p)]=\id.
\]
The same homotopy defines a homotopy between $\iota\circ\lambda$ and the identity.

Now let us assume that $\pi$ is discrete.  The fixed points space $(B^{\sigma}\pi)^{\Z/2}$ is the classifying space of the category $\sym \pi$ of \ref{fixedB11}. The objects of this category are the elements of $\pi$ of order two. A morphism $g\to g'$ is an element $h$ of $\pi$ such that $g=h^{-1}g'h$. Each component of $\sym \pi$ is then represented by the conjugacy class of an element $g$ of order two, and the automorphism group of $g$ is precisely $Z_\pi\langle g\rangle$.

\end{proof}

In the same way as the cyclic nerve of a group-like monoid $G$ is a model for the free loop-space, the dihedral nerve is a $\mathbb{Z}/2$-equivariant model for the free loop space $\Lambda^\sigma B^\sigma G$, where $\Lambda^{\sigma}=\Map(S^{\sigma},-)$ is the free loop space of the sign-representation sphere $S^\sigma$. Establishing this equivalence becomes delicate when $G$ does not have a strict unit. Classically the map $B^{cy}G\to \Lambda BG$ is constructed from the $S^1$-action on $B^{cy}G$ induced by the cyclic structure, but this $S^1$-action is not well-defined on the thick realization. Let $M$ be a non-unital topological monoid with anti-involution, and let $M_+$ denote $M$ with a formally adjoined unit which is fixed by the anti-involution. 
Let us consider the diagram
\[
\xymatrix{
B^{di}M=\|\sd_e N^{di}M\|\ar[r]
&
 \|\sd_e N^{di}(M_+)\|\ar[r]^-\simeq
 &
 |\sd_e N^{di}(M_+)|\ar[r]^{\cong}&|N^{di}(M_+)|\ar[d]
\\
\hspace{2cm}\ \Lambda^\sigma B^{\sigma}M\ar@{=}[r]&\Lambda^\sigma \|\sd_e N^{\sigma}M\|\ar[r]&\Lambda^\sigma \|N^{\sigma}M\|& \Lambda^\sigma |N^{\sigma}(M_+)|\ar[l]_-{\cong}.
}
\]
The first map is induced by the inclusion $M\to M_+$. The second map is the canonical map to the thin geometric realization, which is an equivalence since the inclusion of the disjoint unit is a cofibration. The third map is the canonical homeomorphism $\gamma$ from Remark \ref{thinvsthick}. The vertical map is adjoint to the composite $S^{\sigma}\times |N^{di}(M_+)|\to |N^{di}(M_+)|\to  |N^{\sigma}(M_+)|$ of the circle action induced by the cyclic structure and the canonical projection. The next map is the isomorphism between $N(M_+)$ and the free simplicial space $E(NM)$ on the semi-simplicial space $NM$, followed by the isomorphism $|E(NM)|\cong \|NM\|$ (see Lemma \cite[1.8]{ERW}). The last map is again the map $\gamma$ from Remark \ref{thinvsthick}.


\begin{definition}\label{defq1}
We say that a non-unital topological monoid with anti-involution $M$ is quasi-unital if:
\begin{enumerate}[i)]
\item There is a unital well-pointed topological monoid $A$ and a map of non-unital monoids $f\colon M\xrightarrow{\sim} A$ which is an equivalence on underlying spaces,
\item There is a right $A$-space $B$ and a map of $M$-spaces $\phi\colon M^{\mathbb{Z}/2}\xrightarrow{\sim} B$ which is an equivalence on underlying spaces, where $M$ acts on $B$ via $f$,
\item There is a point $e\in B$ such that $e\cdot f(m)=\phi(w(m)m)$ for every $m\in M$, and $B$ is well-pointed at $e$.
\end{enumerate}
We say that $M$ is group-like if $\pi_0M$ is a group.
\end{definition}

When $M$ is unital, one can of course set $f$ and $\phi$ to be the identity map. The element $e$ plays the role of the unit as an element in the fixed points space $M^{\mathbb{Z}/2}$.

\begin{lemma}\label{lemmathicksub}
Let $M$ be a non-unital topological monoid with anti-involution which is quasi-unital. 
Then the map
\[\gamma\colon B^\sigma M= \|\sd_e N^{\sigma}M\| \stackrel{\simeq}{\longrightarrow}\|N^{\sigma}M\|\]
is a $\mathbb{Z}/2$-equivariant weak equivalence.
\end{lemma}

\begin{proof}
Non-equivariantly, $\gamma$ is an equivalence by Remark \ref{thinvsthick}, since $NM\xrightarrow{\sim}NA$ is an equivalence and $NA$ is a proper simplicial space. 
In order to show that $\gamma$ is an equivalence on fixed points, we factor it as
\[
\|\sd_e N^\sigma M\|\cong |E(\sd_e N^\sigma M)|\longrightarrow  |\sd_e(N^\sigma M_+)|\stackrel{\gamma}{\cong}  |N^\sigma M_+|\cong \|N^\sigma M\|
\]
where the arrow is induced by the inclusion $M\to M_+$. We claim that this map is an equivalence on fixed points. By Proposition \ref{fixedB11} this is the geometric realization of the map of simplicial spaces
\[
|E(\sd N^\sigma M)|^{\mathbb{Z}/2}\cong |N(\sym M)_+|\longrightarrow | N\sym (M_+) |\cong |\sd_e(N^\sigma M_+)|^{\mathbb{Z}/2}
\]
where $(\sym M)_+$ is the category $\sym M$ with freely adjoined identities. We observe that the category $\sym (M_+)$ also has freely added identities. Thus by denoting $\sym_+M$ the category  $\sym (M_+)$ with the identities removed, we see that  $\sym_+M$ is a well-defined category and that $\sym (M_+)=(\sym_+M)_+$. Moreover $\sym M$ is the full subcategory of $\sym_+M$ on the objects in $M^{\mathbb{Z}/2}\subset M^{\mathbb{Z}/2}\amalg 1= Ob \sym_+M$. Therefore we need to show that the map
\[
\|N(\sym M)\|\cong |N(\sym M)_+|\longrightarrow | N\sym (M_+) |\cong \|N(\sym_+ M)\|
\] 
induced by the inclusion of non-unital topological categories $\iota \colon \sym M\to \sym_+M$ is a weak equivalence. The nerve of $\sym_+ M$ is the Bar construction of the right action of $M$ on the fixed points space with a disjoint basepoint $M^{\mathbb{Z}/2}_+$, where $M$ acts on $M^{\mathbb{Z}/2}$ as usual by $n\cdot l=w(l)nl$, and on the added basepoint by $+\cdot l=w(l)l$. 
The nerve of the map $\iota$ identifies under the isomorphism of Proposition \ref{fixedB11} with the canonical map
\[
N(M;M^{\mathbb{Z}/2})\longrightarrow N(M;M^{\mathbb{Z}/2}_+)
\]
induced by the $M$-equivariant inclusion $M^{\mathbb{Z}/2}\to M^{\mathbb{Z}/2}_+$.
Since $M$ is quasi-unital, the realization of this map is weakly equivalent to the thick realization of the map $\iota\colon N(A;B)\to N(A;B_+)$ induced by the inclusion $B\to B_+$, where $A$ acts on $+$ by $+\cdot a=e\cdot a$. The last condition of Definition \ref{defq1} guarantees that the map $(f,\phi)\colon N(M;M^{\mathbb{Z}/2}_+)\to N(A;B_+)$ is compatible with the last face map. Thus it is sufficient to show that $\iota\colon N(A;B)\to N(A;B_+)$ is an equivalence. Since $A$ is unital and well-pointed these semi-simplicial spaces admit degeneracies, and the thick realization of this map is weakly equivalent to its thin realization. Therefore we can exhibit a simplicial retraction $r\colon N(A;B_+)\to N(A;B)$ for $\iota$, and a simplicial homotopy between $\iota\circ r$ and the identity. The map $r$ is induced by the map of $A$-spaces $B_+\to B$ which is the identity on $B$ and that sends $+$ to $e$. The homotopy is induced by the morphism $(e,e)\colon e\to +$ in the topological category $\sym (B_+)$ whose nerve is $N(A;B_+)$. 
\end{proof}

\begin{lemma}\label{lemmafreeloop}
Let $M$ be a non-unital topological monoid with anti-involution which is quasi-unital and group-like. 
Then the map
\[B^{di}M=\|\sd_e N^{di}M\|\stackrel{\simeq}{\longrightarrow}\Lambda^\sigma \|N^{\sigma}M\|\]
is a $\mathbb{Z}/2$-equivariant weak equivalence.
\end{lemma}

\begin{proof}
This map is the middle vertical map of a commutative diagram
\[
\xymatrix{
M\ar[r]\ar[d]^\simeq&\|\sd_e N^{di}M\|\ar[r]^p\ar[d]&\|\sd_e N^{\sigma}M\|\ar[d]^{\simeq}_\gamma
\\
\Omega^{\sigma}\|N^{\sigma}M\|\ar[r]&\Lambda^\sigma\|N^{\sigma}M\|\ar[r]_{\ev_0}&\|N^{\sigma}M\|
}
\]
where $\ev_0$ is the evaluation map, which is a fibration, and $p$ projects off the first coordinate in each simplicial level. The left vertical map is an equivalence by an equivariant version of the group-completion theorem of \cite{Kristian} (see also \cite[\S 6.2]{Thesis} and \cite[Thm. 4.0.5]{Stiennon}). 
We claim that when $M$ is group-like the top row is a fiber sequence of $\mathbb{Z}/2$-spaces, and this will end the proof. 

We start by observing that since $M$ is quasi-unital and group-like the maps
\[(-)\cdot m, m\cdot(-)\colon M\longrightarrow M\ \ \ \ \ \ \ \ \mbox{and} \ \ \ \ \ \ \ \ w(m)(-)m\colon M^{\mathbb{Z}/2}\longrightarrow M^{\mathbb{Z}/2}\]
are weak equivalences. Indeed if we let $m^{-1}$ denote an element of $M$ whose class in $\pi_0M$ is an inverse for the class of $m$, we see that the composites of $m\cdot (-)$ with $m^{-1}\cdot (-)\colon M\to M$ are homotopic to multiplication with an element $1\in M$ whose component is the unit of $\pi_0M$. Since $M$ is quasi-unital, there is a square
\[
\xymatrix@C=60pt@R=14pt{
M\ar[r]^-\simeq\ar[d]_-{1\cdot(-)}&A\ar[d]^{1\cdot(-)=\id}
\\
M\ar[r]_-{\simeq}& A
}
\]
that commutes up to homotopy. Thus $1\cdot(-)$ is a weak equivalence, and so is $m\cdot (-)$. A similar argument shows that $(-)\cdot m$ is a weak equivalence. Similarly, the compositions of $w(m)(-)m$ with $w(m^{-1})(-)m^{-1}\colon M^{\mathbb{Z}/2}\to M^{\mathbb{Z}/2}$ are homotopic to $w(1)(-)1$, and this is an equivalence since it compares to the action of the unit of $A$ under the equivalence $M^{\mathbb{Z}/2}\to B$.

In order to show that $M$ is the homotopy fiber of $p\colon \|\sd_e N^{di}M\|\to \|\sd_e N^{\sigma}M\|$ we use a criterion of Segal, as stated in Theorem \cite[2.12]{ERW}. The diagrams of $\mathbb{Z}/2$-spaces
\[
\xymatrix{
M\times M^{\times 2n+1}\ar[d]^{p}\ar[r]^{d_n}&M\times M^{\times 2n-1}\ar[d]^{p}
\\
M^{\times 2n+1}\ar[r]_{d_n}&M^{\times 2n-1}
}
\ \ \ \ \ \ \ \ \ \ \ \ \ \ \ \ \ \ 
\xymatrix{
M\times M^{\times 3}\ar[d]^{p}\ar[r]^{d_0}&M\times M\ar[d]^{p}
\\
M^{\times 3}\ar[r]_{d_0}&M
}
\]
are homotopy cartesian. The left-hand square is a strict pull-back and the map $p$ is a fibration. For the right-hand square, we see that the strict pull-back $P$ is isomorphic to $M\times M^{\times 3}$, but under this isomorphism the map from the top left corner
\[
M\times M^{\times 3}\longrightarrow M\times M^{\times 3}\cong P
\]
sends $(m_0,m_1,m_2,m_3)$ to $(m_3m_0m_1,m_1,m_2,m_3)$. Since left and right multiplications in $M$ are weak equivalences this map is a non-equivariant equivalence. On fixed points, this is isomorphic to the map
\[
M^{\mathbb{Z}/2}\times M\times M^{\mathbb{Z}/2} \longrightarrow M^{\mathbb{Z}/2}\times M\times M^{\mathbb{Z}/2}
\]
that sends $(m_0,m_1,m_2)$ to $(w(m_1)m_0m_1,m_1,m_2)$, and this is an equivalence since by the argument above $w(m_1)(-)m_1\colon M^{\mathbb{Z}/2}\to M^{\mathbb{Z}/2}$ is an equivalence. It follows by \cite[Lemma 2.11, Thm.2.12]{ERW} that the square
\[
\xymatrix{
M\times M=(\sd_e N^{di}M)_0\ar@<-7.5ex>[d]_-{p}\ar[r]&\|\sd_e N^{di}M\|\ar[d]^{p}
\\
M=(\sd_e N^{\sigma}M)_0\ar[r]&\|\sd_e N^{\sigma}M\|
}
\]
is homotopy cartesian. Thus the homotopy fibers  of the vertical maps are equivalent.
\end{proof}

The constructions $N^{\sigma}$ and $N^{di}$ extend to categories with duality. We will use this generalization occasionally, mostly in \S\ref{secdeloopings}.

\begin{remark}\label{B11forcat}\label{Bdyforcat}
We recall that a category with strict duality is a category (possibly without identities) $\mathscr{C}$ equipped with a functor $D\colon \mathscr{C}^{op}\to \mathscr{C}$ such that $D^2=\id$. If $\mathscr{C}$ has one object this is the same as a monoid with anti-involution. There is a levelwise involution on the nerve $N\mathscr{C}$ which is defined by
\[
(c_0\stackrel{f_1}{\to}c_1\to\dots\stackrel{f_n}{\to}c_n)\longmapsto (Dc_n\stackrel{Df_n}{\to}Dc_{n-1}\to\dots\stackrel{Df_1}{\to}c_0).
\]
We define $B^{\sigma}\mathscr{C}:=\|\sd_e N^{\sigma}\mathscr{C}\|$. There is a category $\sym \mathscr{C}$ whose objects are the morphisms $f\colon c\to Dc$ such that $Df=f$, and the morphisms $f\to f'$ are the maps $\gamma \colon c\to c'$ such that $f=D(\gamma)f'\gamma$. The considerations of Proposition \ref{fixedB11} extend to give an isomorphism
\[
\sd_e N^{\sigma}\mathscr{C}\cong N\sym \mathscr{C}.
\]
Similarly, there is a construction of the dihedral nerve of a category with strict duality.
An $n$-simplex of the cyclic nerve $N^{cy}\mathscr{C}$ is a string of composable morphisms
\[
c_n\stackrel{f_0}{\to}c_0\stackrel{f_1}{\to}c_1\stackrel{f_2}{\to} c_2\to \dots\to c_{n-1}\stackrel{f_n}{\to}c_n,
\]
and the levelwise involution of the dihedral nerve sends this string to
\[
Dc_0\stackrel{Df_0}{\to}Dc_n\stackrel{Df_n}{\to}Dc_{n-1}\stackrel{Df_{n-1}}{\to} Dc_{n-2}\to \dots\to Dc_{1}\stackrel{Df_1}{\to}Dc_0.
\]
We define $B^{di}\mathscr{C}:=\|\sd_e N^{di}\mathscr{C}\|$.
\end{remark}

\subsection{Ring spectra with anti-involution}\label{secRealSpectra}
Let $A$ be an orthogonal ring spectrum. An anti-involution on $A$ is a map of ring spectra $w\colon A^{op}\to A$ such that $w\circ w=\id$. Since the map $w$ is a strict involution, it gives the underlying orthogonal spectrum of $A$ the structure of an orthogonal $\mathbb{Z}/2$-spectrum (cf. \cite{Schwede}). We remark that the (genuine) fixed-points spectrum $A^{\mathbb{Z}/2}$ is no longer a ring spectrum. In this section we will explain how such an object generalizes the Hermitian Mackey functors of \S\ref{secMackey}, and we will define a spectral version of the category of Hermitian forms over $A$.

We let $I$ be B\"{o}kstedt's category of finite sets and injective maps. Its objects are the natural numbers (zero included), and a morphism $i\to j$ is an injective map $\{1,\dots,i\}\to\{1,\dots, j\}$. We recall that the spectrum $A$ induces a diagram $\Omega^{\bullet}A\colon I\to \Top$ (see e.g. \cite[\S 2.3]{SchlichtkrullUnits}) by sending an integer $i$ to the $i$-fold loop space $\Omega^iA$. We denote its homotopy colimit by
\[\Omega^{\infty}_IA:=\hocolim_I\Omega^{\bullet}A.\]
On the one hand the multiplication of $A$ endows $\Omega^{\infty}_IA$ with the structure of a topological monoid (see \cite[\S 2.2, 2.3]{SchlichtkrullUnits}). On the other hand, the category $I$ has an involution which is trivial on objects and that sends a morphism $\alpha\colon i\to j$ to
\[\overline{\alpha}(s)=j+1-\alpha(i+1-s).\]
The diagram  $\Omega^\bullet A$ has a $\mathbb{Z}/2$-structure in the sense of \cite[Def.1.1]{Gdiags}, defined by the maps
\[
\xymatrix@C=40pt{
\Omega^iA\ar[r]^{\Omega^{i}w}
&
\Omega^iA\ar[r]^{\Omega^{i}\chi_i}
&
\Omega^iA\ar[r]^{(-)\circ\chi_i}
&
\Omega^iA.
}
\]
Here $\chi_i\in\Sigma_i$ is the permutation the reverses the order $\{1,\dots,i\}$, applied both to the sphere $S^i$ and through the orthogonal spectrum structure of $A$. This induces a $\mathbb{Z}/2$-action on the space $\Omega^{\infty}_IA$. These two structures make $\Omega^{\infty}_IA$ into a topological monoid with anti-involution. In case $A$ is non-unital, $\Omega^{\infty}_IA$ is a non-unital topological monoid with anti-involution. We refer to \cite{THRmodels} for the details.

\begin{remark}\label{infiniteloopG}
Throughout the paper, we will make extensive use of the fact that, as a $\mathbb{Z}/2$-space, $\Omega^{\infty}_IA$ is equivalent to the genuine equivariant infinite loop space of $A$. There is a comparison map
\[
\iota_\ast\colon \hocolim_{n\in\mathbb{N}}\Omega^{n\rho+1}A_{n\rho+1}\longrightarrow\Omega^{\infty}_IA,
\]
where $\rho$ is the regular representation of $\mathbb{Z}/2$ and $A_V=Iso(\mathbb{R}^d,V)_+\wedge_{O(d)}A_d$ is the value of $A$ at a $d$-dimensional $\mathbb{Z}/2$-representation $V$. This map is induced by the inclusion $\iota\colon \mathbb{N}\to I$ that sends $n$ to $2n+1$ and the unique morphism $n\leq m$ to the equivariant injection 
\[\iota(n\leq m)(k):=k+m-n.\]
The failure of $\iota_\ast$ from being a non-equivariant equivalence is measured by the action of the monoid of self-injections of $\mathbb{N}$, and this action is homotopically trivial since $A$ is an orthogonal spectrum (see \cite[\S2.5]{SagSchCpletion}). A similar comparison exists equivariantly, and the comparison map is an equivariant equivalence since $A$ is an orthogonal $\mathbb{Z}/2$-spectrum. The details can be found in \cite{THRmodels}.
 \end{remark}

Since $\Omega^{\infty}_IA$ is a topological monoid with anti-involution, there is an action
\[
\Omega^{\infty}_IA\times (\Omega^{\infty}_IA)^{\mathbb{Z}/2}\longrightarrow (\Omega^{\infty}_IA)^{\mathbb{Z}/2}
\]
defined by $a\cdot b:=abw(a)$ where $w$ denotes the anti-involution of $\Omega^{\infty}_IA$.
We use this action to define a category of Hermitian forms over $A$ in a way analogous to the category of Hermitian forms over a Hermitian Mackey functor of \S\ref{secHermMackey}.

\begin{definition}
We let $M^{\vee}_n(A)$ be the (non-unital) ring spectrum
\[M^{\vee}_n(A)=\bigvee_{n\times n} A,\]
where the multiplication is defined by the maps $M^{\vee}_n(A_i)\wedge M^{\vee}_n(A_j)\to M^{\vee}_n(A_{i+j})$ that send $((k,l),a)\wedge ((k',l'),a')$, where $(k,l),(k',l')\in n\times n$ indicate the wedge component, to $((k,l'),a\cdot a')$ if $l=k'$, and to the basepoint otherwise. 
\end{definition}
The anti-involution $w\colon A^{op}\to A$ induces an anti-involution on $M^{\vee}_n(A)$, defined as the composite
\[M^{\vee}_n(A)^{op}=(\bigvee_{n\times n}A)^{op}
\stackrel{\tau}{\longrightarrow}\bigvee_{n\times n}A^{op}
\stackrel{\vee w}{\longrightarrow}\bigvee_{n\times n}A=M^{\vee}_n(A)\]
where $\tau$ is the automorphism of $n\times n$ which swaps the product factors.
We now let $\widehat{M}^{\vee}_n(A)$ be the non-unital topological monoid with anti-involution
\[\widehat{M}^{\vee}_n(A):=\Omega^{\infty}_IM^{\vee}_n(A).\]
We let $\widehat{GL}^{\vee}_n$ be the subspace of invertible components, defined as the pullback of non-unital topological monoids with anti-involution
\[
\xymatrix{
\widehat{GL}^{\vee}_n(A)\ar[r]\ar[d]&\widehat{M}^{\vee}_n(A)\ar[d]\\
GL_n(\pi_0 A)\ar[r]&M_n(\pi_0 A)\rlap{\ .}
}
\]
Here $\pi_0A$ is the ring of components of $A$ with the induced anti-involution, and $GL_n(\pi_0A)$ its ring of invertible $(n\times n)$-matrices. The involution on $GL_n(\pi_0A)$ is by entrywise involution and transposition.
The right vertical map is the composite
\[\Omega^{\infty}_IM^{\vee}_n(A)\stackrel{\sim}{\longrightarrow}\Omega^{\infty}_I\prod_{n\times n}A{\longrightarrow}M_n(\pi_0A)\]
which is both equivariant and multiplicative.

\begin{definition}
An $n$-dimensional Hermitian form on $A$ is an element of the fixed points space $\widehat{GL}^{\vee}_n(A)^{\mathbb{Z}/2}$. These form a category $\sym{\widehat{GL}^{\vee}_n(A)}$ as in Proposition \ref{fixedB11}, and we define
\[\Herm_A:=\coprod_{n\geq 0}\sym{\widehat{GL}^{\vee}_n(A)}.\]
\end{definition}

\begin{remark}
The anti-involution of $A$ induces a functor $D\colon \mathscr{M}^{op}_A\to \mathscr{M}_A$ on the category $\mathscr{M}_A$ of right module spectra. It is defined by the spectrum of module maps
\[D(P)=\Hom_A(P,A_w),\]
where $A_w$ is the spectrum $A$ equipped with the right $A$-module structure
\[A\wedge A\stackrel{\tau}{\longrightarrow}A\wedge A\xrightarrow{w\wedge\id}A\wedge A\stackrel{\mu}{\longrightarrow}A,\]
where $\tau$ is the symmetry isomorphisms and $\mu$ is the multiplication of $A$.
The ring spectrum of $(n\times n)$-matrices on $A$ is usually defined as the endomorphism spectrum $End(\bigvee_n A)$ of the sum of $n$-copies of $A$. Since $\Hom_A(A,P)$ is canonically isomorphic to $P$, there is an isomorphism of ring spectra $End(\bigvee_n A)\cong\prod_n\bigvee_n A$. The module $\bigvee_nA$ is homotopically self-dual, as the inclusion of wedges into products
\[\bigvee_n A\stackrel{\sim}{\longrightarrow}\prod_n A\cong D(\bigvee_nA)\]
is a natural equivalence. This defines a homotopy coherent involution on $End(\bigvee_n A)$, and one could define Hermitian forms as the homotopy fixed points of this action (this is essentially the approach of \cite{Spitzweck}).

The point of our construction is to refine this homotopy coherent action to a genuine equivariant homotopy type which incorporates the fixed-points spectrum of $A$ which, morally speaking, determines the notion of ``symmetry'' for the associated Hermitian forms.
The inclusion $M^{\vee}_n(A)\to End(\bigvee_n A)$ is a weak equivalence, it is coherently equivariant, and $M^{\vee}_n(A)$ has a strict involution which defines such a genuine homotopy type. The price we pay is that $M^{\vee}_n(A)$ is non-unital and it has only partially defined block sums (see \S\ref{KRspace}), but it will give rise to small and manageable models for the real $K$-theory spectrum. A different such model for the matrix ring has been provided in \cite{KroThesis}.
%
\end{remark}

The action $\Omega^{\infty}_IA\times (\Omega^{\infty}_IA)^{\mathbb{Z}/2}\to(\Omega^{\infty}_IA)^{\mathbb{Z}/2}$ induces an action of the multiplicative monoid $\pi_0A\cong \pi_0\Omega^{\infty}_IA$ on the homotopy group $\pi_0A^{\mathbb{Z}/2}\cong \pi_0(\Omega^{\infty}_IA)^{\mathbb{Z}/2}$ of the fixed-points spectrum. We recall that for every $\mathbb{Z}/2$-spectrum $A$, the Abelian groups $\pi_0A$ and $\pi_0A^{\mathbb{Z}/2}$ form a $\mathbb{Z}/2$-Mackey functor $\underline{\pi}_0A$.

\begin{proposition}\label{pi0Hermitian}
Let $A$ be a ring spectrum with anti-involution. The action of $\Omega^{\infty}_IA$ on $(\Omega^{\infty}_IA)^{\mathbb{Z}/2}$ defines the structure of a Hermitian Mackey functor on the $\mathbb{Z}/2$-Mackey functor $\underline{\pi}_0A$. Moreover, there is an isomorphism of Hermitian Mackey functors $\underline{\pi}_0 M^{\vee}_nA\cong M_n(\underline{\pi}_0A)$, where $M_n(\underline{\pi}_0A)$ is the Mackey functor of matrices of Definition \ref{MackeyMatrix}.
\end{proposition}

\begin{proof}
The multiplication of $\pi_0A$ clearly anti-commutes with the involution, since it does so for $\Omega^{\infty}_IA$. Similarly, the relation between the restriction and the action is satisfied, because the restriction $R\colon \pi_0A^{\mathbb{Z}/2}\to \pi_0A$ corresponds to the fixed points inclusion of $\Omega^{\infty}_IA$. In order to verify the other conditions we describe the action of $\pi_0A$. 

Let us represent an element of $\pi_0A$ by a map $f\colon S^n\to A_n$, and an element of $\pi_0A^{\mathbb{Z}/2}$ by an equivariant map $x\colon S^{m\rho}\to A_{m\rho}$. We recall that there is a canonical isomorphism $S^{n\rho}\cong S^n\wedge S^n$, where the action on $S^n\wedge S^n$ swaps the two smash factors.
The action is then the homotopy class of the map
\[
f\cdot x\colon S^{(n+m)\rho}\cong S^n\wedge S^{m\rho}\wedge S^n\xrightarrow{f\wedge x\wedge (w\circ f)} A_{n}\wedge A_{m\rho}\wedge A_n\stackrel{\mu}{\longrightarrow} A_{n+m\rho+n}\cong A_{(n+m)\rho}
\]
where the involution on $S^n\wedge S^{m\rho}\wedge S^n$ swaps the two $S^n$ smash factors and acts on $S^{m\rho}$, and on $A_{n}\wedge A_{m\rho}\wedge A_n$ it acts componentwise and swaps the two $A_n$ factors. The last map is the multiplication of $A$, and it is equivariant since the diagram
\[
\xymatrix@C=40pt{
A_n\wedge A_{m\rho}\wedge A_n\ar[d]_{\mu}\ar[r]^-{\tau_3}&A_n\wedge A_{m\rho}\wedge A_n\ar[r]^-{w\wedge w\wedge w}&A_n\wedge A_{m\rho}\wedge A_n\ar[d]^{\mu}
\\
A_{n+m\rho+n}\ar[r]_-w&A_{n+m\rho+n}\ar[r]_-{\chi_{n,n}}&A_{n+m\rho+n}
}
\]
commutes by the definition of a ring spectrum with anti-involution, where $\tau_3\in\Sigma_3$ reverses the order and $\chi_{n,n}\in\Sigma_{n+2m+n}$ is the permutation that swaps the first and last blocks of size $n$. The bottom map is by definition the action of $A_{(n+m)\rho}$.

The transfer of the class of a map $g\colon S^{2m}\to A_{2m}$ is defined as the class of the map
\[
T(g)\colon S^{m\rho}\xrightarrow{p}S^{2m}\vee S^{2m}\xrightarrow{g\vee (w\circ g)} A_{m\rho}\vee A_{m\rho}\stackrel{\nabla}{\longrightarrow} A_{m\rho}
\]
where the first map collapses the fixed points $S^m\subset S^{m\rho}$, the last map is the fold, and the involutions act on, and permute, the wedge summands.
The relations $f\cdot T(g)=T(fgw(f))$ and $(f+f')\cdot x=f\cdot x+f'\cdot x+T(fR(x)w(f'))$ for $f,f',g$ in $\pi_0A$ and $x$ in $\pi_0A^{\mathbb{Z}/2}$ are now an immediate consequence of the naturality of the fold map, and of the distributivity of the smash product of pointed spaces over the wedge sum.

Let us now consider the Mackey functor $\underline{\pi}_0 M^{\vee}_n(A)$.
Even though $M^{\vee}_n(A)$ is not unital, the map $M^{\vee}_n(A)\to \prod_n\bigvee_n A$ is an equivalence, and therefore $\pi_0M^{\vee}_n(A)\cong M_n(\pi_0 A)$ is a unital ring. Moreover the inclusion of indexed wedges into indexed products gives an equivalence
\[
(\Omega_IM^{\vee}_n(A))^{\mathbb{Z}/2}\stackrel{\simeq}{\longrightarrow}(\Omega_I(\prod_{n\times n}A))^{\mathbb{Z}/2}\cong (\prod_{n\times n}\Omega_IA)^{\mathbb{Z}/2}\cong(\prod_n(\Omega_IA)^{\mathbb{Z}/2})\times \prod_{1\leq i<j\leq n}(\Omega_IA).
\]
Under this equivalence the action of $\Omega_IM^{\vee}_n(A)$ on $(\Omega_IM^{\vee}_n(A))^{\mathbb{Z}/2}$ corresponds to the action of $\Omega_IM^{\vee}_n(A)$ on $(\Omega_I\prod_{n\times n}A)^{\mathbb{Z}/2}$ given by $abw(a)$, where the multiplications are the infinite loop spaces of the left and right actions
\[
M^{\vee}_n(A)\wedge (\prod_{n\times n}A)\longrightarrow \prod_{n\times n}A \ \ \ \ \ \ \ \ \ \ (\prod_{n\times n}A)\wedge M^{\vee}_n(A)\longrightarrow \prod_{n\times n}A
\]
defined by the standard matrix multiplication rules. This agrees with the action of Definition \ref{MackeyMatrix}.
\end{proof}

\subsection{The real $K$-theory $\mathbb{Z}/2$-space of a ring spectrum with anti-involution}\label{KRspace}

The goal of this section is to define a $\mathbb{Z}/2$-action on the $K$-theory space of a ring spectrum $A$ with anti-involution $w\colon A^{op}\to A$. We define this action by adapting the group-completion construction of the free $K$-theory space 
\[K(A)=\Omega B\coprod_nB\widehat{GL}_n(A),\]
where $\coprod_nB\widehat{GL}_n(A)$ is group-completed with respect to block-sum,
to the model for the equivariant matrix ring constructed in the previous section.

We recall from Lemma \ref{defB11} that the classifying space of a non-unital monoid with anti-involution $M$ inherits a natural $\mathbb{Z}/2$-action, which we denote by $B^{\sigma}M$. 
Thus the anti-involution on $\widehat{GL}^{\vee}_n(A)$ gives rise to a $\mathbb{Z}/2$-space $B^{\sigma}\widehat{GL}^{\vee}_n(A)$.  The space $\amalg_n B^{\sigma}\widehat{GL}^{\vee}_n(A)$ does not have a strict monoid structure, since the standard block-sum of matrices does not restrict to the matrix rings $M^{\vee}_n(A)$. We can however define a Bar construction using a technique similar to Segal's group completion of partial monoids. The block-sum operation on the ring spectra $M^{\vee}_n(A)$ is a collection of maps
\[\oplus\colon M^{\vee}_n(A)\vee M^{\vee}_k(A)\longrightarrow M^{\vee}_{n+k}(A)\]
induced by the inclusions $n\to n+k$ and $k\to n+k$. We observe that this map commutes with the anti-involutions. There is a simplicial $\mathbb{Z}/2$-space with $p$-simplices
\[
\coprod_{n_1,\dots, n_p}B^{\sigma}\Omega^{\infty}_I(M^{\vee}_{n_1}(A)\vee\dots\vee M^{\vee}_{n_p}(A)).\]
The face maps are induced by the block-sum maps, and the degeneracies are the summand inclusions. This results into a well-defined simplicial object since $B^{\sigma}\Omega^{\infty}_I$ is functorial with respect to maps of ring spectra with anti-involution.
This simplicial structure restricts to the $\mathbb{Z}/2$-spaces
\[
\coprod_{n_1,\dots,n_p} B^{\sigma} \widehat{GL}^{\vee}_{n_1,\dots,n_p}(A)
\]
where $\widehat{GL}^{\vee}_{n_1,\dots,n_p}(A)$ is defined as the pull-back of non-unital monoids with anti-involution
\[\xymatrix{
\widehat{GL}^{\vee}_{n_1,\dots,n_p}(A)\ar[r]\ar[d]&
\Omega^{\infty}_I(M^{\vee}_{n_1}(A)\vee\dots\vee M^{\vee}_{n_p}(A))\ar[d]\\
(GL_{n_1}(\pi_0A)\times\dots\times GL_{n_p}(\pi_0A))\ar@{^(->}[r]&
(M_{n_1}(\pi_0A)\times\dots\times M_{n_p}(\pi_0A))\rlap{\ .}
}
\]


\begin{definition}\label{defKRspace}
The free real $K$-theory space of a ring spectrum with anti-involution $A$ is the $\mathbb{Z}/2$-space defined as the loop space of the thick geometric realization
\[
\KR(A):=\Omega \|\coprod_{n_1,\dots,n_\bullet} B^{\sigma} \widehat{GL}^{\vee}_{n_1,\dots,n_\bullet}(A)\|.\]
\end{definition}

We observe that all of our constructions are homotopy invariant, and therefore the functor $\KR$ sends maps of ring spectra with anti-involution which are stable equivalences of underlying $\mathbb{Z}/2$-spectra to equivalences of $\mathbb{Z}/2$-spaces. We conclude this section by verifying that the underlying space of $\KR(A)$ has the right homotopy type.

\begin{proposition}
Let $A$ be a ring spectrum with anti-involution, and let us denote $A|_1$ the underlying ring spectrum of $A$ and $\KR(A)|_1$ the underlying space of $\KR(A)$. There is a weak equivalence
\[\KR(A)|_1\stackrel{\simeq}{\longrightarrow}K(A|_1).\]
\end{proposition}

\begin{proof}
For convenience, we drop the restriction notation. The inclusion of wedges into products defines an equivalence of ring spectra $M_n^{\vee}(A)\to M_n(A):=\prod_n\bigvee_n A$, and therefore an equivalence of monoids $\widehat{M}^{\vee}_n(A)\to \widehat{M}_n(A)$ on infinite loop spaces. This induces an equivalence of spaces
\[\coprod_n B\widehat{M}^{\vee}_n(A)\longrightarrow \coprod_nB\widehat{M}_n(A)\]
after taking the thick realization.
The block-sum maps of $M_n^{\vee}(A)$ and $M_n(A)$ are compatible, in the sense that the diagram
\[
\xymatrix{
M_n^{\vee}(A)\vee M_k^{\vee}(A)\ar[d]_\sim\ar[r]^-{\oplus}&
M_{n+k}^{\vee}(A)\ar[d]^{\sim}\\
M_n(A)\times M_k(A)\ar[r]_-{\oplus}&
M_{n+k}(A)
}
\]
commutes. It follows that the levelwise equivalences on the Bar constructions
\[\coprod_{n_1,\dots ,n_p}B\Omega^{\infty}_I(M^{\vee}_{n_1}(A)\vee\dots\vee M^{\vee}_{n_p}(A))\stackrel{\sim}{\longrightarrow}\coprod_{n_1,\dots ,n_p}(B\Omega^{\infty}_IM_{n_1}(A))\times\dots\times (B\Omega^{\infty}_IM_{n_p}(A))\]
commute with the face maps. After restricting to invertible components and taking thick geometric realizations this gives the equivalence $KR(A)|_1\simeq K(A|_1)$.
\end{proof}

In the construction of the trace map we will need to compare the dihedral nerve $B^{di} M$ and the free loop space $\Lambda^\sigma B^\sigma M$ for $M=\widehat{GL}^{\vee}_{n_1,\dots,n_p}(A)$. In order to apply Lemma \ref{lemmafreeloop} we will need the following.

\begin{proposition}\label{GLq1gpl}
Suppose that $A$ is levelwise well-pointed, and that the unit map $S^0\to A_0$ is an $h$-cofibration. Then the monoid with anti-involution $\widehat{GL}^{\vee}_{n_1,\dots,n_p}(A)$ is quasi-unital and group-like (see Definition \ref{defq1}).
\end{proposition}

\begin{proof}
By definition $\widehat{GL}^{\vee}_{n_1,\dots,n_p}(A)$ is the monoid of invertible components of $\widehat{M}^{\vee}_{n_1,\dots,n_p}(A):=\Omega^{\infty}_I(M^{\vee}_{n_1}(A)\vee\dots\vee M^{\vee}_{n_p}(A))$, and it is therefore group-like. We show that $\widehat{M}^{\vee}_{n}(A)$ is quasi-unital, which implies that $\widehat{GL}^{\vee}_{n}(A)$ is quasi-unital by restricting to the invertible components. The proof for general $p$ is similar. The inclusion of wedges into products induces an equivalence of non-unital ring spectra
\[
f\colon M^{\vee}_{n}(A)\stackrel{\simeq}{\longrightarrow}M_n(A):=\prod_{n}\bigvee_nA
\]
where the multiplication on $M_nA$ is defined by representing an element in a given spectral degree by a matrix with at most one non-zero entry in each column, and apply the standard matrix multiplication. This induces an equivalence of non-unital topological monoids
\[
\widehat{M}^{\vee}_{n}(A)=\Omega^{\infty}_IM^{\vee}_{n}(A)\stackrel{\simeq}{\longrightarrow} \Omega^{\infty}_IM_n(A),
\]
where $ \Omega^{\infty}_IM_n(A)$ is unital and well-pointed.

The spectrum $\prod_{n\times n}A$ has a $\mathbb{Z}/2$-action, defined by applying the anti-involution entrywise and by composing with the involution of $n\times n$ that switches the factors. The inclusion of indexed wedges into indexed products induces an equivalence of equivariant spectra
\[
\phi\colon M^{\vee}_n(A)=\bigvee_{n\times n}A\stackrel{\simeq}{\longrightarrow} \prod_{n\times n}A,
\]
and therefore an equivalence on fixed points $(\Omega^{\infty}_IM^{\vee}_n(A))^{\mathbb{Z}/2}\xrightarrow{\simeq}(\Omega^{\infty}_I\prod_{n\times n}A)^{\mathbb{Z}/2}$.

The monoid $\prod_{n}\bigvee_nA$ acts on the right on the spectrum $\prod_{n\times n}A$ by right matrix multiplication. Precisely, the action is determined by the maps
\[
(\prod_{n\times n}A_l)\wedge (\prod_{n}\bigvee_nA_k)\longrightarrow \prod_{n\times n}A_{l+k}
\]
that send an element $\{b_{ij}\in A_l\}\wedge (1\leq h_1,\dots,h_n\leq n,a_1,\dots,a_n\in A_k)$ to the matrix with $(i,j)$-entry $\mu(b_{ih_j},a_{j})\in A_{l+k}$ where $\mu$ is the multiplication of $A$. 
There is a second right-action which is defined by left matrix multiplication via the conjugate transposed, namely by sending $\{b_{ij}\}\wedge (1\leq i_1,\dots,i_n\leq n,a_1,\dots,a_n)$ to the matrix with $(i,j)$-entry $\mu(w(a_i),b_{h_ij})$. 
These induce two commuting right actions of $\Omega^{\infty}_IM_n(A)$ on $\Omega^{\infty}_I(\prod_{n\times n}A)$, which we denote respectively by $xm:=x\cdot_1m$ and $w(m)x:=x\cdot_2m$. A straightforward argument shows that these actions satisfy $w(xm)=w(m)w(x)$, where $w$ denotes the involution on $\Omega^{\infty}_I(\prod_{n\times n}A)$. We regard $\Omega^{\infty}_I(\prod_{n\times n}A)$ as a right  $\Omega^{\infty}_IM_n(A)$-space via the action
\[
x\cdot m:=w(m)xm.
\]
We observe that this action restricts to the fixed points space $(\Omega^{\infty}_I\prod_{n\times n}A)^{\mathbb{Z}/2}$, and that the equivalence $(\Omega^{\infty}_IM^{\vee}_n(A))^{\mathbb{Z}/2}\to(\Omega^{\infty}_I\prod_{n\times n}A)^{\mathbb{Z}/2}$ is a map of $\Omega^{\infty}_IM^{\vee}_n(A)$-modules.

Finally, the unit of $\Omega^{\infty}_IM_n(A)$ is mapped to a fixed point under the non-equivariant map $\Omega^{\infty}_IM_n(A)\to \Omega^{\infty}_I\prod_{n\times n}A$ that includes wedges into products. We denote its image by $e$. Since $f$ and $\phi$ are inclusions of wedges into products the relation $e\cdot f(m):=w(f(m))ef(m)=\phi(w(m)m)$ is satisfied for every $m\in M$.
\end{proof}

%
%


\subsection{Connective equivariant deloopings of real $K$-theory}\label{secdeloopings}

We show that the real $K$-theory space of a ring spectrum with anti-involution defined in \S\ref{KRspace} is the equivariant infinite loop space of a (special) $\mathbb{Z}/2$-equivariant $\Gamma$-space. Our construction of these deloopings is an adaptation of Segal's construction (\cite{Segal} and \cite{ShimaShima}) for spectrally enriched symmetric monoidal categories, to a set-up where the symmetric monoidal structure is partially defined.

We start with an explicit definition of the $\mathbb{Z}/2$-$\Gamma$-space, and relate it to Segal's construction in the proof of Proposition \ref{Gamma}. We recall from \cite{ShimaNote} that a $G$-$\Gamma$-space, where $G$ is a finite group, is a functor $X\colon \Gamma^{op}\to \Top^{G}_\ast$ from the category $\Gamma^{op}$ which is a skeleton for the category of pointed finite sets and pointed maps, to the category of pointed $G$-spaces. This induces a symmetric $G$-spectrum whose $n$-th space is the value at the $n$-sphere of the left Kan-extension of $X$ to the category of finite pointed simplicial sets.

For every natural number $n$ and sequence of non-negative integers $\underline{a}=(a_1,\dots, a_n)$ we consider the collections of permutations $\alpha=\{\alpha_{S,T}\in \mathcal{S}_{\sum_{i\in S\amalg T}a_i}\}$ where the indices $S,T$ run through the pairs of disjoint subsets $S\amalg T\subset \{1,\dots,n\}$, and $ \mathcal{S}_k$ denotes the symmetric group. We require that these permutations satisfy the standard conditions of Segal's construction, see e.g. \cite[Def.2.3.1.1]{DGM}. We denote by $\langle\underline{a} \rangle$ the set of such collections $\alpha$ for the $n$-tuple $\underline{a}$. 

Given a ring spectrum with anti-involution $A$ we let ${\KR}(A)\colon \Gamma^{op}\to \Top^{\mathbb{Z}/2}_\ast$ be the functor that sends the pointed set $n_+=\{+,1,\dots, n\}$ to
\[{\KR}(A)_n:=\coprod_{\underline{a}}B^{\sigma}(\widetilde{\langle\underline{a}\rangle}\times\widehat{GL}^{\vee}_{a_1,\dots,a_n}(A))\]
where $\widehat{GL}^{\vee}_{a_1,\dots,a_n}(A)$ is defined in \S\ref{KRspace}, and $\widetilde{\langle\underline{a}\rangle}$ is the category with objects set $\langle\underline{a}\rangle$, and with a unique morphism between any pair of objects. The category $\widetilde{\langle\underline{a}\rangle}$ has a duality that is the identity on objects, and that sends the unique morphism $\alpha\to \beta$ to the unique morphism $\beta\to \alpha$. Thus $\widetilde{\langle\underline{a}\rangle}\times\widehat{GL}^{\vee}_{a_1,\dots,a_n}(A)$ is a non-unital topological category with duality, and $B^{\sigma}$ is the functor of Remark \ref{B11forcat}.

\begin{remark}
Since every object in $\widetilde{\langle\underline{a}\rangle}$ is both initial and final, the projection map \[\widetilde{\langle\underline{a}\rangle}\times\widehat{GL}^{\vee}_{a_1,\dots,a_n}(A)\longrightarrow\widehat{GL}^{\vee}_{a_1,\dots,a_n}(A)\] is an equivalence of topological categories. Moreover by the uniqueness of the morphisms of $\widetilde{\langle\underline{a}\rangle}$ we see that $\sym \widetilde{\langle\underline{a}\rangle}=\widetilde{\langle\underline{a}\rangle}$. Thus the projection map
\[
\sym (\widetilde{\langle\underline{a}\rangle}\times\widehat{GL}^{\vee}_{a_1,\dots,a_n}(A))\cong \sym \widetilde{\langle\underline{a}\rangle}\times\sym\widehat{GL}^{\vee}_{a_1,\dots,a_n}(A) \longrightarrow\sym\widehat{GL}^{\vee}_{a_1,\dots,a_n}(A)
\]
is also an equivalence of categories. Thus $B^{\sigma}(\widetilde{\langle\underline{a}\rangle}\times\widehat{GL}^{\vee}_{a_1,\dots,a_n}(A))$ and $B^{\sigma}\widehat{GL}^{\vee}_{a_1,\dots,a_n}(A)$ are equivariantly equivalent.
\end{remark}

The extra $\widetilde{\langle\underline{a}\rangle}$-coordinate is used for defining ${\KR}(A)$ on morphisms. Given a pointed map $f\colon n_+\to k_+$ and $\alpha\in\langle\underline{a}\rangle$ we let $f_\ast\underline{a}\in\mathbb{N}^{\times k}$ and $f_\ast\alpha \in \langle f_\ast\underline{a}\rangle$ denote respectively
\[
(f_\ast\underline{a})_i:=\sum_{j\in f^{-1}(i)}a_j\ \ \ \ \ \ \ \ \ \ \ \ (f_\ast\alpha)_{S,T}:=\alpha_{f^{-1}S,f^{-1}T}
\]
for every $1\leq i\leq k$ and $S\amalg T\subset \{1,\dots, k\}$. We define $f_\ast\colon {\KR}_n(A)\to {\KR}_k(A)$ by mapping the $\underline{a}$-summand to the $f_\ast\underline{a}$-summand by a map whose first component is $B^{\sigma}$ of the functor
\[
\widetilde{\langle\underline{a}\rangle}\times\widehat{GL}^{\vee}_{a_1,\dots,a_n}(A)\stackrel{}{\longrightarrow}\widetilde{\langle\underline{a}\rangle}\stackrel{f_\ast}{\longrightarrow} \widetilde{\langle f_\ast\underline{a}\rangle},
\]
where the first map is the projection.
The second component of this map is defined as follows. A pair of permutations $\alpha,\beta\in\langle\underline{a}\rangle$ gives rise to a morphism of monoids with anti-involution
\[
(\alpha,\beta)_\ast\colon \widehat{GL}^{\vee}_{a_1,\dots,a_n}(A)\longrightarrow \widehat{GL}^{\vee}_{f_\ast\underline{a}}(A).
\]
It is induced by the map of ring spectra with anti-involution obtained by wedging over $i\in \{1,\dots,k\}$ the maps
\[
(\alpha,\beta)_j\colon \bigvee_{j\in f^{-1}(i)}M^{\vee}_{a_j}(A)\longrightarrow M^{\vee}_{(f_\ast \underline{a})_i}(A)
\]
defined by sending $x\in M^{\vee}_{a_j}(A)$ to $
(\alpha,\beta)_j(x):=\beta_{(f^{-1}i)\backslash j,j}(0\oplus x)\alpha_{(f^{-1}i)\backslash j,j}^{-1}$,
where $0\oplus x$ is the value at $x$ of the block-sum map $
\oplus\colon M^{\vee}_{(f_\ast \underline{a})_i-a_j}(A)\vee M^{\vee}_{a_j}(A)\to M^{\vee}_{(f_\ast \underline{a})_i}(A)$ and $\alpha$ and $\beta$ are considered as permutation matrices.
 More explicitly, an element of $M^{\vee}_{a_j}(A_k)$ consists of a pair $(c,d)\in a_j\times a_j$ and a point $x\in A_k$. This is sent to
\[
(\alpha,\beta)_j(c,d,x)=(\beta_{(f^{-1}i)\backslash j,j}(\iota c),\alpha_{(f^{-1}i)\backslash j,j}(\iota d), x)
\]
where $\iota\colon a_j\to (f_\ast\underline{a})_j$ is the inclusion.
The second component of the map ${\KR}(f)$ is then induced by the functor 
\[\widetilde{\langle\underline{a}\rangle}\times \widehat{GL}^{\vee}_{a_1,\dots,a_n}(A)\to\widehat{GL}^{\vee}_{f_\ast\underline{a}}(A)\]
which is the projection on objects and that sends a morphism $(\alpha,\beta,x)$ in $\widetilde{\langle\underline{a}\rangle}\times \widehat{GL}^{\vee}_{a_1,\dots,a_n}(A)$ to $(\alpha,\beta)_\ast x$. Given a $\mathbb{Z}/2$-$\Gamma$-space $X$, we let $X_{S^1}$ be the first space of the associated $\mathbb{Z}/2$-spectrum, defined as the geometric realization $X_{S^1}=|n\mapsto X_{S^{1}_n}|$ where $S^{1}_n$ is the set of $n$-simplices of the simplicial circle $S^{1}_\bullet=\Delta^1/\partial$.

\begin{proposition}\label{Gamma}
Let $A$ be a ring spectrum with anti-involution.
The functor ${\KR}(A)$ is a special $\mathbb{Z}/2$-$\Gamma$-space in the sense of \cite{ShimaSpecial}. The $\mathbb{Z}/2$-space ${\KR}(A)_{S^{1}}$ is equivalent to the real $K$-theory $\mathbb{Z}/2$-space of \S\ref{defKRspace}. The underlying $\Gamma$-space of ${\KR}(A)$  is equivalent the $K$-theory $\Gamma$-space of $A$.
\end{proposition}

\begin{proof}
Let $\mathcal{F}_A$ be the spectrally enriched category whose objects are the non-negative integers and where the endomorphisms of $k$ consist of the matrix ring $\prod_k\bigvee_kA$. We recall that the Segal construction on $\mathcal{F}_A$ is the $\Gamma$-category enriched in symmetric spectra defined by sending $n_+$ to the category $\mathcal{F}_A[n]$. Its objects are the pairs $\langle \underline{a},\alpha \rangle$ where $\underline{a}=(a_1,\dots, a_n)$ is a collection of non-negative integers, and $\alpha$ is a collection of isomorphisms $\alpha=\{\alpha_{S,T}\colon \sum_{s\in S}a_s+\sum_{t\in T}a_t\to \sum_{i\in S\amalg T}a_i\}$ in the underlying category of $\mathcal{F}_A$ satisfying the conditions of \cite[Def. 2.3.1.1]{DGM}. The spectrum of morphisms $\langle \underline{a},\alpha \rangle\to\langle \underline{b},\beta \rangle$ is non-trivial only if $\underline{a}=\underline{b}$, and it is defined by the collection of elements $\{f_S\in M_{\sum_{s\in S}a_s}(A)\}_{S\subset \underline{n}}$ which satisfy $\beta_{S, T}(f_{S}\oplus f_{T})=f_{S\amalg T}\alpha_{S,T}$. 

There is an equivalence of spectral categories $\mathcal{F}^{\times n}_A\to \mathcal{F}_A[n]$ that sends $\underline{a}=(a_1,\dots, a_n)$ to $\langle \underline{a},\alpha \rangle$ where $\alpha_{S,T}$ is the permutation matrix of the  permutation of $S\amalg T$ that sends the order on $S\amalg T$ induced by the disjoint union of the orders of $S$ and $T$ to the order of  $S\amalg T$ as a subset of $n$ (the point is that $\mathcal{F}_A[n]$ is functorial in $n$ with respect to all maps of pointed sets, whereas $\mathcal{F}^{\times n}_A$ only for order-preserving maps).

Now let $\mathcal{F}^{\vee}_A$ be the equivalent subcategory (without identities) of $\mathcal{F}_A$ with all the objects, but where the endomorphisms of $a$ are the ring spectra $M_{a}^\vee(A)=\bigvee_{n\times n}A$. The space ${\KR}(A)_n$ is roughly the invertible components of the classifying space of the image of $(\mathcal{F}^{\vee}_A)^{\vee n}$ inside $\mathcal{F}_A[n]$. More precisely, there is a commutative square of spectrally enriched categories
\[
\xymatrix{
(\mathcal{F}^{\vee}_{A})^{\vee n}\ar[d]^{\sim}\ar[r]&\mathcal{F}^{\vee}_A[n]\ar[d]\\
\mathcal{F}_{A}^{\times n}\ar[r]_{\sim}&\mathcal{F}_A[n]
}
\]
where $\mathcal{F}^{\vee}_A[n]$ is defined as the subcategory of $\mathcal{F}_A[n]$ on the objects $\langle \underline{a},\alpha \rangle$ where $\alpha_{S,T}$ is a permutation representation, and where the morphisms $\{f_S\}_{S\subset \underline{n}}$ are such that there is a $j\in \underline{n}$ such that $f_S=0$ if $j\notin S$.
The top horizontal arrow is simply the restriction of the bottom horizontal one.

The spectral category $\mathcal{F}^{\vee}_A[n]$ has a strict duality, which is the identity on objects and the anti-involution on the matrix ring $M_{a}^\vee (A)$ on morphisms. We observe that a morphism $\{f_S\}_{S\subset \underline{n}}$ in $\mathcal{F}^{\vee}_A[n]$ is determined by the value $f_{j}$, since for every $S\subset n$ containing $j$ we have that
\[
f_{S}=f_{(S\backslash j)\amalg j}=\beta_{S\backslash j, j} (0_{S\backslash j}\oplus f_{j})\alpha_{S\backslash j, j}^{-1}.
\]
Moreover this equation determines the relation $\beta_{S, T}(f_{S}\oplus f_{T})=f_{S\amalg T}\alpha_{S,T}$, and it follows that the value at $n_+$ of the corresponding $\mathbb{Z}/2$-$\Gamma$-space is
\[
\coprod_{ \underline{a}}B^{\sigma}(\widetilde{\langle\underline{a}\rangle}\times \Omega^{\infty}_I(M^{\vee}_{a_1}(A)\vee\dots\vee M^{\vee}_{a_n}(A)))
.\]
Its invertible components are then by definition ${\KR}_n(A)$ and the functoriality in $\Gamma^{op}$ induced by the ambient category $\mathcal{F}_A[n]$ is the one described above. In particular ${\KR}_n(A)$ is functorial in $n$. The fact that  $f_{S}:=\beta_{S\backslash j, j} (0_{S\backslash j}\oplus f_{j})\alpha_{S\backslash j, j}^{-1}$ determines a well-defined morphism $\langle \underline{a},\alpha \rangle\to\langle \underline{a},\beta \rangle$ follows from the following calculation:
\begin{align*}
&\beta_{S,T}(f_S\oplus 0_T)\alpha_{S,T}^{-1}\\
&=\beta_{S,T}((\beta_{S\backslash j, j} (0_{S\backslash j}\oplus f_{j})\alpha_{S\backslash j, j}^{-1})\oplus 0_T)\alpha_{S,T}^{-1}
\\
&=
\beta_{S,T}(\beta_{S\backslash j, j}\amalg \id_T) (0_{S\backslash j}\oplus f_{j}\oplus 0_T)(\alpha_{S\backslash j, j}^{-1}\amalg \id_T)\alpha_{S,T}^{-1}
\\
&\stackrel{1)}{=}
\beta_{j,S\amalg T\backslash j}(\id_j\amalg \beta_{S\backslash j, T})(\tau_{S\backslash j,j}\amalg \id_T) (0_{S\backslash j}\oplus f_{j}\oplus 0_T)(\alpha_{S\backslash j, j}^{-1}\amalg \id_T)\alpha_{S,T}^{-1}
\\
&\stackrel{2)}{=}
\beta_{j,S\amalg T\backslash j}(\id_j\amalg \beta_{S\backslash j, T})(\tau_{S\backslash j,j}\amalg \id_T) (0_{S\backslash j}\oplus f_{j}\oplus 0_T)(\tau_{j,S\backslash j}\amalg \id_T)
(\id_j\amalg \alpha^{-1}_{S\backslash j,T})\alpha_{j,S\amalg T\backslash j}^{-1}
\\
&=
\beta_{j,S\amalg T\backslash j}(\id_j\amalg \beta_{S\backslash j, T}) (f_{j}\oplus 0_{S\amalg T\backslash j})
(\id_j\amalg \alpha^{-1}_{S\backslash j,T})\alpha_{j,S\amalg T\backslash j}^{-1}
\\
&=\beta_{j,S\amalg T\backslash j}(f_{j}\oplus 0_{S\amalg T\backslash j})
\alpha_{j,S\amalg T\backslash j}^{-1}
\\
&=\beta_{S\amalg T\backslash j,j}\tau_{j,S\amalg T\backslash j}(f_{j}\oplus 0_{S\amalg T\backslash j})\tau_{j,S\amalg T\backslash j}
\alpha_{S\amalg T\backslash j,j}^{-1}
\\
&=\beta_{S\amalg T\backslash j,j}(0_{S\amalg T\backslash j}\oplus f_j)\alpha_{S\amalg T\backslash j,j}^{-1}
\\
&=f_{S\amalg T},
\end{align*}
where $\tau_{S,T}\colon \sum_{s\in S}a_s+\sum_{t\in T}a_t\to \sum_{t\in T}a_t+\sum_{s\in S}a_s$ is the symmetry isomorphism of the symmetric monoidal structure.
From this description of the morphisms of $\mathcal{F}^{\vee}_A[n]$ one can easily see that the top horizontal map of the square above, and hence all of its maps, are equivalences of categories. Thus the $\Gamma$-space underlying ${\KR}(A)$ is equivalent to the $K$-theory of $A$.

We show that ${\KR}$ is a special  $\mathbb{Z}/2$-$\Gamma$-space. For every group homomorphism $\sigma\colon \mathbb{Z}/2\to \Sigma_n$ we need to show that the map
\[
\mathcal{F}^{\vee}_A[n]\longrightarrow (\mathcal{F}^{\vee}_A[1])^{\times n},
\]
whose $j$-component is induced by the map $n_+\to 1_+$ that sends $j$ to $1$ and the rest to the basepoint, is a $\mathbb{Z}/2$-equivariant equivalence. Here the involution is induced by $\sigma\colon \mathbb{Z}/2\to \Sigma_n$ through the functoriality in $n$. The square above provides an equivalence of spectrally enriched categories $\mathcal{F}_A^{\vee}\to \mathcal{F}_A^{\vee}[1]$. This functor is in fact an isomorphism on mapping spectra, and it is therefore an equivariant equivalence. We show that the top horizontal arrow of
\[
\xymatrix{
(\mathcal{F}^{\vee}_A)^{\vee n}\ar[r]\ar[d]_{\sim}&\mathcal{F}^{\vee}_A[n]\ar[d]
\\
(\mathcal{F}^{\vee}_A)^{\times n}\ar[r]_\sim&(\mathcal{F}^{\vee}_A[1])^{\times n}
}
\]
defined as the restriction of $\mathcal{F}^{\times n}_A\to \mathcal{F}_A[n]$ is an equivariant equivalence, which will finish the proof. It has an equivariant inverse $\mathcal{F}^{\vee}_A[n]\to (\mathcal{F}^{\vee}_A)^{\vee n}$ that sends an object $\langle \underline{a},\alpha\rangle$ to $\underline{a}$ and a morphism $\{f_S\}$ to $f_{\{j\}}$.
\end{proof}


\subsection{Pairings in real algebraic $K$-theory}\label{secpairings}

Let $A$ and $B$ be ring spectra with anti-involution. Their smash product $A\wedge B$ is a ring spectrum with anti-involution, where the multiplication is defined componentwise and the anti-involution is diagonal. There is a pairing in the homotopy category of $\Z/2$-spectra
\[
{\KR}(A)\wedge {\KR}(B)\longrightarrow {\KR}(A\wedge B)
\]
defined in a manner analogous to the analogous pairing in $K$-theory. This map does not quite lift to the category of $\Z/2$-spectra because of the failure of thick realizations to commute with products strictly, but this is the only obstruction. The pairing is defined as follows. The standard formula for the Kronecker product of matrices restricts to an isomorphism
\[
\otimes \colon M^{\vee}_{n}(X)\wedge M^{\vee}_{k}(Y)\cong \bigvee_{n\times n\times k\times k}(X\wedge Y)\longrightarrow \bigvee_{nk\times nk}(X\wedge Y)=M^{\vee}_{nk}(X\wedge Y)
\]
(where $X$ and $Y$ are either pointed spaces or spectra) which is determined by the isomorphism $n\times n\times k\times k\cong nk\times nk$ that sends $(i,j,m,l)$ to $((i-1)n+m,(j-1)n+l)$. This isomorphism is moreover $\Z/2$-equivariant.

Given $\underline{a}=(a_1,\dots,a_n)$, we denote $\widehat{M}^{\vee}_{\underline{a}}(A):=\Omega^{\infty}_I(M^{\vee}_{a_1}(A)\vee\dots\vee M^{\vee}_{a_n}(A))$, a topological monoid with anti-involution. Given $\underline{a}=(a_1,\dots,a_n)$ and $\underline{b}=(b_1,\dots,b_k)$, we let $\underline{a}\cdot \underline{b}$ be the $nk$-sequence of non-negative integers
\[\underline{a}\cdot \underline{b}:=(a_1b_1,a_1b_2,\dots a_1b_k,a_2b_1,\dots,a_2b_k,\dots,a_nb_1,\dots, a_nb_k).\]
The pairing above induces $\Z/2$-equivariant maps
\[
\xymatrix@C=0pt@R=15pt{
B^{\sigma}(\widetilde{\langle\underline{a}\rangle}\times\widehat{M}^{\vee}_{\underline{a}}(A))\times B^{\sigma}(\widetilde{\langle\underline{b}\rangle}\times\widehat{M}^{\vee}_{\underline{b}}(B))
\ar@{-->}[dddd]
&
B^{\sigma}(\widetilde{\langle\underline{a}\rangle}\times\widetilde{\langle\underline{b}\rangle}\times\widehat{M}^{\vee}_{\underline{a}}(A)\times \widehat{M}^{\vee}_{\underline{b}}(B))
\ar[l]_-{\simeq}\ar[d]
\\
&
\hspace{-1.5cm}\displaystyle B^{\sigma}(\widetilde{\langle\underline{a}\rangle}\times\widetilde{\langle\underline{b}\rangle}\times\hocolim_{I\times I}\Omega^{i+j}(M^{\vee}_{\underline{a}}(A_i)\wedge M^{\vee}_{\underline{b}}(B_j)))\ar[d]_-\cong
\\
&
\hspace{-2cm} \displaystyle B^{\sigma}(\widetilde{\langle\underline{a}\rangle\times\langle\underline{b}\rangle}\times\hocolim_{I\times I}\Omega^{i+j}(\bigvee_{(h,l)\in n\times k}(M^{\vee}_{a_h}(A_i)\wedge 
M^{\vee}_{b_l}(B_j))))\ar[d]_-{\cong}^-{\otimes}
\\
&
\hspace{-1.5cm}\displaystyle B^{\sigma}(\widetilde{\langle\underline{a}\rangle\times\langle\underline{b}\rangle}\times\hocolim_{I\times I}\Omega^{i+j}(\bigvee_{(h,l)\in n\times k}(M^{\vee}_{a_hb_l}(A_i\wedge 
B_j))))\ar[d]^-{\otimes\times \iota}
\\
\hspace{-1cm}\displaystyle B^{\sigma}(\widetilde{\langle\underline{a}\cdot \underline{b}\rangle}\times\widehat{M}^{\vee}_{\underline{a}\cdot \underline{b}}(A\wedge 
B))\stackrel{+_\ast}{\longleftarrow}
&
\hspace{-1.2cm}
\displaystyle B^{\sigma}(\widetilde{\langle\underline{a}\cdot \underline{b}\rangle}\times\hocolim_{I\times I}\Omega^{i+j}(\bigvee_{(h,l)\in n\times k}(M^{\vee}_{a_hb_l}((A\wedge 
B)_{i+j}))))
}
\]
where $\otimes\times \iota$ is induced by the product of the map $\otimes \colon \langle\underline{a}\rangle\times\langle\underline{b}\rangle\to \langle\underline{a}\cdot\underline{b}\rangle$ defined by taking the Kronecker product of permutations (i.e. the Kronecker product of the associated permutation matrices), and of the canonical map $A_i\wedge B_j\to (A\wedge B)_{i+j}$. By restricting on the invertible components of $\widehat{M}^{\vee}_{\underline{a}}(A)$ and $\widehat{M}^{\vee}_{\underline{b}}(A)$ this induces equivariant maps
\[
{\KR}(A)_n\wedge {\KR}(B)_k\stackrel{\simeq}{\longleftarrow}Z_{n,k}\longrightarrow {\KR}(A\wedge B)_{nk}
\]
where $Z_{n,k}=\coprod_{\underline{a},\underline{b}}B^{\sigma}(\widetilde{\langle\underline{a}\rangle}\times\widetilde{\langle\underline{b}\rangle}\times\widehat{GL}^{\vee}_{\underline{a}}(A)\times \widehat{GL}^{\vee}_{\underline{b}}(B))$, and this diagram is natural in both variables $n,k\in\Gamma^{op}$. This therefore induces a pairing in the homotopy category of $\Z/2$-spectra, and pairings
\[
\otimes\colon{\KR}(A)^{\Z/2}\wedge{\KR}(B)^{\Z/2}\longrightarrow {\KR}(A\wedge B)^{\Z/2}
\]
\[
\otimes\colon\Phi^{\Z/2}{\KR}(A)\wedge \Phi^{\Z/2}{\KR}(B)\longrightarrow \Phi^{\Z/2}{\KR}(A\wedge B)
\]
in the homotopy category of spectra.
Let $A$ be a ring spectrum with anti-involution and $\pi$ a well-pointed topological group. The corresponding group-algebra is the ring spectrum
\[A[\pi]:=A\wedge \pi_+\]
with the anti-involution defined diagonally from the anti-involution of $A$ and the inversion map of $\pi$. This is an $A$-algebra via the map
\[
A\wedge A[\pi]=A\wedge A\wedge \pi_+\xrightarrow{\mu\wedge \id}A\wedge\pi_+=A[\pi]
\]
where $\mu$ denotes the multiplication of $A$.
If moreover $A$ is a commutative ring spectrum with anti-involution, that is a commutative  $\Z/2$-equivariant orthogonal ring spectrum, the map $\mu\wedge \id$ is a
morphism of ring spectra with anti-involution. 
Thus one can compose the pairings above for $B=A[\pi]$ with the induced map ${\KR}(A\wedge A[\pi])\to {\KR}(A[\pi])$.
The associativity of the Kronecker product of matrices then gives the following.

\begin{proposition}\label{multiplication}
Let $A$ be a commutative  $\Z/2$-equivariant orthogonal ring spectrum and $\pi$ a topological group. The graded Abelian groups $\pi_\ast{\KR}(A)^{\Z/2}$ and $\pi_\ast\Phi^{\Z/2}{\KR}(A)$ are graded rings. The graded Abelian group $\pi_\ast{\KR}(A[\pi])^{\Z/2}$ is a graded $\pi_\ast{\KR}(A)^{\Z/2}$-module, and $\pi_\ast\Phi^{\Z/2}{\KR}(A[\pi])$ is a graded $\pi_\ast\Phi^{\Z/2}{\KR}(A)$-module. \qed
\end{proposition}

\subsection{The Hermitian $K$-theory of a ring spectrum with anti-involution}

Let $A$ be a ring spectrum with anti-involution.

\begin{definition}
The free Hermitian $K$-theory space of $A$ is the fixed points space
\[
\KH(A):=\KR(A)^{\mathbb{Z}/2}\cong\Omega B\big(\coprod_n \big(B^{\sigma}\widehat{GL}^{\vee}_n(A)\big)^{\mathbb{Z}/2}\big).
\]
The free Hermitian $K$-theory spectrum of $A$ is the spectrum ${\KH}(A)$ associated to the fixed points $\Gamma$-space
\[{\KH}(A)_n:={\KR}(A)^{\mathbb{Z}/2}_n\cong \coprod_{\underline{a}}(B^{\sigma}(\widetilde{\langle\underline{a}\rangle}\times\widehat{GL}^{\vee}_{a_1,\dots,a_n}(A)))^{\mathbb{Z}/2}.\]
\end{definition}

\begin{remark}\label{naiveKR}
The spectrum associated to the $\Gamma$-space ${\KH}(A)$ is the na\"{i}ve fixed-points spectrum of the $\mathbb{Z}/2$-spectrum associated to the $\Z/2$-$\Gamma$-space ${\KR}(A)$. Since ${\KR}(A)$ is special as an equivariant $\Gamma$-space (see Proposition \ref{Gamma}), the canonical map of spectra ${\KH}(A)\to {\KR}(A)^{\mathbb{Z}/2}$ is a stable equivalence, where ${\KR}(A)^{\mathbb{Z}/2}$ is the genuine fixed-points spectrum of ${\KR}(A)$.
\end{remark}

\begin{definition}
The free genuine $L$-theory spectrum of the ring spectrum with anti-involution $A$ is the geometric fixed-points spectrum $\Lg(A):=\Phi^{\mathbb{Z}/2}{\KR}(A)$.
\end{definition}

We analyze the $\Gamma$-space ${\KH}(A)$ and we interpret it as the Segal construction of a symmetric monoidal category of Hermitian forms on $A$.
We recall from Proposition \ref{fixedB11} that the the fixed points space of $B^{\sigma}\widehat{GL}^{\vee}_{a_1,\dots,a_n}(A)$ is homeomorphic to the classifying space of a topological category $\sym\widehat{GL}^{\vee}_{a_1,\dots,a_n}(A)$. Its space of objects is the space of invertible components of the fixed points space 
\[
\widehat{M}^{\vee}_{a_1,\dots,a_n}(A)^{\mathbb{Z}/2}:=(\Omega^{\infty}_I(M^{\vee}_{a_1}(A)\vee\dots\vee M^{\vee}_{a_n}(A)))^{\mathbb{Z}/2},\]
which is equivalent to the infinite loop space of the fixed-points spectrum $M^{\vee}_{a_1}(A)^{\mathbb{Z}/2}\times\dots\times M^{\vee}_{a_1}(A)^{\mathbb{Z}/2}$ (see Remark \ref{infiniteloopG}).  A morphism $l\colon m\to n$ of $\sym\widehat{GL}^{\vee}_{a_1,\dots,a_n}(A)$  is a homotopy invertible element of $\widehat{M}^{\vee}_{a_1,\dots,a_n}(A)$ which satisfies $m=w(l)nl$, where $w$ denotes the involution on $\widehat{M}^{\vee}_{a_1,\dots,a_n}(A)$. Thus we think of ${\KH}(A)$ as the Segal construction of a symmetric monoidal category
\[\Herm_{A}=\coprod_{n}\sym\widehat{GL}^{\vee}_{n}(A)\]
of spectral Hermitian forms on $A$.

\begin{proposition}\label{compBF}
Suppose that $R$ is a simplicial ring with an anti-involution $w\colon R^{op}\rightarrow R$. There is a weak equivalence between ${\KH}(HR)$ and the connective cover of the Hermitian $K$-theory spectrum ${_1\widetilde{L}(R)}$ of \cite{BF}. In particular if $R$ is discrete this is equivalent to the connective Hermitian $K$-theory of free $R$-modules of \cite{Karoubi}, if $2\in R$ is invertible.

If $R$ is moreover commutative, this equivalence induces a ring isomorphism on homotopy groups with respect to the multiplication of \ref{multiplication}.
\end{proposition}

\begin{proof}
The inclusion of wedges into products defines a map of ring spectra with anti-involution
\[M^{\vee}_n(HR)=\bigvee_{n\times n}HR\longrightarrow \prod_{n\times n}HR\cong HM_n(R),\]
where $M_n(R)=\bigoplus_{n\times n}R$ is the ring of $n\times n$-matrices with entries in $R$.
On the underlying $\mathbb{Z}/2$-spectra, this is an inclusion of indexed wedges into indexed products and it is therefore a stable equivalence. On the level of $\Gamma$-categories this shows that the composite
\[
\mathcal{F}^{\vee}_{HR}[n]\longrightarrow \mathcal{F}^{}_{HR}[n]\longrightarrow H\mathcal{F}^{}_{R}[n]
\]
is an equivalence, where $\mathcal{F}_{R}[n]$ is Segal's construction of the symmetric monoidal category of free $R$-modules $(\mathcal{F}_R,\oplus)$, with the duality induced by conjugate transposition of matrices (we observe that the middle term $ \mathcal{F}^{}_{HR}[n]$ does not have a duality). 
At the level of $\Gamma$-spaces this induces an equivalence
\[(B^{\sigma}\Omega^{\infty}H\mathcal{F}^{\vee}_{HR}[n])^{\mathbb{Z}/2}\stackrel{\sim}{\longrightarrow}
(B^{\sigma}\Omega^{\infty}H\mathcal{F}^{}_{R}[n])^{\mathbb{Z}/2}\stackrel{\sim}{\longleftarrow}(B^{\sigma}\mathcal{F}^{}_{R}[n])^{\mathbb{Z}/2}\cong B\sym(\mathcal{F}_{R}[n]),
\]
that restricted to invertible components gives an equivalence
\[
{\KH}_n(R)\simeq B\sym(i\mathcal{F}_{R}[n]).
\]
Moreover there is a functor of $\Gamma$-categories  $(\sym i\mathcal{F}_{R})[n]\to\sym(i\mathcal{F}_{R}[n])$, and since both categories are equivalent to $\sym i\mathcal{F}_{R}^{\times n}$ it is an equivalence. Finally, $\sym i\mathcal{F}_{R}$ is the category of Hermitian forms over the simplicial ring $R$ of \cite{BF}.

When $R$ is commutative, both ring structures on ${\KH}(R)$ and ${\KH}(HR)$ are defined from the Kronecker product of matrices, and by inspection the equivalence above is multiplicative.
\end{proof}

We relate the Hermitian $K$-theory of ring spectra with the Hermitian $K$-theory of Hermitian Mackey functors defined in \S\ref{secHermMackey}.

\begin{proposition}\label{KRHM} Let $HL$ be a ring spectrum with anti-involution whose underlying $\mathbb{Z}/2$-spectrum is the Eilenberg-MacLane spectrum of a Mackey functor $L$.
There is a stable equivalence of $\Gamma$-spaces
\[
{\KH}(HL)\stackrel{\sim}{\longrightarrow}{\KH}(L)
\]
induced by the projection maps $\Omega^{\infty}_IHL\to L(\mathbb{Z}/2)$ and  $(\Omega^{\infty}_IHL)^{\mathbb{Z}/2}\to L(\ast)$  onto $\pi_0$, 
where the Hermitian structure of $L$ is induced by the multiplication of $HL$ as in Proposition \ref{pi0Hermitian}.

If moreover $HL$ is commutative, this equivalence is multiplicative with respect to the ring structures of \ref{multiMackey} and \ref{multiplication}.
\end{proposition}

\begin{proof}
We recall that since $\Omega^{\infty}_IHL$ is a topological monoid with anti-involution, there is an action
\[
\Omega^{\infty}_IHL\times (\Omega^{\infty}_IHL)^{\mathbb{Z}/2}\longrightarrow (\Omega^{\infty}_IHL)^{\mathbb{Z}/2}
\]
defined by sending $(m,n)$ to $mnw(m)$ where $w$ is the involution on $\Omega^{\infty}_IHL$. The Hermitian structure on $\underline{\pi}_0HL$ is defined by taking $\pi_0$ of this map. We also recall from Proposition \ref{pi0Hermitian} that $M^{\vee}_n(HL)$ is a model for the Eilenberg-MacLane spectrum of the Hermitian Mackey functor of matrices $M_n(L)$ of Definition \ref{MackeyMatrix}. Thus the projections onto $\pi_0$ define an equivalence of topological categories
\[
\coprod_{n}\sym \widehat{GL}^{\vee}_{n}(HL)\stackrel{\sim}{\longrightarrow}i\Herm_{L}
\]
onto the category of Hermitian forms on $M$ and isomorphisms.
At the level of $\Gamma$-spaces this gives an equivalence
\[{\KH}(HL)_n\cong
\coprod_{\underline{a}}B(\langle\underline{a}\rangle\times \sym\widehat{GL}^{\vee}_{a_1,\dots,a_n}(HL))\stackrel{\sim}{\longrightarrow}i\Herm_M[n]\]
onto the Segal $\Gamma$-category associated to $(i\Herm_L,\oplus)$, by the same argument of Proposition \ref{compBF}.
\end{proof}

\begin{proposition}\label{Lgeom} 
Let $R$ be a discrete ring with anti-involution. There is a natural isomorphism
\[
\Lqsub{\ast\geq 0}(R)[\tfrac{1}{2}]\cong \Lgsub{\ast}(R)[\tfrac{1}{2}]
\]
after inverting $2$, where $\Lqsub{\ast\geq 0}$ are the quadratic $L$-groups. Under this isomorphism, the splitting $\KH_\ast(R)[\tfrac 12]\cong (\K_\ast(R)[\frac 12])^{\Z/2}\oplus\Lgsub{\ast}(R)[\frac12]$ of the isotropy separation sequence agrees with the splitting $\KH_\ast(R)[\tfrac 12]\cong (\K_\ast(R)[\frac 12])^{\Z/2}\oplus\Lqsub{\ast\geq 0}(R)[\frac12]$ of \cite[Prop.6.2]{BF}.
\end{proposition}

\begin{proof} The isotropy separation sequence for the $\Z/2$-spectrum ${\KR}(R)$ splits away from $2$, giving short-exact sequences
\[
\xymatrix{
0\ar[r]&\pi_n(\K(R)[\frac 12])_{h\Z/2})\ar[r]^-{\overline{H}}& \KH_n(R)[\frac12]\ar[r]^-{q}&\Lgsub{n}(R)[\frac12]\ar[r]&0
}
\]
for every $n\geq 0$, where $\overline{H}$ is induced by the hyperbolic map $H\colon {\K}(R)\to {\KH}(R)$. Since $\pi_n(({\K}(R)[\frac 12])_{h\Z/2})$ is isomorphic to the orbits of the $\Z/2$-action on $\K_n(R)[\frac 12]$, the cokernel of the map $\overline{H}$ is isomorphic to the cokernel of $H\colon \K_n(R)\to \KH_n(R)$, and these are by definition the Witt groups $W_n(R)$. These agree with the $L$-groups after inverting $2$, see e.g. \cite[Thm 3.2.6]{jl}, or combine \cite[\S9]{RanickiIntro} and \cite[4.10]{Schlichting}.

%
Burghelea and Fiedorowicz show in \cite[Prop.6.2]{BF} that there is a natural isomorphism on homotopy groups
\[
\KH_\ast(R)[\tfrac{1}{2}]\cong(\KH_\ast(R)[\tfrac{1}{2}])^s\oplus(\KH_\ast(R)[\tfrac{1}{2}])^a\cong (\K_\ast(R)[\tfrac{1}{2}])^s\oplus W_{\ast\geq 0}(R)[\tfrac{1}{2}]
\]
where the superscripts $^s$ at $^a$ on $\KH_\ast$ denote respectively the $1$ and $(-1)$-eigenspaces of the $\mathbb{Z}/2$-action on $\KH$ defined by taking the opposite sign of the entries of the matrix of a Hermitian form. Here $(\K_\ast(R)[\frac{1}{2}])^s$ is the $1$-eigenspace, that is the fixed points, of the involution on homotopy groups induced by the involution on ${\KR}$.
To show that the splittings agree, we show that the composite isomorphism
\[
(\K_\ast(R)[\tfrac{1}{2}])^s\oplus(\KH_\ast(R)[\tfrac{1}{2}])^a\xrightarrow{\tfrac{H}{2}+\iota} \KH_\ast(R)[\tfrac{1}{2}]\xrightarrow{(r,q)}\pi_\ast({\KR}(R)[\tfrac{1}{2}])^{\mathbb{Z}/2}\oplus \Lgsub{\ast}(R)[\tfrac{1}{2}]
\]
is block-diagonal, where $r\colon {\KH}(R)[\frac 12]\to ({\K}(R)[\frac 12])^{h\Z/2}\simeq  ({\K}(R)[\frac 12])_{h\Z/2}$ is the map from fixed points to homotopy fixed points, which is the natural splitting of the isotropy separation sequence away from $2$.
 The first map is the restriction of half the hyperbolic map $H\colon {\K}\to {\KH}$ on the first summand, and the inclusion $\iota$ on the second summand. Thus we need to show that the composites $q\circ \frac{H}{2}$ and $R\circ\iota$ are null. The hyperbolic map is the transfer on the Mackey structure of ${\KR}_\ast$, and therefore $q\circ H=0$.
For the composite $R\circ \iota$, we observe that the restriction map $r\colon\KH_\ast(R)[\frac{1}{2}]\to \pi_\ast({\KR}(R)[\frac{1}{2}])^{\mathbb{Z}/2}$ is induced by the functor that sends a Hermitian form to its underlying module. It is therefore invariant under the involution on ${\KH}$, and it sends the $-1$-eigenspace $(\KH_\ast(R)[\frac{1}{2}])^a$ to zero.
\end{proof}

%
%

\subsection{The assembly map of real algebraic $K$-theory}\label{secass}

We define an assembly map for the real $K$-theory functor in the same spirit as Loday's definition of \cite{jl}, using the multiplicative pairing of \S\ref{secpairings}.

Let $A$ be a ring spectrum with anti-involution and $\pi$ a well-pointed topological group. The corresponding group-algebra is the ring spectrum
\[A[\pi]:=A\wedge \pi_+\]
with the anti-involution defined diagonally from the anti-involution of $A$ and the inversion map of $\pi$.

\begin{remark}\label{gpring}
Suppose that $R$ is a discrete ring with anti-involution $w$. Then the inclusion of indexed wedges into indexed products defines an equivalence of ring spectra with anti-involution
\[
(HR)[\pi]=HR\wedge \pi_+=\bigvee_\pi HR\stackrel{\simeq}{\longrightarrow}\bigoplus_\pi HR\cong H(R[\pi])
\]
where the anti-involution on the group-ring $R[\pi]$ sends $r\cdot g$ to $w(r)\cdot g^{-1}$, for all $r\in R$ and $g\in \pi$. More generally, the fixed-points spectrum of $A[\pi]$ decomposes as
\[
(A[\pi])^{\mathbb{Z}/2}\stackrel{\simeq}{\longrightarrow}(\bigoplus_{\pi}A)^{\mathbb{Z}/2}\cong (\bigoplus_{\pi^{\mathbb{Z}/2}}A^{\mathbb{Z}/2})\times(\bigoplus_{\pi^{free}/\mathbb{Z}/2}A),
\]
and an argument analogous to Proposition \ref{pi0Hermitian} shows that the Mackey functor of components $\underline{\pi}_0(A[\pi])$ is isomorphic to the group-Mackey functor $(\underline{\pi}_0A)[\pi]$ of Example \ref{groupMackey}.
\end{remark}


There is a morphism of $\mathbb{Z}/2$-$\Gamma$-spaces $\gamma\colon\mathbb{S}\wedge B^\sigma\pi_+\to {\KR}(\mathbb{S}[\pi])$, which is adjoint to the map of $\Z/2$-spaces
\[
B^\sigma\pi \hookrightarrow B^{\sigma} \widehat{GL}^{\vee}_{1}(\mathbb{S}[\pi]) \hookrightarrow \coprod_{n} B^{\sigma} \widehat{GL}^{\vee}_{n}(\mathbb{S}[\pi])={\KR}(\mathbb{S}[\pi])_1
\]
where the first map is induced by the canonical map $\pi\to \hocolim_{I}\Omega^i(S^i\wedge \pi_+)$ which includes at the object $i=0$.

\begin{definition}\label{defassKR} Let $A$ be a ring spectrum with anti-involution, and $\pi$ a topological group.
The assembly map of the real $K$-theory of $A[\pi]$ is the map in the homotopy category of $\mathbb{Z}/2$-spectra
\[
{\KR}(A)\wedge B^{\sigma}\pi_+\xrightarrow{\id\wedge\gamma}{\KR}(A)\wedge  {\KR}(\mathbb{S}[\pi])\stackrel{\otimes}{\longrightarrow} {\KR}(A\wedge \mathbb{S}[\pi])
\cong {\KR}(A[\pi]),
\]
where $\otimes$ is the pairing of \S\ref{secpairings}. If $A$ is commutative, this is a map of ${\KR}(A)$-modules for the module structures of \ref{multiplication}. 
\end{definition}

We now explain how to extract from this map an assembly map for Hermitian $K$-theory and $L$-theory by taking fixed-points spectra. We recall that there are natural transformations
\[
X^{\mathbb{Z}/2}\wedge K^{\mathbb{Z}/2}\stackrel{}{\longrightarrow} (X\wedge K)^{\mathbb{Z}/2}
\ \ \ \ \ \ \mbox{and}\ \ \ \ \ \ 
(\Phi^{\mathbb{Z}/2}X)\wedge K^{\mathbb{Z}/2}\stackrel{\simeq}{\longrightarrow} \Phi^{\mathbb{Z}/2}(X\wedge K)
\]
for any $\mathbb{Z}/2$-spectrum $X$ and pointed $\mathbb{Z}/2$-space $K$, where the first map is generally not an equivalence. Thus by applying fixed points and geometric fixed points to the $\KR$-assembly we obtain maps
\[
{\KH}(A)\wedge (B^{\sigma}\pi)^{\mathbb{Z}/2}_+\stackrel{}{\longrightarrow} ({\KR}(A)\wedge B^{\sigma}\pi_+)^{\mathbb{Z}/2}\longrightarrow {\KH}(A[\pi])
\]
\[
\Lg(A)\wedge (B^{\sigma}\pi)^{\mathbb{Z}/2}_+\stackrel{\simeq}{\longrightarrow} \Phi^{\mathbb{Z}/2}({\KR}(A)\wedge B^{\sigma}\pi_+)\longrightarrow \Lg(A[\pi]).
\]
We recall from Lemma \ref{funnymap} that there is an equivariant map $B\pi\to B^{\sigma}\pi$, which induces a summand inclusion $B\pi\to (B^{\sigma}\pi)^{\mathbb{Z}/2}$ on fixed points. By precomposing with this map we obtain the following.

\begin{definition}\label{defassKHL}
The assembly map in Hermitian $K$-theory and genuine $L$-theory of $A[\pi]$ are respectively the maps
\[
{\KH}(A)\wedge B\pi_+\stackrel{}{\longrightarrow}{\KH}(A)\wedge (B^{\sigma}\pi)^{\mathbb{Z}/2}_+\longrightarrow {\KH}(A[\pi]),
\]
\[
\Lg(A)\wedge B\pi_+\longrightarrow\Lg(A)\wedge (B^{\sigma}\pi)^{\mathbb{Z}/2}_+\longrightarrow\Lg(A[\pi]).
\]
If $A$ is commutative, these are maps of ${\KH}(A)$-modules and $\Lg(A)$-modules, respectively.
\end{definition}

\begin{proposition}\label{compfulass}
Let $R$ be a discrete ring with anti-involution. Then the assembly map for the Hermitian $K$-theory of $(HR)[\pi]$ described above agrees with the connective assembly for the Hermitian $K$-theory of $R[\pi]$ of \cite{jl}, under the equivalence of Proposition \ref{compBF} and the equivalence $(HR)[\pi]\to H(R[\pi])$ of Remark \ref{gpring}. It follows that the rationalized assembly maps of $\Lg(R[\pi])$ and $\Lq(R[\pi])$ agree as well.
\end{proposition}

\begin{proof}
By inspection, one sees that the pairing of \S\ref{secpairings} agrees with Loday's pairing under the equivalences of  \ref{compBF} and \ref{gpring}, and so do the maps $\mathbb{S}\wedge B\pi_+\to {\KH}(\mathbb{S}[\pi])$ and $\mathbb{S}\wedge B\pi_+\to {\KH}(\mathbb{Z}[\pi])$. More precisely, the diagram
\[
\xymatrix{
{\KH}(HR)\wedge B\pi_+\ar[r]^-{\id\wedge \gamma}\ar[d]_-{\simeq}
&
{\KH}(HR)\wedge{\KH}(\mathbb{S}[\pi])\ar[r]^-{\otimes}\ar[d]
&
{\KH}(HR\wedge\mathbb{S}[\pi])\ar[r]^-{\cong}\ar[d]
&
{\KH}((HR)[\pi])\ar[d]_-{\simeq}
\\
{\KH}(R)\wedge B\pi_+\ar[r]^-{\id\wedge \gamma_{\mathbb{Z}}}
&
{\KH}(R)\wedge{\KH}(\mathbb{Z}[\pi])\ar[r]^-{\otimes_{\mathbb{Z}}}&{\KH}(R\otimes_{\mathbb{Z}}\mathbb{Z}[\pi])\ar[r]^-{\cong}
&
{\KH}(R[\pi])
}
\]
commutes in the homotopy category of spectra, where $\gamma_{\mathbb{Z}}$ and $\otimes_{\Z}$ are from \cite[Chap. III]{jl}.
The top row is precisely the restriction of the assembly of Definition \ref{defassKR} to the $\Gamma$-space of fixed points, which compares to the assembly of  Definition \ref{defassKHL} by the diagram
\[
\xymatrix{
{\KH}(HR)\wedge B\pi_+\ar@{=}[r]\ar[d]_\sim
&
{\KH}(HR)\wedge B\pi_+\ar[r]\ar[d]
&
{\KH}(HR[\pi])\ar[d]^\sim
\\
{\KR}(HR)^{\mathbb{Z}/2}\wedge B\pi_+\ar[r]\ar[d]
&
({\KR}(HR)\wedge B\pi_+)^{\mathbb{Z}/2}\ar[r]\ar[d]
&
{\KR}(HR[\pi])^{\mathbb{Z}/2}
\\
{\KR}(HR)^{\mathbb{Z}/2}\wedge (B^{\sigma}\pi)^{\mathbb{Z}/2}_+\ar[r]
&
({\KR}(HR)\wedge B^{\sigma}\pi_+)^{\mathbb{Z}/2}\rlap{\ .}\ar[ur]
}
\]
The map from the first to the second row is the transformation from na\"{i}ve to genuine fixed-points spectra, and therefore the upper rectangle commutes.
The bottom left square commutes by naturality of the transformation $X^{\mathbb{Z}/2}\wedge K^{\mathbb{Z}/2}\to (X\wedge K)^{\mathbb{Z}/2}$. This shows that the assemblies in Hermitian $K$-theory agree.

By naturality of the isotropy separation sequence, the assembly map of $\Phi^{\mathbb{Z}/2}{\KR}(A)$ is the cofiber of the assembly maps for $\K(A)_{h\Z/2}$ and $\KH(A)$ under the hyperbolic map. It is then clear from \ref{Lgeom} that this agrees with the assembly map in Witt theory after inverting $2$. It is well believed by the experts that the assembly maps in Witt theory and quadratic $L$-theory agree away from $2$, but we were unfortunately unable to track down a reference (see e.g. \cite[8.2]{BF}, diagram $(5)$ and footnote $(^8)$). We prove that the (connective) assemblies agree rationally, at least for the ring of integers. Our argument is far from optimal, but is sufficient for our applications.

By periodicity, the Witt groups and the quadratic $L$-groups are rationally modules over the Laurent polynomial algebra $\mathbb{Q}[\beta,\beta^{-1}]$, where $\beta$ is of degree $4$ (see resp. \cite[4.10]{Karoubi} and \cite[Appendix B]{RanickiTopMan}). Let us choose isomorphisms of $\mathbb{Q}[\beta]$-modules $\phi\colon \Lqsub{\ast\geq 0}(\Z)\otimes \mathbb{Q}\cong\mathbb{Q}[\beta]\cong W_{\ast\geq0}(\Z)\otimes\mathbb{Q}$.
Then any choice of isomorphisms $\phi^{\pi}_i\colon \Lqsub{i}(\Z[\pi])\otimes \mathbb{Q}\cong W_{i}(\Z[\pi])\otimes \mathbb{Q}$ for $i=0,1,2,3$ determines an isomorphism of $\mathbb{Q}[\beta]$-modules  
\[\phi^\pi\colon \Lqsub{\ast\geq 0}(\Z[\pi])\otimes \mathbb{Q}\stackrel{\cong}{\longrightarrow} W_{\ast\geq 0}(\Z[\pi])\otimes\mathbb Q,\] 
which is given in degree $i=4k+l$, for $k>0$ and $l=0,1,2,3$, by $\phi^{\pi}_i=\beta^k\phi^{\pi}_l\beta^{-k}$. Since this is an isomorphism of $\mathbb{Q}[\beta]$-modules the right square of the diagram defining the assemblies
\[
\xymatrix@C=40pt{
\Lqsub{\ast\geq 0}(\Z)\otimes H_\ast(B\pi;\mathbb{Q})\ar[r]^-{\cong\otimes \gamma_L}\ar[d]^{\phi\otimes\id}
&
\mathbb{Q}[\beta]\otimes \Lqsub{\ast\geq 0}(\Z[\pi])\ar[r]\ar[d]^{\id\otimes\phi^\pi}&\Lqsub{\ast\geq 0}(\Z[\pi])\otimes \mathbb{Q}\ar[d]^{\phi_\pi}
\\
W_{\ast\geq 0}(\Z)\otimes H_\ast(B\pi;\mathbb{Q})\ar[r]^-{\cong\otimes \gamma_W}&\mathbb{Q}[\beta]\otimes W_{\ast\geq 0}(\Z[\pi])\ar[r]&W_{\ast\geq 0}(\Z[\pi])\otimes \mathbb{Q}
}
\]
commutes. It is therefore sufficient to show that $\phi^\pi\gamma_L=\gamma_W$. The map $\gamma_W$ is the composite $\mathbb{S}\wedge B\pi_+\to \KH(\Z[\pi])\to W(\Z[\pi])$, and the map $\gamma_L\colon \mathbb{S}\wedge B\pi_+\to L(\Z[\pi])$ is the ``pre-assembly'' of \cite[Appendix B]{RanickiTopMan}. Since $\pi$ is discrete, these maps are determined by the corresponding group homomorphisms $\delta_L\colon\pi\to L_1(\Z[\pi])$ and $\delta_W\colon \pi\to W_1(\Z[\pi])$. By the comparison of \cite[9.11]{RanickiIntro}, one can see that the isomorphism $\phi^{\pi}_1$ can be chosen so that $\phi^{\pi}_1\delta_L=\delta_W$.
Since the homotopy category of rational spectra is equivalent to graded $\mathbb{Q}$-vector spaces, the isomorphism $\phi_\pi$ can be realized as a zig-zag of maps of rational spectra. Since $\phi^{\pi}_1\delta_L=\delta_W$, the diagram
\[
\xymatrix@C=50pt@R=15pt{
B\pi_+\ar[r]\ar[dr]&\Omega^{\infty}(\Lq(\Z[\pi])\otimes\mathbb{Q})\ar[d]^{\Omega^\infty\phi^\pi}\\
&\Omega^{\infty}(W(\Z[\pi])\otimes\mathbb{Q})
}
\]
commutes in the homotopy category of spaces. It follows that $(\phi^\pi\otimes\mathbb{Q})\gamma_L=\gamma_W$ commutes in the homotopy category of spectra, which concludes the proof.
\end{proof}

The unit of the ring spectrum with anti-involution $A$ induces a map of $\mathbb{Z}/2$-spectra 
\[\eta\colon
\mathbb{S}\longrightarrow \KR(A).\]
It is adjoint to the map of $\mathbb{Z}/2$-spaces $
S^0\to \coprod_{k\geq 0}B^{\sigma}\widehat{GL}^{\vee}_{k}(A)$
that sends the basepoint of $S^0$ to the unique point in the component $k=0$, and the non-basepoint of $S^0$ to the point in the component $k=1$ determined by the $0$-simplex of $(\sd_e N^{\sigma}\widehat{GL}^{\vee}_{1}(A))_0=\widehat{GL}^{\vee}_{1}(A)$ defined by the unit map $S^{0}\to A_0$ of $A$.

\begin{definition}\label{defresassKR} The restricted assembly map of the real $K$-theory of $A[\pi]$ is the map of $\mathbb{Z}/2$-spectra
\[
\mathbb{S}\wedge B^{\sigma}\pi_+\xrightarrow{\eta\wedge \id}{\KR}(A)\wedge B^{\sigma}\pi_+\longrightarrow {\KR}(A[\pi]).
\]
\end{definition}

The geometric fixed points of the map $\eta\colon \mathbb{S}\to \KR(A)$ provide a map 
$\mathbb{S}\cong \Phi^{\mathbb{Z}/2}\mathbb{S}\to \Phi^{\mathbb{Z}/2}{\KR}(A)=\Lg(A)$, which immediately leads to the corresponding restricted assembly in genuine $L$-theory. In Hermitian $K$-theory however, the Tom Dieck splitting
provides a map 
\[
\mathbb{S}\vee\mathbb{S}\longrightarrow \mathbb{S}\vee\mathbb{S}_{h\mathbb{Z}/2}\stackrel{\simeq}{\longrightarrow}\mathbb{S}^{\mathbb{Z}/2}\longrightarrow {\KH}(A)
\]
from two summands of the sphere spectrum.

\begin{definition}\label{defresass}
The restricted assembly maps of the Hermitian $K$-theory and genuine $L$-theory of $A[\pi]$ are respectively the maps of spectra
\[
(\mathbb{S}\vee\mathbb{S})\wedge B\pi_+\longrightarrow {\KH}(A)\wedge B\pi_+\longrightarrow {\KH}(A[\pi])
\]
\[
\mathbb{S}\wedge B\pi_+\longrightarrow\Lg(A)\wedge B\pi_+\longrightarrow \Lg(A[\pi]).
\]
\end{definition}

The restricted assembly map of the Hermitian $K$-theory of $\mathbb{Z}$ is usually defined on homotopy groups by the composite 
\[ 
\mathcal{A}^{0}_{\mathbb{Z}[\pi]}\colon \KH_0(\mathbb{Z})\otimes\pi_\ast (\mathbb{S}\wedge B\pi_+)\!\hookrightarrow \KH_\ast(\mathbb{Z})\otimes\pi_\ast (\mathbb{S}\wedge B\pi_+)\rightarrow \pi_\ast({\KH}(\mathbb{Z})\wedge B\pi_+)\xrightarrow{\mathcal{A}_{\mathbb{Z}[\pi]}}\KH_\ast(\mathbb{Z}[\pi])
\]
where the first map includes ${\KH}_0(\mathbb{Z})$ into ${\KH}_\ast(\mathbb{Z})$ as the degree zero summand. Both $\pi_0(\mathbb{S}\vee\mathbb{S})$ and  ${\KH}_0(\mathbb{Z})$ are isomorphic to $\mathbb{Z}\oplus \mathbb{Z}$, and the unit map $\eta\colon\mathbb{S}\to {\KR}(\mathbb{Z})$ provides such an isomorphism. The following compares the resulting assemblies.

\begin{proposition}\label{compresass}
The map $\mathbb{S}\vee\mathbb{S}\to {\KH}(\mathbb{Z})$ sends the two generators in $\pi_0$ respectively to the hyperbolic form $\left(\begin{smallmatrix}0&1\\1&0\end{smallmatrix}\right)$ and to the unit form $\langle 1\rangle$.
It follows that the
restricted assembly 
\[(\mathbb{S}\vee\mathbb{S})\wedge B\pi_+\longrightarrow {\KH}(\mathbb{Z}[\pi])\]
agrees on homotopy groups with the restricted assembly $\mathcal{A}^{0}_{\mathbb{Z}[\pi]}$ upon identifying $\pi_0(\mathbb{S}\vee\mathbb{S})$ and ${\KH}_0(\mathbb{Z})$ by the isomorphism
\[
\pi_0(\mathbb{S}\vee\mathbb{S})\stackrel{\eta}{\longrightarrow}{\KH}_0(\mathbb{Z})\xrightarrow{
\left(\begin{smallmatrix}1&1\\0&1
\end{smallmatrix}\right)
}{\KH}_0(\mathbb{Z}).
\]
\end{proposition}

\begin{proof}
The isotropy separation sequences for the $\mathbb{Z}/2$-spectra $\mathbb{S}$ and ${\KR}(\mathbb{Z})$ give a commutative diagram
\[\xymatrix{
0\ar[r]&\pi_0\mathbb{S}_{h\mathbb{Z}/2}\ar[r]\ar[d]^{\cong}&\pi_0\mathbb{S}^{\mathbb{Z}/2}\ar[r]\ar[d]^{}&\pi_0\Phi^{\mathbb{Z}/2}\mathbb{S}\cong \pi_0\mathbb{S}\ar[r]\ar[d]^{\cong}\ar@/_1pc/[l]_-{}&0
\\
\dots\ar[r]&\pi_0\KR(\mathbb{Z})_{h\mathbb{Z}/2}\ar[r]&\pi_0\KH(\mathbb{Z})\ar[r]&\pi_0\Phi^{\mathbb{Z}/2}\KR(\mathbb{Z})\cong W_0(\mathbb{Z})\ar[r]&0
}
\]
with exact rows. The outer vertical maps are the units of the ring structures on $\K(\mathbb{Z})$ and $W(\mathbb{Z})$ respectively. The splitting of the upper sequence is the Tom Dieck splitting. It follows that the bottom sequence splits as well, and that the middle map is an isomorphism. The diagram
\[\xymatrix{
\pi_0\mathbb{S}\ar[d]\ar[r]^-{\cong}&\pi_0\mathbb{S}_{h\mathbb{Z}/2}\ar[r]\ar[d]^{\cong}&\pi_0\mathbb{S}^{\mathbb{Z}/2}\ar[d]^{\cong}
\\
\pi_0\K(\mathbb{Z})\ar[r]_-{\cong}&\pi_0\KR(\mathbb{Z})_{h\mathbb{Z}/2}\ar[r]&\pi_0\KH(\mathbb{Z})
}
\]
commutes, and the composite of the two lower maps takes the isomorphism class of a free $\mathbb{Z}$-module to its hyperbolic form. Moreover, the composite $\pi_0\mathbb{S} \to \pi_0\mathbb{S}^{\mathbb{Z}/2}\to \pi_0\KH(\mathbb{Z})$ is the unit of the ring structure of $\pi_0\KH(\mathbb{Z})$, and it takes the generator to $\langle 1\rangle$.
The identification of the restricted assemblies now follows from Proposition \ref{compfulass}.
\end{proof}


\section{The real trace map} \label{secfour}

\subsection{Real topological Hochschild homology}\label{secTHR}

We recollect some of the constructions of real topological Hochschild homology of \cite{IbLars} (see also \cite{Thesis}, \cite{Amalie}, \cite{THRmodels}). Let $A$ be a ring spectrum with anti-involution, possibly non-unital. The real topological Hochschild homology of $A$ is a genuine $\mathbb{Z}/2$-spectrum ${\THR}(A)$. It is determined by a strict $\mathbb{Z}/2$-action on the B\"{o}kstedt model for topological Hochschild homology ${\THH}(A)$ of the underlying ring spectrum. We recall its construction.

Let $I$ be the category of finite sets and injective maps. For any non-negative integer $k$ there is a functor $\Omega^{\bullet} A\colon I^{\times k+1}\to Sp$ that sends $\underline{i}=(i_0,i_1,\dots,i_k)$ to the spectrum
\[
\Omega^{i_0+i_1+\dots+i_k}(\mathbb{S}\wedge A_{i_0}\wedge A_{i_1}\wedge\dots\wedge A_{i_k}).
\]
Its homotopy colimit constitutes the $k$-simplices of a semi-simplicial orthogonal spectrum
\[
{\THH}_k(A):=\hocolim_{\underline{i}\in I^{\times k+1}}\Omega^{i_0+i_1+\dots+i_k}(\mathbb{S}\wedge A_{i_0}\wedge A_{i_1}\wedge\dots\wedge A_{i_k}),
\]
see e.g. \cite[Def.4.2.2.1]{DGM}. The involution on $I$ described in \S\ref{secRealSpectra} induces an involution on $I^{\times k+1}$, by sending $(i_0,i_1,\dots,i_k)$ to $(i_0,i_k,\dots,i_1)$ (it is the $k$-simplices of the dihedral Bar construction on $I$ with respect to the disjoint union).  The diagram $\Omega^\bullet A$ admits a $\mathbb{Z}/2$-structure in the sense of \cite[Def.1.1]{Gdiags}, defined by conjugating a loop with the maps
\[\xymatrix{
S^{i_0+i_1+\dots+i_k}\ar[rrr]^-{\chi_{i_0}\wedge\chi_{i_1}\wedge\dots\wedge \chi_k}&&&
S^{i_0+i_1+\dots+i_k}\ar[rr]^{\id_{i_0}\wedge \chi_{k}}&&S^{i_0+i_k+\dots+i_1}
}
\]
and
\[\xymatrix{
A_{i_0}\wedge A_{i_1}\wedge\dots\wedge A_{i_k}\ar[rr]^-{\chi_{i_0}\wedge\chi_{i_1}\wedge\dots\wedge \chi_k}&&
A_{i_0}\wedge A_{i_1}\wedge\dots\wedge A_{i_k}\ar[rr]^{\id_{i_0}\wedge \chi_{k}}&&A_{i_0}\wedge A_{i_k}\wedge\dots\wedge A_{i_1}\ar[d]^-{w\wedge\dots\wedge w}
\\
&&&&
A_{i_0}\wedge A_{i_k}\wedge\dots\wedge A_{i_1}
}
\]
where $\chi_j\in\Sigma_j$ is the permutation that reverses the order of $\{1,\dots,j\}$.
Thus the homotopy colimit ${\THH}_k(A)$ inherits a $\mathbb{Z}/2$-action which induces a semi-simplicial map ${\THH}_{\bullet}(A)^{op}\to {\THH}_\bullet(A)$. This therefore forms a real semi-simplicial spectrum ${\THH}_{\bullet}(A)$.

\begin{definition}[\cite{IbLars}]
The real topological Hochschild homology of $A$ is the $\mathbb{Z}/2$-spectrum ${\THR}(A)$ defined as the thick geometric realization of the semi-simplicial $\mathbb{Z}/2$-spectrum $\sd_e{\THH}_{\bullet}(A)$, where $\sd_e$ is Segal's subdivision.
\end{definition}

\begin{remark}
When $A$ is unital, levelwise well-pointed, and the unit $S^0\to A_0$ is an $h$-cofibration, the real semi-simplicial spectrum ${\THH}_{\bullet}(A)$ is in fact simplicial, and the map ${\THR}(A)\to |{\THH}_{\bullet}(A)|$ is a stable equivalence of $\mathbb{Z}/2$-spectra. In order to consider the circle action induced by the cyclic structure one should work under this extra assumptions with the thin realization.
\end{remark}

\begin{lemma}\label{fibrantTHR}
Suppose that for every $n\geq 0$ the space $A_{n}$ is $(n-1)$-connected, and that the fixed points space $A^{\mathbb{Z}/2}_{n\rho}$ is $(n-1)$-connected, where $\rho$ is the regular representation of $\mathbb{Z}/2$.
Then the $\mathbb{Z}/2$-spectrum ${\THR}(A)$ is an equivariant $\Omega$-spectrum. In particular the map
\[\|\sd_e\hocolim_{\underline{i}\in I^{\times k+1}}\Omega^{i_0+i_1+\dots+i_k}(A_{i_0}\wedge A_{i_1}\wedge\dots\wedge A_{i_k})\|\longrightarrow \Omega^{\infty\mathbb{Z}/2}{\THR}(A)
\]
is an equivalence, where $\Omega^{\infty\mathbb{Z}/2}$ denotes the genuine equivariant infinite loop space functor.
\end{lemma}

\begin{example}
\begin{enumerate}[i)]
\item Any suspension spectrum satisfies the hypotheses of the lemma. Indeed $(\mathbb{S}\wedge X)_n=S^n\wedge X$ is $(n+\conn X)$-connected non-equivariantly, and $\conn X\geq -1$. Similarly, the fixed points $(\mathbb{S}\wedge X)^{\mathbb{Z}/2}_{n\rho}=S^{n}\wedge X^{\mathbb{Z}/2}$ are $(n+\conn X^{\mathbb{Z}/2})$-connected.
\item Eilenberg-MacLane spectra of Abelian groups with $\mathbb{Z}/2$-action satisfy this condition as well, see e.g. \cite[Prop. A.1.1]{Gcalc}.
\end{enumerate}
\end{example}

\begin{proof}[Proof of \ref{fibrantTHR}]
We need to show that for every $\mathbb{Z}/2$-representation $V$, the adjoint structure map
\[
\|\hocolim_{\underline{i}\in I^{\times 2k+2}}\Omega^{\underline{i}}(S^V\wedge A_{i_0}\wedge A_{i_1}\wedge\dots\wedge A_{i_{2k+1}})\|\stackrel{\sigma}{\longrightarrow} \Omega^{\rho}\|\hocolim_{\underline{i}\in I^{\times 2k+2}}\Omega^{\underline{i}}(S^\rho\wedge S^V\wedge A_{i_0}\wedge A_{i_1}\wedge\dots\wedge A_{i_{2k+1}})\|
\]
is an equivalence. It is shown in \cite{THRmodels} using a semi-stability argument that there is an equivalence
\[\hocolim_{n\in\mathbb{N}}\Omega^{n\rho\otimes(k\rho+2)}(S^\rho\wedge S^V\wedge A_{n\rho}\wedge A_{n\rho}^{\wedge 2k+1})\stackrel{\simeq}{\longrightarrow}\hocolim_{\underline{i}\in I^{\times 2k+2}}\Omega^{\underline{i}}(S^\rho\wedge S^V\wedge A_{i_0}\wedge A_{i_1}\wedge\dots\wedge A_{i_{2k+1}})\]
where the involution on $2k+1$ reverses the order. The source of this map is equivariantly connected for every $k$, by our connectivity assumption. Indeed, it is non-equivariantly
\begin{align*}
(2+\dim V)+(2k+2)(\conn A_{n\rho})+2k+1-\dim (n\rho\otimes(k\rho+2))=\\
(2+\dim V)+(2k+2)(2n-1)+2k+1-2n(2k+2)=\\
(1+\dim V)+(2k+2)(2n)-2n(2k+2)=(1+\dim V)
\end{align*}
connected, and its connectivity on fixed points is the minimum between $1+\dim V$ and
\begin{align*}
\dim(\rho+V)^{\mathbb{Z}/2}+\conn(A_{n\rho}^{\mathbb{Z}/2}\wedge A_{n\rho}^{\wedge k}\wedge A_{n\rho}^{\mathbb{Z}/2})-\dim (2n\rho+nk\rho\otimes\rho)^{\mathbb{Z}/2}=
\\
1+\dim V^{\mathbb{Z}/2}+(n-1)+k(2n-1)+(n-1)+k+1-(2n+2nk)
=\dim V^{\mathbb{Z}/2}\geq 0.
\end{align*}
By \cite[Lemma 2.4]{Wittvect} we can therefore commute realization and loops, and the map $\sigma$ above is equivalent to the realization of the semi-simplicial map
\[
\hocolim_{\underline{i}\in I^{\times 2k+2}}\Omega^{\underline{i}}(S^V\wedge A_{i_0}\wedge A_{i_1}\wedge\dots\wedge A_{i_{2k+1}})\stackrel{\sigma_k}{\longrightarrow} \Omega^{\rho}\hocolim_{\underline{i}\in I^{\times 2k+2}}\Omega^{\underline{i}}(S^\rho\wedge S^V\wedge A_{i_0}\wedge A_{i_1}\wedge\dots\wedge A_{i_{2k+1}}).
\]
We show that $\sigma_k$ is an equivariant equivalence for all $k\geq 0$. Again by \cite{THRmodels} this map is equivalent to the map
\[
\hocolim_{n\in\mathbb{N}}\Omega^{n\rho\otimes(k\rho+2)}(S^V\wedge A_{n\rho}\wedge A_{n_\rho}^{\wedge 2k+1})\stackrel{}{\longrightarrow}\hocolim_{n\in\mathbb{N}}\Omega^{n\rho\otimes(k\rho+2)}\Omega^\rho(S^{\rho}\wedge S^V\wedge A_{n\rho}\wedge A_{n_\rho}^{\wedge 2k+1}),
\]
which is the homotopy colimit of $\Omega^{n\rho\otimes(k\rho+2)}$ applied to the unit $\eta_n\colon X_n\to \Omega^{\rho}(S^\rho\wedge X_n)$ of the loop-suspension adjunction, where $X_n:=S^V\wedge A_{n\rho}\wedge A_{n_\rho}^{\wedge 2k+1}$. By the equivariant Freudenthal suspension Theorem $\eta_n$ is non-equivariantly roughly
\[
2\conn X_n=2(\dim V+(2k+2)(2n-1)+2k+1)=2(\dim V+4n(k+1)-1)
\]
connected, and its connectivity on fixed points is roughly the minimum of $\conn X_n=\dim V+4n(k+1)-1$ and
\begin{align*}
2\conn X^{\mathbb{Z}/2}_n=
2(\dim V^{\mathbb{Z}/2}+2(n-1)+k(2n-1)+k+1).
\end{align*}
Thus on fixed points $\eta_n$ is roughly $(\dim V+4n(k+1))$-connected. It follows that $\Omega^{n\rho\otimes(k\rho+2)}\eta_n$ is non-equivariantly approximately
\begin{align*}
c_n:=2(\dim V+4n(k+1))-\dim(n\rho\otimes(k\rho+2))=\\
2\dim V+8n(k+1)-2n(2k+2)=2\dim V+4n(k+1)
\end{align*}
connected.
Its connectivity on fixed points is the minimum of $2\dim V+4n(k+1)$ and
\[
(\dim V+4n(k+1))-\dim(n\rho\otimes(k\rho+2))^{\mathbb{Z}/2}\!=\!(\dim V+4n(k+1))-(2n+2nk)=\!\dim V+2n(k+1),
\]
which is $d_n:=\dim V+2n(k+1)$. Since both $c_n$ and $d_n$ diverge with $n$ for every $k\leq 0$, the maps $\Omega^{n\rho\otimes(k\rho+2)}\eta_n$ induce an equivalence on homotopy colimits.
\end{proof}

\begin{remark}\label{THRGamma}
Under the connectivity assumptions of Lemma \ref{fibrantTHR} the $\mathbb{Z}/2$-spectrum ${\THR}(A)$ arises as the $\mathbb{Z}/2$-spectrum of a  $\mathbb{Z}/2$-$\Gamma$-space with value at the pointed set $n_+=\{+,1,\dots,n\}$ the $\mathbb{Z}/2$-space
\[
{\THR}(A)_n:=\|\sd_e\hocolim_{\underline{i}\in I^{\times k+1}}\Omega^{i_0+i_1+\dots+i_k}(A_{i_0}\wedge A_{i_1}\wedge\dots\wedge A_{i_k}\wedge{n_+})\|.
\]
Indeed, the value of the associated spectrum at a sphere $S^n$ is the geometric realization
\[{\THR}(A)_{S^n}:=|[p]\mapsto {\THR}(A)_{S_{p}^n}|\cong \|\sd_e \hocolim_{\underline{i}\in I^{\times k+1}}|\Omega^{i_0+i_1+\dots+i_k}(A_{i_0}\wedge A_{i_1}\wedge\dots\wedge A_{i_k}\wedge{S_{p}^n})|\|,\]
and under our connectivity assumptions the canonical map
\[
|\Omega^{i_0+i_1+\dots+i_k}(A_{i_0}\wedge A_{i_1}\wedge\dots\wedge A_{i_k}\wedge{S_{p}^n})|\stackrel{\sim}{\longrightarrow}\Omega^{i_0+i_1+\dots+i_k}(A_{i_0}\wedge A_{i_1}\wedge\dots\wedge A_{i_k}\wedge{|S_{p}^n|})
\]
is an equivariant equivalence with respect to the action of the stabilizer group of $(i_0,\dots,i_k)\in I^{\times k+1}$ (see \cite[Lemma 2.4]{Wittvect}). It follows from \cite[Cor.2.22]{Gdiags} that the map on homotopy colimits is an equivariant equivalence.
\end{remark}

\begin{lemma}\label{THRrat}
The real topological Hochschild homology functor ${\THR}$ commutes with rationalizations on ring spectra with anti-involution with flat underlying orthogonal $\mathbb{Z}/2$-spectrum.
\end{lemma}

\begin{proof}
Under this cofibrancy condition the spectrum ${\THR}(A)$ is naturally equivalent to the dihedral Bar construction of $A$ with respect to the smash product. This result is a generalization of Theorem \cite[4.2.8]{Shipley} and \cite{PatchkSagave}, and a proof can be found in \cite{THRmodels}.

Let $\mathbb{S}_{\mathbb{Q}}$ be a flat model for the rational $\mathbb{Z}/2$-equivariant sphere spectrum. We notice that if $K_+$ is any finite pointed $\mathbb{Z}/2$-set, the map
\[
\mathbb{S}_{\mathbb{Q}}\cong \mathbb{S}_{\mathbb{Q}}\wedge\bigwedge_K \mathbb{S}\stackrel{\sim}{\longrightarrow}\bigwedge_{K_+}\mathbb{S}_{\mathbb{Q}}
\]
given by the $K$-fold smash product of the unit maps of $\mathbb{S}_{\mathbb{Q}}$ smashed with $\mathbb{S}_{\mathbb{Q}}$ is an equivalence. Non-equivariantly this is clear since $\mathbb{S}_{\mathbb{Q}}\simeq H\mathbb{Q}$ is idempotent. On geometric fixed points this is the map
\[
\Phi^{\mathbb{Z}/2}\mathbb{S}_{\mathbb{Q}}\wedge\bigwedge_{[k]\in K/(\mathbb{Z}/2)}\Phi^{(\mathbb{Z}/2)_k}\mathbb{S}\cong \Phi^{\mathbb{Z}/2}\mathbb{S}_{\mathbb{Q}}\wedge\bigwedge_{[k]\in K/(\mathbb{Z}/2)} \mathbb{S}\stackrel{}{\longrightarrow}\Phi^{\mathbb{Z}/2}\mathbb{S}_{\mathbb{Q}}\wedge\bigwedge_{K/(\mathbb{Z}/2)}\Phi^{(\mathbb{Z}/2)_k}\mathbb{S}_{\mathbb{Q}}
\]
which is the smash of the identity with the $K/(\mathbb{Z}/2)$-fold smash of the unit maps of $\Phi^{(\mathbb{Z}/2)_k}\mathbb{S}_{\mathbb{Q}}$, where $(\mathbb{Z}/2)_k$ is the stabilizer group of $k\in K$. Since the geometric fixed points $\Phi^{\mathbb{Z}/2}\mathbb{S}_{\mathbb{Q}}$ are also equivalent to $H\mathbb{Q}$ they are idempotent, and the map is an equivalence.
Thus we have constructed natural equivalences
\[
{\THR}(A)\wedge \mathbb{S}_{\mathbb{Q}}\simeq |A^{\wedge \bullet+1}\wedge \mathbb{S}_{\mathbb{Q}}|\stackrel{\sim}{\longrightarrow} |A^{\wedge \bullet+1}\wedge \mathbb{S}_{\mathbb{Q}}^{\wedge \bullet+1}|\cong |(A\wedge \mathbb{S}_{\mathbb{Q}})^{\wedge \bullet+1}|\simeq {\THR}(A\wedge \mathbb{S}_{\mathbb{Q}}).
\]
\end{proof}

The real topological Hochschild homology spectrum also supports an assembly map. Given a ring spectrum with anti-involution $A$ and a well-pointed topological group $\pi$, it is a map
\[
{\THR}(A)\wedge B^{di}\pi_+\longrightarrow {\THR}(A[\pi])
\]
which is defined as the geometric realization of the map
\[
(\hocolim_{\underline{i}\in I^{\times k+1}}\Omega^{\underline{i}}(A_{i_0}\wedge \dots\wedge A_{i_{k}}\wedge{n_+}))\wedge \pi^{\times k+1}_+\longrightarrow\hocolim_{\underline{i}\in I^{\times k+1}}\Omega^{\underline{i}}(A_{i_0}\wedge \pi_+\wedge \dots\wedge A_{i_{k}}\wedge\pi_+\wedge{n_+})
\]
that commutes the smash-product with the homotopy colimit and the loops. It is shown in \cite{THRmodels} that this map is in fact an equivalence. When $A=\mathbb{S}$ is the sphere, there is a unit map $\mathbb{S}\to {\THR}(\mathbb{S})$ which is defined by the map into the homotopy colimit from the object $\underline{i}=0$.

\begin{proposition}[\cite{Amalie}]\label{THRass}
The composite $\mathbb{S}\wedge B^{di}\pi_+\to {\THR}(\mathbb{S})\wedge B^{di}\pi_+\to {\THR}(\mathbb{S}[\pi])$ is an equivalence.
\end{proposition}

\subsection{The definition of the real trace map}\label{sectrace}

We adapt the construction of the trace of \cite{BHM} to define a natural map of $\mathbb{Z}/2$-$\Gamma$-spaces
\[
\tr\colon {\KR}(A)\longrightarrow {\THR}(A)
\]
for every ring spectrum with anti-involution $A$ that satisfies the connectivity hypothesis of Lemma \ref{fibrantTHR}. The $\mathbb{Z}/2$-$\Gamma$-space ${\THR}(A)$ models the real topological Hochschild spectrum of $A$ (see Remark \ref{THRGamma}). At an object $n_+\in\Gamma^{op}$ the trace is defined as the composite
\[
\xymatrix@R=12pt{
\displaystyle {\KR}(A)_n=\coprod_{\underline{a}}B^{\sigma}(\widetilde{\langle \underline{a}\rangle}\times \widehat{GL}_{a_1,\dots,a_{n}}^\vee(A)) \ar[r]^-{c}
&
\displaystyle\coprod_{\underline{a}}\Lambda^{\sigma}\|N^{\sigma}(\widetilde{\langle \underline{a}\rangle}\times \widehat{GL}_{a_1,\dots,a_{n}}^\vee(A))\|
\\
\displaystyle \coprod_{\underline{a}}B_{\wedge}^{di}(\widetilde{\langle \underline{a}\rangle}\times \Omega^{\infty}_I(M^{\vee}_{a_1}(A)\vee\dots\vee M^{\vee}_{a_{n}}(A)))
&
\displaystyle \coprod_{\underline{a}}B^{di}(\widetilde{\langle \underline{a}\rangle}\times \widehat{GL}_{a_1,\dots,a_{n}}^\vee(A)) \ar[l]\ar[u(.7)]^-{\simeq} 
\\
\displaystyle \coprod_{ \underline{a}}(B^{di}\widetilde{\langle \underline{a}\rangle})\times\THR(M^{\vee}_{a_1}(A)\vee\dots\vee M^{\vee}_{a_{n}}(A))\ar@{<-}[u(.7)]\ar[r]
&{\THR}(A)_n \rlap{\ .}
}
\]
All the maps except for the last one leave the $\langle \underline{a}\rangle$-coordinate untouched. The first map is the composition of the equivalence $B^{\sigma}(\widetilde{\langle \underline{a}\rangle}\times \widehat{GL}_{a_1,\dots,a_{n}}^\vee(A))\to \|N^{\sigma}(\widetilde{\langle \underline{a}\rangle}\times \widehat{GL}_{a_1,\dots,a_{n}}^\vee(A))\|$ of Lemma \ref{lemmathicksub} and the inclusion of constant loops. We recall that $\Lambda^{\sigma}=\Map(S^{\sigma},-)$ is the free loop space with respect to the sign representation.
The second map is the map of Lemma \ref{lemmafreeloop}, and it is an equivalence because $\widehat{GL}_{a_1,\dots,a_{n}}^\vee(A)$ is quasi-unital and group-like, by Proposition \ref{GLq1gpl}. The third map includes the invertible components and projects the products onto the smash products, where $B^{di}_\wedge$ denotes the dihedral Bar construction with respect to the smash product of spaces. The fourth map commutes the smash products and the loops. The fifth map projects off the $\widetilde{\langle\underline{a}\rangle}$-component, and on the $\THR$ factor it is induced by the maps of spaces
\[
(M^{\vee}_{a_1}(A_{i_0})\vee\dots\vee M^{\vee}_{a_{n}}(A_{i_0}))\wedge\dots\wedge (M^{\vee}_{a_1}(A_{i_k})\vee\dots\vee M^{\vee}_{a_{n}}(A_{i_k}))\longrightarrow A_{i_0}\wedge\dots\wedge A_{i_k}\wedge n_+
\]
defined as follows. An element of $M^{\vee}_{a_1}(A_{i})\vee\dots\vee M^{\vee}_{a_{n}}(A_{i})$ is an integer $1\leq j\leq n$, a pair $(c,d)\in a_j\times a_j$, and an element $x\in A_{i}$. The map above is then defined by
\[
(j_0,c_0,d_0,x_0)\wedge \dots\wedge  (j_k,c_k,d_k,x_k)\longmapsto \!\left\{
\begin{array}{ll}
x_0\wedge \dots\wedge  x_k\wedge j_0& \mbox{if}\ j_0=\dots=j_k\ \mbox{and}\\
& d_0=c_1, d_1=c_2,\dots, d_{k-1}=c_k,
\\
& d_k=c_0
\\
\\
\ast&\mbox{otherwise}.
\end{array}
\right.
\]
This map remembers the entries of a string of matrices precisely when they are all composable, and it sends the rest to the basepoint. The underlying map is analogous to the trace map of \cite[\S1.6.17]{ringfctrs}, and it is a weak homotopy inverse for the map induced by the inclusion $A\to M^{\vee}_a(A)$ into the $1\times 1$-component. Although we won't use this here, it is also an equivariant equivalence (see \cite{THRmodels}).

Let us denote by $\tr^{cy}$ the composite
\[
\tr^{cy}\colon {\KR}^{cy}(A):=\coprod_{\underline{a}}B^{di}(\widetilde{\langle \underline{a}\rangle}\times \widehat{GL}_{a_1,\dots,a_{n}}^\vee(A))\longrightarrow {\THR}(A).
\]
All the spaces above extend to $\mathbb{Z}/2$-$\Gamma$-spaces by a construction similar to the definition of the $\Gamma$-structure on ${\KR}$. It is immediate to verify that the map $c$ and the upper-pointing equivalence are maps of $\mathbb{Z}/2$-$\Gamma$-spaces. We verify that $\tr^{cy}$ is compatible with the $\Gamma$-structure.

\begin{proposition}
The map $\tr^{cy}\colon  {\KR}^{cy}(A)\to  {\THR}(A)$ is a well-defined map of $\mathbb{Z}/2$-$\Gamma$-spaces.
\end{proposition}

\begin{proof}
Let $f\colon n_+\to k_+$ be a pointed map. We need to verify that for every collection of non-negative integers $\underline{a}=(a_1,\dots,a_n)$ the square
\[
\xymatrix{
B^{di}(\widetilde{\langle \underline{a}\rangle}\times \widehat{GL}_{a_1,\dots,a_{n}}^\vee(A))\ar[r]^-{\tr^{cy}}\ar[d]^{f_\ast}
& 
\displaystyle\|\hocolim_{\underline{i}\in I^{\times p+1}}\Omega^{i_0+i_1+\dots+i_p}(A_{i_0}\wedge A_{i_1}\wedge\dots\wedge A_{i_p}\wedge n_+)\|\ar[d]^{f_\ast}
\\
B^{di}(\widetilde{\langle f_\ast\underline{a}\rangle}\times \widehat{GL}_{f_\ast\underline{a}}^\vee(A))\ar[r]^-{\tr^{cy}}
& 
\displaystyle\|\hocolim_{\underline{i}\in I^{\times p+1}}\Omega^{i_0+i_1+\dots+i_p}(A_{i_0}\wedge A_{i_1}\wedge\dots\wedge A_{i_p}\wedge k_+)\|
}
\]
commutes. We prove that this diagram commutes in simplicial degree $p=1$, the argument for higher $p$ is similar. A $1$-simplex of the upper left corner consists of two pairs of families of permutations $(\beta,\alpha)$ and $(\alpha,\beta)$, where $\alpha,\beta\in \langle\underline{a}\rangle$, and a pair of elements $x,y\in \Omega^{\infty}_I(M_{a_1}^{\vee}(A)\vee\dots\vee M^{\vee}_{a_{n}}(A))$ belonging to an invertible component. For a fixed pair $(\alpha,\beta)$, we need to show that the square
\[
\xymatrix{
(M_{a_1}^{\vee}(A_{i_0})\vee\dots\vee M^{\vee}_{a_{n}}(A_{i_0}))\wedge(M_{a_1}^{\vee}(A_{i_1})\vee\dots\vee M^{\vee}_{a_{n}}(A_{i_1}))\ar[r]\ar[d]_{(\vee_{j=1}^k(\beta,\alpha)_j)\wedge(\vee_{j=1}^k(\alpha,\beta)_j)}
&
A_{i_0}\wedge A_{i_1}\wedge n_+\ar[d]^{\id\wedge\id\wedge f}
\\
(M_{a_1}^{\vee}(A_{i_0})\vee\dots\vee M^{\vee}_{a_{n}}(A_{i_0}))\wedge(M_{a_1}^{\vee}(A_{i_1})\vee\dots\vee M^{\vee}_{a_{n}}(A_{i_1}))\ar[r]
&
A_{i_0}\wedge A_{i_1}\wedge k_+
}
\]
commutes, where $(\alpha,\beta)_j$ are the maps defined in \S\ref{secdeloopings} and the horizontal maps are defined at the beginning of the section. The upper composite takes $(j_0,(c_0,d_0),x_0)\wedge (j_1,(c_1,d_1),x_1)$, where $1\leq j_l\leq n$, $(c_l,d_l)\in a_{j_l}\times a_{j_l}$ and $x\in A_{i_l}$ for $l=0,1$, to
\[
x_0\wedge  x_1\wedge f(j_0)\ \ \  \mbox{if}\ \ \ \ j_0=j_1\ \ \  \mbox{and}\ \ \   d_0=c_1, \ d_1=c_0,
\]
and to the basepoint otherwise. The lower composite takes it to
\[
\begin{array}{ll}x_0\wedge  x_1\wedge f(j_0)\ \ \ \mbox{if}\ f(j_0)=f(j_1)\ \ \ \mbox{and} &
\beta_{f^{-1}f(j_1)\backslash j_0,j_0}(\iota_0 d_0)=\beta_{f^{-1}f( j_1)\backslash j_1,j_1}(\iota_1 c_1),\\
&\alpha_{f^{-1}f(j_1)\backslash j_1,j_1}(\iota_1 d_1)
=\alpha_{f^{-1}f( j_0)\backslash j_0,j_0}(\iota_0 c_0)
,\end{array}
\]
where $\iota_0\colon a_{j_0}\to \amalg_{i\in f^{-1}f(j_0)}a_i$ is the inclusion, and similarly for $\iota_1$. We need to show that these conditions are equivalent. Clearly the first condition implies the second one.

Suppose that the second condition holds, and set $i:=f(j_0)=f(j_1)$.
By construction, the family of permutations $\alpha$ satisfies the condition
\[
\alpha_{(f^{-1}i)\backslash j_1,j_1}\circ (\alpha_{j_0,(f^{-1}i)\backslash \{j_0,j_1\}}\amalg \id_{a_{j_1}})
=\alpha_{(f^{-1}i)\backslash j_0,j_0}\circ (\id_{a_{j_0}}\amalg \alpha_{
(f^{-1}i)\backslash \{j_0,j_1\},j_1
}).
\]
By evaluating this expression at $\iota_0c_0$ we obtain that
\[
\alpha_{(f^{-1}i)\backslash j_1,j_1}\circ (\alpha_{j_0,(f^{-1}i)\backslash \{j_0,j_1\}}\amalg \id_{a_{j_1}})(\iota_0c_0)
=\alpha_{(f^{-1}i)\backslash j_0,j_0}(\iota_0c_0)=\alpha_{(f^{-1}i)\backslash j_1,j_1}(\iota_1d_1).
\]
Since $\alpha_{(f^{-1}i)\backslash j_1,j_1}$ is invertible we must have that
\[
(\alpha_{j_0,(f^{-1}i)\backslash \{j_0,j_1\}}\amalg \id_{a_{j_1}})(\iota_0c_0)
=\iota_1d_1,
\]
but since the left-hand map is the identity on $a_{j_1}$ and $\iota_1$ includes in $a_{j_1}$ we must have that $j_0=j_1$ and $c_0=d_1$. A similar argument shows that $d_0=c_1$.
\end{proof}


\subsection{The trace splits the restricted assembly map}\label{secsplitting}

Let $A$ be a ring spectrum with anti-involution and $\pi$ a well-pointed topological group. We recall that the restricted assembly map of ${\KR}$ is a map of $\mathbb{Z}/2$-spectra
\[
\mathcal{A}^0\colon \mathbb{S}\wedge B^{\sigma}\pi_+\xrightarrow{\eta\wedge\id}{\KR}(A)\wedge B^{\sigma}\pi_+ \longrightarrow {\KR}(A[\pi]),
\]
see Definition \ref{defresassKR}. We let $p\colon B^{di}\pi\longrightarrow B^{\sigma}\pi$ denote the projection.

\begin{theorem}\label{asssplitKR}
The restricted assembly map for the sphere spectrum $\mathcal{A}^0\colon \mathbb{S}\wedge B^{\sigma}\pi_+\to {\KR}(\mathbb{S}[\pi])$ is a split monomorphism in the homotopy category of $\mathbb{Z}/2$-spectra. A natural retraction is provided my the map
\[
{\KR}(\mathbb{S}[\pi])\stackrel{\tr}{\longrightarrow}\THR(\mathbb{S}[\pi])\stackrel{\simeq}{\longleftarrow}\mathbb{S}\wedge B^{di}\pi_+\stackrel{p}{\longrightarrow}\mathbb{S}\wedge B^{\sigma}\pi_+.
\] 
\end{theorem}

\begin{proof}
We define dashed arrows that makes the diagram of $\mathbb{Z}/2$-$\Gamma$-spaces
\[
\xymatrix{
n_+\wedge B^{\sigma}\pi_+\ar[r]^{\mathcal{A}^0}\ar[d]_-{c}
&
{\KR}_n(\mathbb{S}[\pi])\ar[r]^-{c}
&
\coprod_{\underline{a}}\Lambda^{\sigma}\|N^{\sigma}(\widetilde{\langle\underline{a}\rangle}\times\widehat{GL}^{\vee}_{a_1,\dots,a_n}(\mathbb{S}[\pi]))\|
\\
n_+\wedge(\Lambda^\sigma \|N^{\sigma}\pi\|)_+\ar@{-->}[urr]&&
\coprod_{\underline{a}}B^{di}(\widetilde{\langle\underline{a}\rangle}\times\widehat{GL}^{\vee}_{a_1,\dots,a_n}(\mathbb{S}[\pi]))\ar[u]_\simeq^{l}\ar[d]^{\tr^{cy}}
\\
n_+\wedge B^{di}\pi_+\ar[rr]_-{\simeq}\ar@{-->}[urr]\ar[u]^-{\simeq}_{l}
&
&
{\THR}_n(\mathbb{S}[\pi])
}
\]
commute. Here the bottom map is the equivalence of Proposition \ref{THRass}, the map $l$ is the equivalence of Lemma \ref{lemmafreeloop}, and the vertical map $c$ is induced by the composite
\[
B^{\sigma}\pi\longrightarrow \Lambda^\sigma B^{\sigma}\pi\stackrel{\simeq}{\longrightarrow} \Lambda^\sigma \|N^{\sigma}\pi\|
\]
of the inclusion of constant loops and the canonical equivalence of Lemma \ref{lemmathicksub}.
This shows that in the homotopy category the map ${\KR}(\mathbb{S}[\pi])\stackrel{\tr}{\to}\THR(\mathbb{S}[\pi])\simeq\mathbb{S}\wedge B^{di}\pi_+$ equals the composition of $c$ and the inverse of $l$. Since $p=\ev\circ l$, where the evaluation map $\ev\colon \Lambda^\sigma \|N^{\sigma}\pi\|\to \|N^\sigma\pi\|$ splits $c$, we have that $p$ splits $l^{-1}\circ c$, and this will conclude the proof.

The lower dashed map is defined as the adjoint of the map of $\mathbb{Z}/2$-spaces
\[
B^{di}\pi_+\longrightarrow \coprod_{k\geq 0}B^{di}\widehat{GL}^{\vee}_{k}(\mathbb{S}[\pi])
\]
that sends $+$ to zero, and that includes $B^{di}\pi$ in the $k=1$ summand by the dihedral nerve of the map of monoids with anti-involution $\pi\to\widehat{GL}^{\vee}_{1}(\mathbb{S}[\pi])$. The upper dashed map is adjoint to the map
of $\mathbb{Z}/2$-spaces
\[
(\Lambda^\sigma \|N^{\sigma}\pi\|)_+\longrightarrow \coprod_{k\geq 0}\Lambda^\sigma \|N^{\sigma}\widehat{GL}^{\vee}_{k}(\mathbb{S}[\pi])\|,
\]
induced by the same map $\pi\to\widehat{GL}^{\vee}_{1}(\mathbb{S}[\pi])$.
The bottom right triangle commutes since the trace map leaves the $k=1$ summand essentially untouched. The middle rectangle commutes by naturality of the map $l$.
The upper left triangle of the diagram commutes by construction, since both $\mathcal{A}^0$ and the upper dashed map are induced by $\pi\to\widehat{GL}^{\vee}_{1}(\mathbb{S}[\pi])$.
\end{proof}

\begin{corollary} Let $\pi$ be a topological group, which is cofibrant as a $\Z/2$-space. The fixed-points spectrum ${\KH}(\mathbb{S}[\pi])$ splits off a copy of 
\[(\mathbb{S}\wedge B^{\sigma}\pi_+)^{\Z/2}\simeq \mathbb{S}\wedge ((B\pi\times\mathbb{RP}^\infty)\amalg (B^{\sigma}\pi)^{\Z/2})_+.\]
If $\pi$ is discrete, the second term decomposes further as $(B^{\sigma}\pi)^{\Z/2}\cong \coprod_{\{[g]\ |\ g^2=1\}}BZ_\pi\langle g\rangle$ by \ref{funnymap}.
\end{corollary}

\begin{proof}
The splitting follows immediately from Theorem \ref{asssplitKR}. By the Segal-Tom Dieck splitting there is a natural decomposition
\[
(\mathbb{S}\wedge B^{\sigma}\pi_+)^{\Z/2}\simeq \mathbb{S}\wedge ((B^{\sigma}\pi)_{h\Z/2}\amalg (B^{\sigma}\pi)^{\Z/2})_+.
\]
We recall by \ref{funnymap} that there is an equivariant map $B\pi\to B^{\sigma}\pi$ which is a non-equivariant equivalence, where $B\pi$ has the trivial involution. This gives an equivalence 
\[(B^{\sigma}\pi)_{h\Z/2}\stackrel{\simeq}{\longleftarrow}(B\pi)_{h\Z/2}=B\pi\times  \mathbb{RP}^\infty.\qedhere\]
\end{proof}

\begin{corollary}\label{resasssplit}
The restricted assembly maps of the Hermitian $K$-theory and genuine $L$-theory of the spherical group-ring of \ref{defresass}
\[\mathbb{S}\wedge (B\pi\amalg B\pi)_+\longrightarrow {\KH}(\mathbb{S}[\pi])\]
 \[ \mathbb{S}\wedge B\pi_+\longrightarrow \Lg(\mathbb{S}[\pi])\]
are naturally split monomorphism in the homotopy category of spectra.
\end{corollary}

\begin{proof}
We start by proving the claim for $L$-theory. 
By Theorem \ref{asssplitKR} the second map in the composite
\[
\xymatrix{\mathbb{S}\wedge B\pi_+\stackrel{\lambda}{\longrightarrow}\mathbb{S}\wedge (B^{\sigma}\pi)^{\mathbb{Z}/2}_+\ar[r]^-{\simeq}
&
\Phi^{\mathbb{Z}/2}(\mathbb{S}\wedge B^{\sigma}\pi_+)\ar[r]& \Lg(\mathbb{S}[\pi])\ar@/_1.5pc/[l]
}
\]
splits. Thus it is sufficient to show that the first map splits, and a retraction is provided by the inclusion of fixed points $\iota\colon(B^{\sigma}\pi^{\mathbb{Z}/2})\to B\pi$, by Lemma \ref{funnymap}. Similarly, the restricted assembly of the Hermitian $K$-theory is the composite
\[
\xymatrix{\mathbb{S}\wedge (B\pi\amalg B\pi)_+\to\mathbb{S}\wedge ((B^{\sigma}\pi)^{\mathbb{Z}/2}\amalg (B^{\sigma}\pi)_{h\mathbb{Z}/2})_+\ar[r]^-{\simeq}
&
(\mathbb{S}\wedge B^{\sigma}\pi_+)^{\mathbb{Z}/2}\ar[r]& {\KH}(\mathbb{S}[\pi])\ar@/_1.5pc/[l],
}
\]
and it is sufficient to show that the first map is a split monomorphism. The first summand is again split by the inclusion of fixed points. The second summand is split by the projection map
\[
(B^{\sigma}\pi)_{h\mathbb{Z}/2}\stackrel{\simeq}{\longleftarrow} (B\pi)_{h\mathbb{Z}/2}\cong B\pi\times \mathbb{RP}^\infty\longrightarrow B\pi.\qedhere
\]
\end{proof}

We remark that the same argument of the proof of Corollary \ref{resasssplit} shows that the restricted assembly map of any ring spectrum with anti-involution which is rationally equivalent to the sphere spectrum splits rationally. We will be particularly interested in the case where $A=H\mathbb{A}[\tfrac{1}{2}]$ is the Eilenberg MacLane spectrum of the Burnside Mackey functor $\mathbb{A}[\tfrac{1}{2}]$ with $2$ inverted. In particular for $L$-theory we obtain the following.

\begin{corollary}\label{resasssplitLBurn}
Let $A$ be a flat ring spectrum with anti-involution, such that the unit map $\mathbb{S}\to A$ is a rational equivalence of underlying $\mathbb{Z}/2$-spectra.
Then the rationalized restricted assembly map in $L$-theory
 \[\mathcal{A}^0\colon H\mathbb{Q}\wedge B\pi_+\simeq (\mathbb{S}\wedge B\pi_+)\otimes\mathbb{Q}\longrightarrow \Lg(A[\pi])\otimes\mathbb{Q}\]
is naturally split by the map 
\[ T\colon \Lg(A[\pi])\stackrel{\tr}{\longrightarrow} \Phi^{\mathbb{Z}/2}{\THR}(A[\pi])\stackrel{\simeq_\mathbb{Q}}{\longleftarrow}\mathbb{S}\wedge (B^{di}\pi)^{\mathbb{Z}/2}_+\stackrel{p}{\longrightarrow} \mathbb{S}\wedge (B^{\sigma}\pi)^{\mathbb{Z}/2}_+\stackrel{\ref{funnymap}}{\longrightarrow} \mathbb{S}\wedge B\pi_+ .\]
\end{corollary}

\section{Application to the Novikov conjecture}\label{secmain}\label{sectwo}

Let $\pi$ be a discrete group and $\Lq(\mathbb{Z}[\pi])$ the quadratic $L$-theory spectrum of the corresponding integral group-ring. The assembly map of quadratic $L$-theory is a map of spectra
\[
\mathcal{A}_{\mathbb{Z}[\pi]}\colon\Lq(\mathbb{Z})\wedge B\pi_+\longrightarrow \Lq(\mathbb{Z}[\pi]).
\]
The Novikov conjecture for the discrete group $\pi$ is equivalent to the injectivity on rational homotopy groups of the map $\mathcal{A}_{\mathbb{Z}[\pi]}$. Rationally, $\Lq(\mathbb{Z})$ is a Laurent polynomial algebra on one generator $\beta$ of degree $4$. Thus on rational homotopy groups the assembly map above takes the form
\[
\mathcal{A}_{\mathbb{Z}[\pi]}\colon \mathbb{Q}[\beta,\beta^{-1}]\otimes H_\ast(B\pi;\mathbb{Q})\longrightarrow \Lqsub{\ast}(\mathbb{Z}[\pi])\otimes \mathbb{Q}.
\]

\begin{remark}\label{remconnective}
Since the assembly map is a map of $\mathbb{Q}[\beta,\beta^{-1}]$-modules, it is sufficient to show that $\mathcal{A}_{\mathbb{Z}[\pi]}$ does not annihilate the polynomials with non-zero constant term
\[
\underline{x}=1\otimes x_n+\beta\otimes x_{n-4}+\dots+\beta^{k}\otimes x_{n-4k}\ \ \ \in \ \ (\mathbb{Q}[\beta]\otimes H_\ast(B\pi;\mathbb{Q}))_n
\]
where $k\geq 0$, and $x_n\neq 0$, for every $n\geq 0$.
Indeed, any degree $j$ element $\underline{y}$ of $\mathbb{Q}[\beta,\beta^{-1}]\otimes H_\ast(B\pi;\mathbb{Q})$ can be written as $\underline{y}=\beta^{-l}\underline{x}$ where $\underline{x}$ is of the form above and $l$ is the lowest power with non-zero coefficient of  $\underline{y}$. Then $\mathcal{A}_{\mathbb{Z}[\pi]}(\underline{y})=\beta^{-l}\mathcal{A}_{\mathbb{Z}[\pi]}(\underline{x})$ is non-zero if and only if $\mathcal{A}_{\mathbb{Z}[\pi]}(\underline{x})$ is non-zero. In particular, we can restrict our attention to the connective cover of the assembly map.
\end{remark}

We let $\mathbb{A}_{\frac{1}{2}}$ be the Burnside Tambara functor with $2$ inverted, and $H\mathbb{A}_{\frac{1}{2}}$ a cofibrant strictly commutative orthogonal $\mathbb{Z}/2$-equivariant ring spectrum model for its Eilenberg-MacLane spectrum (for the existence, see e.g. \cite[Thm.5.1]{Ullman}). We let $d\colon\mathbb{A}_{\frac 12}[\pi]\to \underline{\mathbb{Z}}_{\frac 12}[\pi]$ be the rank map, as defined in Example \ref{dimonKH}.

\begin{theorem}\label{main}
For any discrete group $\pi$, there is a lift $\overline{\mathcal{A}}_{\mathbb{Z}[\pi]}$ of the assembly map for the integral group-ring
\[
\xymatrix@C=75pt{
&& \Lgsub{\ast}(\mathbb{A}_{\frac 12}[\pi])\otimes\mathbb{Q}\ar[d]^d
\\
\mathbb{Q}[\beta]\otimes H_\ast (B\pi;\mathbb{Q})\ar[r]_-{\mathcal{A}_{\mathbb{Z}[\pi]}}\ar[urr]^-{\overline{\mathcal{A}}_{\mathbb{Z}[\pi]}}&&\llap{$ \Lqsub{\ast\geq 0}(\mathbb{Z}[\pi])\otimes\mathbb{Q}\cong$}   \Lgsub{\ast}(\mathbb{Z}_{\frac 12}[\pi])\otimes\mathbb{Q}
}
\]
which does not annihilate the polynomials with non-zero constant term. Thus the Novikov conjecture holds for $\pi$ if and only if the image of $\overline{\mathcal{A}}_{\mathbb{Z}[\pi]}$ intersects the kernel of $d$ trivially.
\end{theorem}

The way the lift $\overline{\mathcal{A}}_{\mathbb{Z}[\pi]}$ is constructed allows us to further reduce the Novikov conjecture to a statement about the algebraic structure of $\Lg_\ast(\mathbb{A}_{\frac 12}[\pi])\otimes\mathbb{Q}$, that does not involve the map $d$. After inverting $2$ the morphism of Hermitian Mackey functors $d\colon \mathbb{A}_{\frac 12}\to \underline{\mathbb{Z}}_{\frac 12}$ splits (although not as a Tambara functor). This induces a map
\[
s_\pi\colon \Lg(\mathbb{Z}_{\frac 12}[\pi]) \longrightarrow\Lg(\mathbb{A}_{\frac 12}[\pi])
\]
which is a section for $d$ (see Lemma \ref{dsurj}). For the trivial group $\pi=1$ this provides an additive inclusion $s=s_1\colon\mathbb{Q}[\beta]\to \Lgsub{\ast}(\mathbb{A}_{\frac 12}[\pi])\otimes\mathbb{Q}$. In particular $s(1)$ defines an element in the ring $\Lgsub{0}(\mathbb{A}_{\frac 12})$, which acts on $\Lgsub{\ast}(\mathbb{A}_{\frac 12}[\pi])$ by \ref{multiplication}.

\begin{corollary}\label{maincor}
Every element $\underline{x}\in\mathbb{Q}[\beta]\otimes H_\ast (B\pi;\mathbb{Q})$ satisfies the identity
\[
s_\pi\mathcal{A}_{\mathbb{Z}[\pi]}(\underline{x})=s(1)\cdot \overline{\mathcal{A}}_{\mathbb{Z}[\pi]}(\underline{x}) \ \ \ \ \in \ \ \Lgsub{\ast}(\mathbb{A}_{\frac 12}[\pi])\otimes\mathbb{Q},
\]
where $s_\pi$ is injective and $\overline{\mathcal{A}}_{\mathbb{Z}[\pi]}(\underline{x})$ is non-zero when $\underline{x}$ has non-zero constant term.
It follows that the Novikov conjecture holds for $\pi$ if and only if multiplication by $s(1)$
\[
s(1)\colon \Lgsub{\ast}(\mathbb{A}_{\frac 12}[\pi])\longrightarrow \Lgsub{\ast}(\mathbb{A}_{\frac 12}[\pi])
\]
is injective on the image of $\overline{\mathcal{A}}_{\mathbb{Z}[\pi]}$.
\end{corollary}

\begin{proof}[Proof of \ref{main}] We first note that the isomorphism $\Lgsub{\ast}(\mathbb{Z}_{\frac{1}{2}}[\pi])\otimes\mathbb{Q}\cong \Lqsub{\ast\geq 0}(\mathbb{Z}[\pi])\otimes\mathbb{Q}$ is a consequence of Proposition \ref{Lgeom} and the fact that the map $\mathbb{Z}[\pi]\to \mathbb{Z}_{\frac{1}{2}}[\pi]$ induces an isomorphism $\Lqsub{\ast}(\mathbb{Z}[\pi])[\frac{1}{2}]\cong \Lqsub{\ast}(\mathbb{Z}_{\frac{1}{2}}[\pi])[\frac{1}{2}]$ on the quadratic $L$-groups. The latter is an immediate consequence of \cite[Prop.~3.9]{BF} (see also \cite[Thm.3.2.6]{jl}). It follows from \ref{compfulass} that the assemblies agree under this isomorphism.

The additive section $s\colon \mathbb{Q}[\beta]\to\Lgsub{\ast}(\mathbb{A}_{\frac 12})\otimes\mathbb{Q}$ of Lemma \ref{dsurj} prescribes a lift $s(\beta)$ of the polynomial generator $\beta$. Since $\Lgsub{\ast}(\mathbb{A}_{\frac 12})$ is a ring, $s(\beta)$ defines a multiplicative section $u\colon \mathbb{Q}[\beta]\to\Lgsub{\ast}(\mathbb{A}_{\frac 12})\otimes\mathbb{Q}$. We define $\overline{\mathcal{A}}_{\mathbb{Z}[\pi]}$ to be the composite
\[
\overline{\mathcal{A}}_{\mathbb{Z}[\pi]}\colon \mathbb{Q}[\beta]\otimes H_\ast(B\pi;\mathbb{Q})\xrightarrow{u\otimes \id} \Lgsub{\ast}(\mathbb{A}_{\frac 12})\otimes H_\ast(B\pi;\mathbb{Q})\xrightarrow{\mathcal{A}_{\mathbb{A}[\pi]}} \Lgsub{\ast}(\mathbb{A}_{\frac 12}[\pi])\otimes\mathbb{Q}
\]
where $\mathcal{A}_{\mathbb{A}[\pi]}$ is the assembly of \ref{defassKHL}. Since $d\mathcal{A}_{\mathbb{A}[\pi]}=\mathcal{A}_{\mathbb{Z}[\pi]}(d\otimes \id)$ and $u$ is a section for $d$, it follows that $d\overline{\mathcal{A}}_{\mathbb{Z}[\pi]}=\mathcal{A}_{\mathbb{Z}[\pi]}$. 

Now let us show that $\overline{\mathcal{A}}_{\mathbb{Z}[\pi]}$ does not annihilate the polynomials with non-zero constant term, that is that for every
\[
\underline{x}=1\otimes x_n+\beta\otimes x_{n-4}+\dots+\beta^{k}\otimes x_{n-4k}\ \ \ \in \ \ (\mathbb{Q}[\beta]\otimes H_\ast(B\pi;\mathbb{Q}))_n,
\]
with $k\geq 0$, and $x_n\neq 0$, we have that $\overline{\mathcal{A}}_{\mathbb{Z}[\pi]}(\underline{x})$ is non-zero. We let 
\[
T\colon \Lgsub{\ast}(\mathbb{A}_{\frac 12}[\pi])\otimes\mathbb{Q}\longrightarrow H_\ast(B\pi;\mathbb{Q})
\]
be the map of Corollary \ref{resasssplitLBurn}, which defines a retraction for the restricted assembly map $\mathcal{A}_{\mathbb{A}[\pi]}^0$. 
In order to show that $\overline{\mathcal{A}}_{\mathbb{Z}[\pi]}(\underline{x})$ is non-zero it is sufficient to show that $T\overline{\mathcal{A}}_{\mathbb{Z}[\pi]}(\underline{x})$ is non-zero in $H_\ast(B\pi;\mathbb{Q})$. We write
\[
T\overline{\mathcal{A}}_{\mathbb{Z}[\pi]}(\underline{x})=T\overline{\mathcal{A}}_{\mathbb{Z}[\pi]}(1\otimes x_n)+T\overline{\mathcal{A}}_{\mathbb{Z}[\pi]}(\beta\otimes x_{n-4}+\dots+\beta^{k}\otimes x_{n-4k}).
\]
We remark that since the section $u$ is multiplicative, the restriction of $\overline{\mathcal{A}}_{\mathbb{Z}[\pi]}$ to the summand $(\mathbb{Q}\cdot 1)\otimes H_\ast(B\pi;\mathbb{Q})$ agrees with the restricted assembly map $\mathcal{A}_{\mathbb{A}[\pi]}^0$.
Since $T$ splits $\mathcal{A}^{0}_{\mathbb{A}[\pi]}$ by Corollary \ref{resasssplitLBurn}, it follows that 
\[T\overline{\mathcal{A}}_{\mathbb{Z}[\pi]}(1\otimes x_n)=T\mathcal{A}_{\mathbb{A}[\pi]}^0(x_n)=x_n\neq 0.\]
 It is therefore sufficient to show that $T\overline{\mathcal{A}}_{\mathbb{Z}[\pi]}(\beta\otimes x_{n-4}+\dots+\beta^{k}\otimes x_{n-4k})$ is zero.
We prove this by invoking the naturality of $T$ and $\overline{\mathcal{A}}_{\mathbb{Z}[\pi]}$ with respect to group homomorphisms.
By the Kan-Thurston theorem there is a group $\pi_{\langle n-1\rangle}$ and a map $B\pi_{\langle n-1\rangle}\to (B\pi)^{(n-1)}$ to the  $(n-1)$-skeleton of $B\pi$ which is a homology isomorphism. Taking the first homotopy group of the composite $B\pi_{\langle n-1\rangle}\to (B\pi)^{(n-1)}\hookrightarrow B\pi$ gives a group homomorphism $\lambda\colon \pi_{\langle n-1\rangle}\to \pi$. In order to emphasize the naturality of our transformations in the group $\pi$ we add a superscript to our notations.
By naturality, there is a commutative diagram
\[
\xymatrix@C=20pt{
(\mathbb{Q}[\beta]\otimes H_\ast(B\pi_{\langle n-1\rangle};\mathbb{Q}))_n\ar[rr]^-{\overline{\mathcal{A}}_{\mathbb{Z}[\pi_{\langle n-1\rangle}]}}\ar[d]^{}_{\id\otimes \lambda} & &\Lgsub{n}(\mathbb{Z}[\pi_{\langle n-1\rangle}])\ar[d]_{\lambda}\ar[rrr]^(0.45){T^{\pi_{\langle n-1\rangle}}} & & &  H_n(B\pi_{\langle n-1\rangle};\mathbb{Q}) = 0\ar[d]_{\lambda}\\
(\mathbb{Q}[\beta]\otimes H_\ast(B\pi;\mathbb{Q}))_n\ar[rr]^{\overline{\mathcal{A}}_{\mathbb{Z}[\pi]}} & &\Lgsub{n}(\mathbb{Z}[\pi])\ar[rrr]^{T^\pi} & & & H_n(B\pi;\mathbb{Q})\rlap{\ .}
}
\]
Since $\lambda\colon H_j((B\pi)^{(n-1)};\mathbb{Q})\cong H_j(B\pi_{\langle n-1\rangle};\mathbb{Q})\to H_j(B\pi;\mathbb{Q})$ is surjective for $j<n$, we can write
\[
\beta\otimes x_{n-4}+\dots+\beta^{k}\otimes x_{n-4k}=\beta\otimes \lambda(y_{n-4})+\dots+\beta^{k}\otimes \lambda(y_{n-4k})
\]
for some coefficients $y_j\in H_j(B\pi_{\langle n-1\rangle};\mathbb{Q})$. By the commutativity of the diagram above we see that
\[
T^\pi\overline{\mathcal{A}}_{\mathbb{Z}[\pi]}(\beta\otimes x_{n-4}+\dots+\beta^{k}\otimes x_{n-4k})=\lambda \underbrace{T^{\pi_{\langle n-1\rangle}}\overline{\mathcal{A}}_{\mathbb{Z}[\pi_{\langle n-1\rangle}]}(\beta\otimes y_{n-4}+\dots+\beta^{k}\otimes y_{n-4k})}_{0}
\] 
must vanish.

Now suppose that the kernel of $d$ intersects the image of $\overline{\mathcal{A}}_{\mathbb{Z}[\pi]}$ trivially, and let $\underline{x}\in \mathbb{Q}[\beta]\otimes H_\ast(B\pi;\mathbb{Q})$ be a polynomial with non-zero constant term. By the previous argument $\overline{\mathcal{A}}_{\mathbb{Z}[\pi]}(\underline{x})$ is non-zero, and therefore it cannot belong to the kernel of $d$. That is $d\overline{\mathcal{A}}_{\mathbb{Z}[\pi]}(\underline{x})=\mathcal{A}_{\mathbb{Z}[\pi]}(\underline{x})$ is non-zero. By Remark \ref{remconnective} the Novikov conjecture holds for $\pi$.
Conversely, if $\mathcal{A}_{\mathbb{Z}[\pi]}$ is injective $d$ must be injective on the image of $\overline{\mathcal{A}}_{\mathbb{Z}[\pi]}$.
\end{proof}

\begin{lemma}\label{dsurj}
The rank map $d\colon \Lg(\mathbb{A}_{\frac{1}{2}}[\pi])\to \Lg(\mathbb{Z}_{\frac{1}{2}}[\pi])$ admits a section $s_\pi$ in the homotopy category, which is natural in $\pi$ with respect to group homomorphisms. When $\pi=1$ is the trivial group, this section $s\colon \Lg(\mathbb{Z}_{\frac{1}{2}})\to \Lg(\mathbb{A}_{\frac{1}{2}})$ is multiplicative, but not unital, on rational homotopy groups.
\end{lemma}

\begin{proof}
The map of Hermitian Mackey functors $d\colon\mathbb{A}_{\frac{1}{2}}\to \underline{\mathbb{Z}}_{\frac{1}{2}}$ splits. A section $\frac{T}{2}\colon \underline{\mathbb{Z}}_{\frac{1}{2}}\to \mathbb{A}_{\frac{1}{2}}$ is defined by the identity on the underlying ring, and by the map
\[
(0,\frac{1}{2})\colon \mathbb{Z}_{\frac{1}{2}}\longrightarrow  \mathbb{Z}_{\frac{1}{2}}\oplus  \mathbb{Z}_{\frac{1}{2}}
\]
on fixed points. We remark that $\frac{T}{2}$ is not a map of Tambara functors, since $(0,\frac{1}{2})$ is not unital with respect to the ring structure of the Burnside ring $ \mathbb{Z}_{\frac{1}{2}}\oplus  \mathbb{Z}_{\frac{1}{2}}$.
However, it is a morphism of Hermitian Mackey functors, and it extends to a morphism of Hermitian Mackey functors $ \underline{\mathbb{Z}}_{\frac{1}{2}}[\pi]\to \mathbb{A}_{\frac{1}{2}}[\pi]$.
This induces a section for $d$ in Hermitian $K$-theory
\[
\overline{s}_\pi\colon {\KH}(\mathbb{Z}_{\frac{1}{2}}[\pi])\longrightarrow {\KH}(\mathbb{A}_{\frac{1}{2}}[\pi]).
\]
Since $\frac{T}{2}$ is not a map of Tambara functors it cannot be realized as a map of $\mathbb{Z}/2$-equivariant commutative ring spectra $H\mathbb{Z}_{\frac{1}{2}}\to H\mathbb{A}_{\frac{1}{2}}$, and thus $\overline{s}_\pi$ does not a priori come from a map of real $K$-theory spectra. It is therefore not immediately clear if $\overline{s}_\pi$ induces a map on geometric fixed-points spectra.

The isotropy separation sequences for ${\KR}(\mathbb{A}_{\frac{1}{2}}[\pi])$ and ${\KR}(\mathbb{Z}_{\frac{1}{2}}[\pi])$ are compared by a commutative diagram
\[
\xymatrix{
\K(\mathbb{Z}_{\frac{1}{2}}[\pi])_{h\Z/2}\ar[r]^-{H}\ar@{=}[d]
&
{\KH}(\mathbb{Z}_{\frac{1}{2}}[\pi])\ar[d]^-{\overline{s}_\pi}
\ar[r]^-{\phi}
&
\Lg(\mathbb{Z}_{\frac{1}{2}}[\pi])\ar@{-->}[d]^-{s_\pi}
\\
\K(\mathbb{A}_{\frac{1}{2}}[\pi])_{h\Z/2}\ar[r]^-{H}
&
{\KH}(\mathbb{A}_{\frac{1}{2}}[\pi])
\ar[r]^-{\phi}
&
\Lg(\mathbb{A}_{\frac{1}{2}}[\pi])
}
\] 
where the rows are cofiber sequences, and $H$ is the map induced by the hyperbolic functor from the category of finitely generated free $\Z_{\frac{1}{2}}[\pi]$-modules to the category of Hermitian forms, respectively over $\Z_{\frac{1}{2}}[\pi]$ and over $\mathbb{A}_{\frac{1}{2}}[\pi]$. The left-hand square commutes at the level of categories of modules and forms, and therefore the dotted arrow $s_\pi$ exists.

When $\pi$ is trivial, the maps $\phi$ are maps of rings, and $\overline{s}$ is multiplicative by \ref{remsmult}. On rational homotopy groups (or in fact away from $2$) the isotropy separation sequence splits giving short exact sequences
\[
\xymatrix{
\pi_n\K(\mathbb{Z}_{\frac{1}{2}})_{h\Z/2}\otimes\mathbb{Q}\ar@{>->}[r]^-{H}\ar@{=}[d]
&
{\KH}_n(\mathbb{Z}_{\frac{1}{2}})\otimes\mathbb{Q}\ar[d]^-{\overline{s}_n}
\ar@{->>}[r]^-{\phi_n}
&
\Lgsub{n}(\mathbb{Z}_{\frac{1}{2}})\otimes\mathbb{Q}\ar[d]^-{s_n}
\\
\pi_n\K(\mathbb{A}_{\frac{1}{2}})_{h\Z/2}\otimes\mathbb{Q}\ar@{>->}[r]^-{H}
&
{\KH}_n(\mathbb{A}_{\frac{1}{2}})\otimes\mathbb{Q}
\ar@{->>}[r]^-{\phi_n}
&
\Lgsub{n}(\mathbb{A}_{\frac{1}{2}})\otimes\mathbb{Q}
}
\] 
for every $n\geq 0$. Then $s_n$ is given by $s_n(x)=\phi_n\overline{s}_n(y)$ for some $y\in {\KH}_n(\mathbb{Z}_{\frac{1}{2}})\otimes\mathbb{Q}$ such that $\phi_n(y)=x$, and it does not depend of such choice. Since $\phi$ is multiplicative, if $\phi_n(y)=x$ and $\phi_m(y')=x'$ we have that $\phi_{n+m}(yy')=xx'$. Thus 
\[
s_{n+m}(xx')=\phi_{n+m}\overline{s}_{n+m}(yy')=(\phi_n\overline{s}_n(y))(\phi_m\overline{s}_m(y'))=s_n(x)s_m(x'),
\]
proving that $s$ is multiplicative on rational homotopy groups.
\end{proof}

\begin{proof}[Proof of \ref{maincor}]
Let $s_\pi$ denote the natural section of Lemma \ref{dsurj}. This fits into a diagram
\[
\xymatrix@C=75pt{
\Lgsub{\ast}(\mathbb{A}_{\frac 12})\otimes H_\ast (B\pi;\mathbb{Q})
 \ar[rr]^-{\mathcal{A}_{\mathbb{A}[\pi]}}
&&
\Lgsub{\ast}(\mathbb{A}_{\frac 12}[\pi])\otimes\mathbb{Q}
\\
\mathbb{Q}[\beta]\otimes H_\ast (B\pi;\mathbb{Q})\ar[r]_-{\mathcal{A}_{\mathbb{Z}[\pi]}}\ar@<-1.5ex>[u]_-{s\otimes \id}\ar@<1.5ex>[u]^-{u\otimes\id}
&&
\llap{$\Lqsub{\ast\geq 0}(\mathbb{Z}[\pi])\otimes\mathbb{Q}\cong$} \Lgsub{\ast}(\mathbb{Z}_{\frac 12}[\pi])\otimes\mathbb{Q}\ar[u]_-{s_\pi}
}
\]
where the right rectangle commutes. We observe that $s$ does not commute with the unit maps of ${\KR}_\ast(\mathbb{A}_{\frac 12})$ and ${\KR}_\ast(\mathbb{Z}_{\frac 12})$, and therefore that $s_\pi$ does not commute with the restricted assembly maps. The map $s$ is however multiplicative by Lemma \ref{dsurj}, and since $u$ sends by definition $\beta$ to $s(\beta)$ we have that $s=s(1)\cdot u$. Since $\mathcal{A}_{\mathbb{A}[\pi]}$ is a map of $\Lgsub{\ast}(\mathbb{A}_{\frac 12})$-modules, we have that for every non-zero $\underline{x}\in \mathbb{Q}[\beta]\otimes H_\ast (B\pi;\mathbb{Q})$
\begin{align*}
s_\pi\mathcal{A}_{\mathbb{Z}[\pi]}(\underline{x})&=\mathcal{A}_{\mathbb{A}[\pi]}(s\otimes\id)(\underline{x})=\mathcal{A}_{\mathbb{A}[\pi]}((s(1)\cdot u)\otimes \id)(\underline{x})=\\
&=s(1)\cdot (\mathcal{A}_{\mathbb{A}[\pi]}(u\otimes\id)(\underline{x}))=s(1)\cdot \overline{\mathcal{A}}_{\mathbb{Z}[\pi]}(\underline{x}).
\end{align*}
If the Novikov conjecture holds, the left-hand term must be non-zero since $s_\pi$ is injective. Thus $s(1)$ must act injectively on the image of $\overline{\mathcal{A}}_{\mathbb{Z}[\pi]}$. Conversely, if $s(1)$ acts injectively on the image of $\overline{\mathcal{A}}_{\mathbb{Z}[\pi]}$ the right-hand term must be non-zero when $\underline{x}$ is a polynomial with non-zero constant term, since $\overline{\mathcal{A}}_{\mathbb{Z}[\pi]}(\underline{x})$ is non-zero by Theorem \ref{main}. It follows that $\mathcal{A}_{\mathbb{Z}[\pi]}(\underline{x})$ is non-zero, and this implies the Novikov conjecture by Remark \ref{remconnective}.
\end{proof}

\bibliographystyle{amsalpha}
\bibliography{bib}

\end{document}